\documentclass[final,leqno,onefignum,onetabnum]{siamltex}

\usepackage{amsmath,amssymb,amsfonts,color,epsfig}
\usepackage{algorithmic}
\usepackage{algorithm,booktabs}
\usepackage[mathscr]{eucal}

\newtheorem{remark}[theorem]{Remark}
\newtheorem{assumption}{Assumption}

\renewcommand{\theequation}{\arabic{section}.\arabic{equation}}

\renewcommand{\thefigure}{\arabic{section}.\arabic{figure}}
\renewcommand{\thetable}{\arabic{section}.\arabic{figure}}

\newcommand{\UAD}{{\mathbb U_\mathsf{ad}}}
\newcommand{\Uad}{{U_\mathsf{ad}}}
\newcommand{\te}{{t_\mathsf e}}
\newcommand{\R}{{\mathbb R}}
\newcommand{\wnp}{p}

\author{A. Alla\footnotemark[1]
\and M. Falcone\footnotemark[2]
\and S. Volkwein\footnotemark[3]}

\begin{document}

\renewcommand{\thefootnote}{\fnsymbol{footnote}}

\title{Error analysis for POD Approximations of\\ infinite horizon problems via the\\ Dynamic Programming approach}

\maketitle

\footnotetext[1]{Department of Mathematics, Universit\"at Hamburg, Bundesstra{ss}e 55, D-20146 Hamburg, Germany, \texttt{alessandro.alla@uni-hamburg.de}(corresponding author)}
\footnotetext[2]{Dipartimento di Matematica,
SAPIENZA - Universit\`a di Roma, P.le Aldo Moro 2, I-00185 Rome, Italy, \texttt{falcone@mat.uniroma1.it} }
\footnotetext[3]{Department of Mathematics and Statistics, Universit\"at Konstanz, Universit\"atsstra{\ss}e 10, D-78457 Konstanz, Germany, \texttt{Stefan.Volkwein@uni-konstanz.de}}

\begin{abstract}
In this paper infinite horizon optimal control problems for nonlinear high-dimensional dynamical systems are studied. Nonlinear feedback laws can be computed via the value function characterized as the unique viscosity solution to the corresponding Hamilton-Jacobi-Bellman (HJB) equation which stems from the dynamic programming approach. However, the bottleneck is mainly due to the {\em curse of dimensionality} and HJB equations are only solvable in a relatively small dimension. Therefore, a reduced-order model is derived for the dynamical system and for this purpose the method of proper orthogonal decomposition (POD) is used. The resulting errors in the HJB equations are estimated by an a-priori error analysis, which suggests a new sampling strategy for the POD method. Numerical experiments illustrates the theoretical findings.
\end{abstract}



%%%%%%%%%%%%%%%%%%%%%%%%%%%%%%%%%%%%%%%%%%%%%%%%%%%%%
\section{Introduction}
\label{Section1}
\setcounter{section}{1}
\setcounter{equation}{0}
\setcounter{theorem}{0}
\setcounter{algorithm}{0}
\renewcommand{\theequation}{\arabic{section}.\arabic{equation}}
%%%%%%%%%%%%%%%%%%%%%%%%%%%%%%%%%%%%%%%%%%%%%%%%%%%%%

The  Dynamic Programming approach to the solution of optimal control problems driven by dynamical systems in $\mathbb{R}^n$ offers a nice framework for the approximation of feedback laws and optimal trajectories. It suffers from the bottleneck of the computation of the value function since this requires the approximation of a nonlinear partial differential equation in dimension $n$. This is a very challenging problem in high dimension due to the huge amount of memory allocations necessary to work on a grid and to the low regularity properties of the value function (which is typically only Lipschitz continuous even for regular dynamics and running costs). Despite the number of theoretical results established for many classical control problems via the dynamic programming approach (see e.g.  the monographies by Bardi and Capuzzo-Dolcetta \cite{BC97} on deterministic control problems and by Fleming and Soner \cite{FS93} on stochastic control problems) this has always been the main obstacle to apply this nowadays rather complete theory to real applications. The "curse of dimensionality"  has been mitigated via domain decomposition techniques and the development of rather efficient numerical methods but it is still a big obstacle. Although a detailed description of these contributions goes beyond the scopes of this paper, we want to mention  \cite{FLS94} for a domain decomposition method with overlapping between the subdomains and \cite{CFLS94} for similar results without overlapping. It is important to note that  in these papers the method  is applied to subdomains with a rather simple geometry (see the book by Quarteroni and Valli \cite{QV99} for a general introduction to this technique) in order to apply transmission conditions at the boundaries. More recently another way to decompose the problem  has been proposed by Krener and Navasca \cite{NK07} who have used a patchy decomposition based on Al'brekht method. Later in the paper \cite{CCFP12} the patchy idea has been implemented  taking into account an approximation of the underlying optimal dynamics  to obtain subdomains which are almost invariant with respect to the optimal dynamics, clearly in this case the geometry of the subdomains can be rather complex but the transmission conditions at the internal boundaries can be eliminated saving on the overall complexity of the algorithm. In general, domain decomposition methods  reduce a huge problem into subproblems of manageable size and allows to mitigate the storage limitation distributing the computation over several processors. However, the approximation schemes used in every subdomain are rather standard. Another improvement can be obtained  using efficient acceleration methods  for the computation of the value function in every subdomain. To this end one can use Fast Marching methods \cite{S99,SV03} and Fast Sweeping methods \cite{TCOZ04} for specific classes of Hamilton-Jacobi equations. In the framework of optimal control problems  an efficient acceleration technique based on the coupling between value and policy iterations has been recently proposed and studied by Alla, Falcone and Kalise in \cite{AFK13, AFK15}. Finally, we should mention that the interested reader can find in \cite{FFbook} a number of successful applications  to optimal control problems and games in rather low dimension.

In parallel to these results several model reduction techniques have been developed to deal with high dimensional dynamics in a rather economic way. These techniques are really necessary when dealing with optimal control problems governed by partial differential equations. Despite the vast literature concerning the analysis and numerical approximation of optimal control problems for PDEs, the amount of works devoted to the synthesis of feedback controllers is rather small. In this direction, the application of the dynamic programming principle (DPP) is a powerful technique which has been applied mainly to linear dynamics, quadratic cost functions and unbounded control space,  the so-called linear quadratic regulator (LQR) control problem. For this problem an explicit feedback controller can be computed by means of the solution of an algebraic Riccati equation. However if the underlying structural assumptions are removed, the feedback control has to be obtained via the approximation of  a Hamilton-Jacobi-Bellman equation defined over the state space of the system dynamics.  As we mentioned, this is a major bottleneck for the application of DPP-based techniques in the optimal control of PDEs, as the natural approach for this class of control problems is to consider a semi-discretization (in space) via finite elements or finite differences of the abstract governing equations, leading to an inherently high-dimensional state space.  However, in the last years several steps have been made to obtain reduced-order models for complicated dynamics and by the application of these techniques it is now possible to have a reasonable approximation of large-scale dynamics using a rather small number of basis functions. This can open the way to the DPP approach in high-dimensional systems. 

Reduced-order models are used in PDE-constrained optimization in various ways; see, e.g., \cite{GV13,HV05,SV10} for a survey. However, the  main stream for the optimal control of PDEs is still related to open-loop controls based on the Pontryagin Maximum Principle (an extensive presentation of this classical approach can be found in the monograph \cite{HPUU09,Tro10}). Let us refer to  \cite{AH14,AK01,KVX04,KX05,LV06}, where it has been observed that models of reduced order can play an important and very useful role in the implementation of feedback laws.
More recently, the Proper Orthogonal Decomposition (POD) has been proposed for PDE control problems in order to reduce the dynamics to a small number of state variable via a careful selection of the snapshots. This technique, coupled with the Dynamic Programming approach, has been developed mainly for linear equations starting from the heat equation where one can take advantage of the regularity of the solutions to reduce the dimension  \cite{AK01} and then attacking more difficult problems as the advection-diffusion equation \cite{AF12, AF13b, AH15, KK14}, Burgers's equation \cite{KVX04,KX05} and Navier-Stokes \cite{AH14}. 

The aim of this paper is to study the interplay between reduced-order dynamics, the associated dynamic programming equation, the resulting feedback controller and its performance over the high-dimensional system. In our analysis we will derive some a-priori error estimates which take into account the time and space discretization parameter $\Delta t$ and $\Delta x$ as well as the dimension $\ell$ of the POD basis functions used for the reduced model. 

The paper is organized as follows: in Section~\ref{Section2} we recall some basic facts about the approximation of the infinite horizon problem via the dynamic programming approach. Section~\ref{Section3} is devoted to present in short the POD technique and the basic ideas behind the construction of the reduced model.  In Section~\ref{Section4} we present the main results and our a-priori estimates for the numerical approximation of the reduced model. These a-priori estimates have been also used in the construction of the algorithm which is described in detail in Section~\ref{Section5}. Some numerical tests are presented and analyzed in Section~\ref{Section6} and finally we draw some conclusions in Section~\ref{Section7}. 

%%%%%%%%%%%%%%%%%%%%%%%%%%%%%%%%%%%%%%%%%%%%%%%%%%%%%
\section{Optimal control problem}
\label{Section2}
\setcounter{section}{2}
\setcounter{equation}{0}
\setcounter{theorem}{0}
\setcounter{algorithm}{0}
\renewcommand{\theequation}{\arabic{section}.\arabic{equation}}
%%%%%%%%%%%%%%%%%%%%%%%%%%%%%%%%%%%%%%%%%%%%%%%%%%%%%

In this section we will recall the Dynamic Programming approach and its numerical approximation for the solution of infinite horizon control problem.

%%%%%%%%%%%%%%%%%%%%%%%%%%%%%%%%%%%%%%%%%%%%%%%%%%%%%
\subsection{The infinite horizon problem}
\label{Section2.1}
%%%%%%%%%%%%%%%%%%%%%%%%%%%%%%%%%%%%%%%%%%%%%%%%%%%%%

For given nonlinear mapping $f:\mathbb R^n\times\mathbb R^m\to\mathbb R^n$ and initial condition $y_\circ\in\mathbb R^n$ let us consider the following controlled nonlinear dynamical systems
\begin{equation}
\label{Eq:ContrDynamic}
\dot y(t)=f\big(y(t),u(t)\big)\in\mathbb R^n \text{ for } t>0,\quad y(0)=y_\circ\in\mathbb R^n
\end{equation}
together with the infinite horizon cost functional
\begin{equation}
\label{Eq:Cost}
J(y,u)=\int_0^\infty g\big(y(t),u(t)\big)e^{-\lambda t}\,\mathrm dt
\end{equation}
In \eqref{Eq:Cost} we assume that $\lambda> 0$ is a given weighting parameter and $g$ maps $\mathbb R^n\times\mathbb R^m$ to $\mathbb R$. We call $y$ the {\em state} and $u$ the {\em control}. The set of admissible controls has the form
\[
\UAD=\big\{u\in \mathbb U\,\big|\,u(t)\in \Uad\text{ for almost all }t\ge 0\big\},
\]
where we set $\mathbb U=L^2(0,\infty;\mathbb R^m)$ and $\Uad\subset\mathbb R^m$ denotes a compact, convex subset.

Let $\mathrm M\in\mathbb R^{n\times n}$ denote a symmetric, positive definite (mass) matrix with smallest and largest positive eigenvalues $\lambda_{min}$ and $\lambda_{max}$, respectively. Then, we introduce the following weighted inner product in $\mathbb R^n$:
\[
{\langle y,\tilde y\rangle}_\mathrm M=y^\top \mathrm M\tilde y\quad\text{for }y,\tilde y\in\mathbb R^n,
\]
where `$\top$' stands for the transpose of a given vector or matrix. By $\|\cdot\|_\mathrm M=\langle\cdot\,,\cdot\rangle_\mathrm M^{1/2}$ we define the associated induced norm. Recall that we have
\[
\lambda_{min}\,{\|y\|}_2^2\le{\|y\|}_\mathrm M^2\le\lambda_{max}\,{\|y\|}_2^2\quad\text{for all }y\in\mathbb R^n.
\]
Then, $y$ solves \eqref{Eq:ContrDynamic} if
\begin{equation}
\label{VarForm}
\begin{aligned}
{\langle \dot y(t)-f(y(t),u(t)),\varphi\rangle}_\mathrm M&=0&&\text{for all }\varphi\in \mathbb R^n\text{ and for almost all }t>0,\\
{\langle y(0)-y_\circ,\varphi\rangle}_\mathrm M&=0&&\text{for all }\varphi\in \mathbb R^n
\end{aligned}
\end{equation}
We call \eqref{VarForm} the {\em variational formulation} of the dynamical system. Let us suppose that \eqref{Eq:ContrDynamic} has a unique solution $y=y(u;y_\circ)\in \mathbb Y=H^1(0,\infty;\mathbb R^n)$ for every admissible control $u\in\UAD$ and for every initial condition $y_\circ\in\mathbb R^n$; see, e.g., \cite[Chapter~III]{BC97}. Thus, we can define the reduced cost functional as follows:
\[
\widehat J(u;y_\circ)=J(y(u;y_\circ),u)\quad\text{for }u\in\UAD\text{ and }y_\circ\in\mathbb R^n,
\]
where $y(u;y_\circ)$ solves \eqref{Eq:ContrDynamic} for given control $u$ and initial condition $y_\circ$. Then, our optimal control can be formulated as follows: for given $y_\circ\in\mathbb R^n$ we consider
\begin{equation}
\label{Phat}
\tag{$\mathbf{\widehat P}$} 
\min_{u\in\UAD} \widehat J(u;y_\circ).
\end{equation}

%%%%%%%%%%%%%%%%%%%%%%%%%%%%%%%%%%%%%%%%%%%%%%%%%%%%%
\subsection{The Hamilton-Jacobi-Bellman equation and its time discretization}
\label{Section2.2}
%%%%%%%%%%%%%%%%%%%%%%%%%%%%%%%%%%%%%%%%%%%%%%%%%%%%%

We define the  {\em value function} of the problem $v:\mathbb R^n\to\mathbb R$ as follows:
\[
v(y)=\inf \big\{\widehat J(u;y)\,\big|\,u\in\UAD\big\}\quad \text{for }y\in\mathbb R^n.
\]
This function gives the best value for every initial condition, given the set of admissible controls $\Uad$. It is characterized as the unique viscosity solution of the  {\em Hamilton-Jacobi-Bellman (HJB) equation} corresponding the infinite horizon 
\begin{equation}
\label{hjb}
\lambda v(y)+\sup_{u\in\Uad}\big\{-f(y,u)\cdot \nabla v(y)-g(y,u)\big\}=0\quad \text{for }y\in\mathbb R^n.
\end{equation}
In order to construct the approximation scheme (as in \cite{Fal87}) let us consider first a time discretization where $h$ is a strictly positive step size. A dynamic programming principle for the discrete time problem holds true giving the following semi-discrete scheme for \eqref{hjb} 
\begin{equation}
\label{hjb_h}
v_h(y)=\min_{u\in\Uad}\big\{(1-\lambda h)v_h(y+h f(y,u))+hg(y,u)\big\}\quad \text{for }y\in\mathbb R^n.
\end{equation}
Throughout our paper we suppose the following hypotheses.

\begin{assumption}
\label{Assumption1}
\begin{enumerate}
\item [\rm 1)] The right-hand side $f:\mathbb R^n\times\mathbb R^m\to\mathbb R^n$ is continuous and globally Lipschitz-continuous in the first argument, i.e., there exists an $L_f>0$ satisfying
\[
{\|f(y,u)-f(\tilde y,u)\|}_2\le L_f\,{\|y-\tilde y\|}_2\text{ for all }y,\tilde y\in\mathbb R^n\text{ and }u\in\Uad
\]
Furthermore, $\|f(y,u)\|_\infty=\max_{1\le i\le n}\big|f_i(y,u)\big|$ is bounded by a constant $M_f$ for all $y\in\overline\Omega$ and $u\in\Uad$.
\item [\rm 2)] The running cost $g:\mathbb R^n\times\mathbb R^m\to\mathbb R^n$ is continuous and globally Lipschitz-continuous in the first argument with a Lipschitz constant $L_g>0$. Moreover,  $\|g(y,u)\|_\infty\le M_g$ for all $(y,u)\in\overline\Omega\times\Uad$ with $M_g>0$.
\end{enumerate}
\end{assumption}

If Assumption~\ref{Assumption1} and $\lambda>L_f$ hold, the function $v_h$ is Lipschitz-continuous satisfying
\begin{equation}
\label{vh-Lipschitz}
\big|v_h(y)-v_h(\tilde y)\big|\le \frac{L_g}{\lambda-L_f}\,{\|y-\tilde y\|}_2\quad\text{for all }y,\tilde y\in\overline\Omega\text{ and }h\in[0,1/\lambda);
\end{equation}
see \cite[p.~473]{Fal97}. Let us recall the following result \cite[Theorem~2.3]{Fal87}:

\begin{theorem}
\label{Th:ConvRate-vh}
Let Assumption~{\rm\ref{Assumption1}} and  $\lambda>\max\{L_g,L_f\}$ hold. Let $v$ and $v_h$ be the solutions of \eqref{hjb} and \eqref{hjb_h}, respectively. Suppose the semiconcavity assumptions
\begin{equation}
\label{SemiConc}
\begin{aligned}
{\|f(y+\tilde y,u)-2f(y,u)+f(y-\tilde y,u)\|}_2&\le C_f\,{\|\tilde y\|}^2_2,\\\big|g(y+\tilde y,u)-2g(y,u)+g(y-\tilde y,u)\big|&\le C_g\,{\|\tilde y\|}^2_2
\end{aligned}
\end{equation}
for all $(y,\tilde y,u)\in\mathbb R^n\times\mathbb R^n\times\Uad$. Then, there is a constant $C\ge 0$ satisfying
\[
\sup_{y\in\mathbb R^n}\big|v(y)-v_h(y)\big|\le C h\quad\text{for any }h\in[0,1/\lambda).
\]
\end{theorem}

\begin{remark}
\label{Remark:TimeDer}
\rm
The constant $C$ in Theorem~\ref{Th:ConvRate-vh} can be bounded by
\[
C\le\max\bigg\{\frac{L_g}{\lambda},\frac{2M_f^2}{\lambda},\frac{L_f}{\lambda},M_f\bigg\}\bigg(2+\frac{L_g}{\lambda-L_f}\bigg)^2;
\]
see \cite[Remark~1]{Fal87}.\hfill$\Diamond$
\end{remark}

%%%%%%%%%%%%%%%%%%%%%%%%%%%%%%%%%%%%%%%%%%%%%%%%%%%%%
\subsection{The large-scale approximation of the HJB equations}
\label{Section2.3}
%%%%%%%%%%%%%%%%%%%%%%%%%%%%%%%%%%%%%%%%%%%%%%%%%%%%%

For the numerical realization we have to restrict ourselves to a bounded subset of $\mathbb R^n$. Suppose that there exists a (bounded) polyhedron $\Omega\subset\mathbb R^n$ such that for sufficiently small $h>0$
\begin{equation}
\label{AssumpOmega}
y+hf(y,u)\in\overline\Omega,\quad\text{for all }y\in\overline\Omega\text{ and }u\in\Uad.
\end{equation}
We want to point out that the above invariance condition is used here to simplify the problem and focus on the main issue of the a-priori error estimate. If \eqref{AssumpOmega} is not satisfied one can apply state constraints to the problem and use appropriate boundary conditions provided at avery point of the boundary there exists at least one control point inside $\Omega$
(for this and even more general state constraints boundary conditions the interested reader can find in \cite{FFbook} some hints and additional references).
Let $\{\mathscr S_j\}_{j=1}^{m_\mathscr S}$ be a family of simplices which defines a regular triangulation of the polyhedron $\Omega$ (see, e.g., \cite{GLT76}) such that
\[
\overline\Omega=\bigcup_{j=1}^{m_\mathscr S}\mathscr S_j\quad\text{and}\quad k=\max_{1\le j\le m_\mathscr S}\big(\mathrm{diam}\,\mathscr S_j\big).
\]
Throughout this paper we assume that we have $n_\mathscr S$ vertices/nodes $y_1,\ldots,y_{n_\mathscr S}$ in the triangulation. 
Let $V^k$ be the space of piecewise affine functions from $\overline\Omega$ to $\mathbb R$ which are continuous in $\overline\Omega$ having constant gradients in the interior of any simplex $\mathscr S_j$ of the triangulation. Then, a fully discrete scheme for the HJB equations is given by
\begin{equation}
\label{hjb_hk}
v_{hk}(y_i)=\min_{u\in\Uad}\big\{(1-\lambda h)v_{hk}\big(y_i+h f(y_i,u)\big)+hg(y_i,u)\big\}
\end{equation}
for any vertex $y_i\in\overline\Omega$. Clearly, a solution to \eqref{hjb_h} satisfies \eqref{hjb_hk}.\\
\noindent
Let us recall the following result \cite[Corollary~2.4]{Fal87} and \cite[Theorem~1.3]{Fal97}:

\begin{theorem}
\label{Theorem:ConvRate}
Assume that Assumption~{\rm\ref{Assumption1}}, \eqref{SemiConc} and \eqref{AssumpOmega} hold. Let $v$, $v_h$ and $v_{hk}$ be the solutions of \eqref{hjb}, \eqref{hjb_h} and \eqref{hjb_hk}, respectively. For $\lambda>L_f$ we obtain
\begin{equation}
\label{HJB-vh}
\sup_{y\in\overline\Omega}\big|v_h(y)-v_{hk}(y)\big|\le \frac{L_f}{\lambda(\lambda-L_f)}\,\frac{k}{h}\quad\text{for any }h\in[0,1/\lambda).
\end{equation}
For $\lambda>\max\{L_f,2L_g\}$ we have
\begin{equation}
\label{hjb:est2}
\sup_{y\in\overline\Omega}\big|v(y)-v_{hk}(y)\big|\le C h+\frac{L_g}{\lambda-L_f}\,\frac{k}{h}\quad\text{for any }h\in[0,1/\lambda).
\end{equation}
\end{theorem}

\begin{corollary}
\label{Co:Estim}
Assume that Assumption~{\rm\ref{Assumption1}}, \eqref{SemiConc} and $\lambda>L_f$ hold. Let $v_{hk}$ be the solution of \eqref{hjb_hk}. Then, we have for $h\in[0,1/\lambda)$
\[
\big|v_{hk}(y)-v_{hk}(\tilde y)\big|\le C_1\,\frac{k}{h}+C_2\,{\|y-\tilde y\|}_2\quad\text{for all }y,\tilde y\in\overline\Omega,
\]
where $C_1=2L_f/(\lambda(\lambda-L_f))$ and $C_2=L_g/(\lambda-L_f)$.
\end{corollary}

\begin{proof}
Suppose that $v_h$ is the solution to \eqref{hjb_h}. Then, we derive from \eqref{vh-Lipschitz} and \eqref{HJB-vh} that
\begin{align*}
\big|v_{hk}(y)-v_{hk}(\tilde y)\big|&\le\big|v_{hk}(y)-v_h(y)\big|+\big|v_h(y)-v_h(\tilde y)\big|+\big|v_h(\tilde y)-v_{hk}(\tilde y)\big|\\
&\le\frac{2L_f}{\lambda(\lambda-L_f)}\,\frac{k}{h}+\frac{L_g}{\lambda-L_f}\,{\|y-\tilde y\|}_2,
\end{align*}
which gives the claim.
\end{proof}

%%%%%%%%%%%%%%%%%%%%%%%%%%%%%%%%%%%%%%%%%%%%%%%%%%%%%
\section{The POD method and reduced-order modeling}
\label{Section3}
\setcounter{section}{3}
\setcounter{equation}{0}
\setcounter{theorem}{0}
\setcounter{algorithm}{0}
\renewcommand{\theequation}{\arabic{section}.\arabic{equation}}
%%%%%%%%%%%%%%%%%%%%%%%%%%%%%%%%%%%%%%%%%%%%%%%%%%%%%

The focus of this section is the construction of surrogate models by means of the Proper Orthogonal Decomposition (POD). Here we recall the basics of the method and apply the POD method to optimal control problems.

%%%%%%%%%%%%%%%%%%%%%%%%%%%%%%%%%%%%%%%%%%%%%%%%%%%%%
\subsection{POD for parametrized nonlinear dynamical systems}
\label{Section3.1}
%%%%%%%%%%%%%%%%%%%%%%%%%%%%%%%%%%%%%%%%%%%%%%%%%%%%%

For $\wnp\in\mathbb N$ let us choose different pairs controls $\{(u^\nu,y_\circ^\nu)\}_{\nu=1}^\wnp$ in $\UAD\times\overline\Omega$. By  $y^\nu=y(u^\nu;y_\circ^\nu)\in\mathbb Y$, $\nu=1,\ldots,\wnp$, we denote the solution to \eqref{Eq:ContrDynamic}. We introduce the {\em snapshot subspace} as
\[
\mathscr V=\mathrm{span}\,\big\{y^\nu(t)\,\big|\,t\in[0,\infty)\text{ and }1\le \nu\le\wnp\big\}\subset\mathbb R^n.
\]
For every $\ell\in\{1,\ldots,d\}$, with dimension $d\le n$, a {\em POD basis of rank $\ell$} is defined as a solution to the minimization problem (see, e.g., \cite{HLBR12})
\begin{equation}
\label{Pell}
\tag{$\mathbf P^{\boldsymbol \ell}$}
\left\{
\begin{aligned}
&\min \sum_{\nu=1}^\wnp\int_0^\infty\Big\|y^\nu(t)-\sum_{i=1}^\ell {\langle y^\nu(t),\psi_i\rangle}_\mathrm M\,\psi_i\Big\|_\mathrm M^2\,\mathrm dt\\
&\hspace{1mm}\text{such that }\{\psi_i\}_{i=1}^\ell\subset\mathbb R^n\text{ and }{\langle\psi_i,\psi_j\rangle}_\mathrm M=\delta_{ij},~1\le i,j\le \ell,
\end{aligned}
\right.
\end{equation}
where $\delta_{ij}$ is the Kronecker symbol satisfying $\delta_{ii}=0$ and $\delta_{ij}=0$ for $i\neq j$. It is well-known that a solution to \eqref{Pell} is given by a solution to the eigenvalues problem
\[
\mathcal R\psi_i=\lambda_i\psi_i\quad\text{for }\lambda_1\ge\lambda_2\ge\ldots\ge\lambda_\ell\ge\lambda_d>0
\]
with the linear, bounded, symmetric integral operator $\mathcal R:\mathbb R^n\to\mathscr V$
\[
\mathcal R\psi=\sum_{\nu=1}^\wnp\int_0^\infty{\langle y^\nu(t),\psi\rangle}_\mathrm M\,y^\nu(t)\,\mathrm dt\quad\text{for }\psi\in\mathbb R^n
\]
(compare, e.g., \cite{CGS12,GV13,Sin13}). If $\{\psi_i\}_{i=1}^\ell$ is a solution to \eqref{Pell}, we have the approximation error
\begin{equation}
\label{ErrorFormula}
\sum_{\nu=1}^\wnp\int_0^\infty\Big\|y^\nu(t)-\sum_{i=1}^\ell {\langle y^\nu(t),\psi_i\rangle}_\mathrm M\,\psi_i\Big\|_M^2\,\mathrm dt=\sum_{i=\ell+1}^d\lambda_i.
\end{equation}

In real computations, we do not have the whole trajectory $y(t)$ for all $t\in[0,\infty)$. For that purpose we choose $T\gg 0$ sufficiently large and define a grid in $[0,\te]$, where $\te\ge T$, by $0=t_1<t_2<\ldots<t_{n_T}=\te$. Let $y_j^\nu\approx y^\nu(t_j)\in\mathbb R^n$ denote approximations for the introduced trajectories $\{y_j^\nu\}_{\nu=1}^\wnp$ at the time instance $t_j$ for $j=1,\ldots,n_T$. We set $\mathscr V^{n_T}=\mathrm{span}\,\{y_j^\nu\,|\,1\le j\le n_T,\,1\le\nu\le \wnp\}$ with $d^{n_T}=\dim\mathscr V^{n_T}\le \min(n,n_T\wnp)$. Then, for every $\ell\in\{1,\ldots,d^{n_T}\}$ we consider the minimization problem
\begin{equation}
\label{Pelldisc}
\tag{$\mathbf P^{\boldsymbol \ell}_{\boldsymbol n_T}$}
\left\{
\begin{aligned}
&\min \sum_{\nu=1}^\wnp\sum_{j=1}^{n_T}\alpha_j^{n_T}\,\Big\|y^\nu_j-\sum_{i=1}^\ell {\langle y^\nu_j,\psi_i^{n_T}\rangle}_\mathrm M\,\psi_i^{n_T}\Big\|_\mathrm M^2\\
&\hspace{1mm}\text{such that }\{\psi_i^{n_T}\}_{i=1}^\ell\subset\mathbb R^n\text{ and }{\langle\psi_i^{n_T},\psi_j^{n_T}\rangle}_\mathrm M=\delta_{ij},~1\le i,j\le \ell,
\end{aligned}
\right.
\end{equation}
instead of \eqref{Pell}. In \eqref{Pelldisc} the $\alpha_j$'s stand for the trapezoidal weights
\[
\alpha_1^{n_T}=\frac{t_2-t_1}{2},~\alpha_j^{n_T}=\frac{t_j-t_{j-1}}{2}\text{ for }2\le j\le n_T-1,\alpha_{n_T}^{n_T}=\frac{t_{n_T}-t_{n_T-1}}{2}
\]
The solution to \eqref{Pelldisc} is given by the solution to the eigenvalue problem \cite{GV13,HLBR12}
\[
\mathcal R^{n_T}\psi_i^{n_T}=\lambda_i^{n_T}\psi_i^{n_T}\quad\text{for }\lambda_1^{n_T}\ge\lambda_2^{n_T}\ge\ldots\ge\lambda_\ell^{n_T}\ge\lambda_{d{n_T}}>0
\]
with the linear, bounded, symmetric and nonnegative operator
\[
\mathcal R^{n_T}\psi=\sum_{\nu=1}^{n_T}\sum_{j=1}^{n_T}\alpha_j^{n_T}\,{\langle y_j^\nu,\psi\rangle}_\mathrm M\,y_j^\nu\quad\text{for }\psi\in\mathbb R^n.
\]
Analogous to \eqref{ErrorFormula} a solution to \eqref{Pelldisc} satisfies
\[
\sum_{\nu=1}^\wnp\sum_{j=1}^{n_T}\alpha_j^{n_T}\,\Big\|y^\nu_j-\sum_{i=1}^\ell {\langle y^\nu_j,\psi_i^{n_T}\rangle}_\mathrm M\,\psi_i^{n_T}\Big\|_\mathrm M^2=\sum_{i=\ell+1}^{d^{n_T}}\lambda_i^{n_T}.
\]
The relationship between \eqref{Pell} and \eqref{Pelldisc} is investigated in \cite{GV13,KV02}.

%%%%%%%%%%%%%%%%%%%%%%%%%%%%%%%%%%%%%%%%%%%%%%%%%%%%%
\subsection{Reduced-order modelling for the state equation}
\label{Section3.2}
%%%%%%%%%%%%%%%%%%%%%%%%%%%%%%%%%%%%%%%%%%%%%%%%%%%%%

We introduce the POD coefficient matrix
\[
\Psi=\big[\psi_1\,|\,\ldots\,|\,\psi_\ell\big]\in\mathbb R^{n\times \ell}
\]
and the subspace $V^\ell=\mathrm{span}\,\{\psi_1,\ldots,\psi_\ell\}\subset\mathbb R^n$. In particular, the matrix $\mathrm M^\ell=\Psi^\top\mathrm M\Psi\in\mathbb R^{\ell\times\ell}$ is the identity matrix. The reduced-order model for \eqref{VarForm} is derived as follows: we replace the vector $y(t)\in\mathbb R^n$ by its POD approximation $\Psi y^\ell(t)\in\mathbb R^n$ with the unknown time dependent coefficients $y^\ell(t)\in\mathbb R^\ell$ and choose $\varphi=\psi_i$ for $i=1,\ldots,\ell$. It follows that
\begin{equation}
\label{Eq:ROMSystem}
\dot y^\ell(t)=f^\ell\big(y^\ell(t),u(t)\big)\in\mathbb R^\ell\text{ for } t>0,\quad y^\ell(0)=y_\circ^\ell\in\mathbb R^\ell,
\end{equation}
where we have set $y^\ell_\circ=\Psi^\top\mathrm My_\circ\in\mathbb R^\ell$ and $f^\ell(y^\ell,u)=\Psi^\top\mathrm Mf(\Psi y^\ell,u)\in\mathbb R^\ell$ for $(y^\ell,u)\in\mathbb R^\ell\times\Uad$, i.e. no discrete interpolation method is used at the moment (compare, e.g., \cite{BMNP04,CS10}).

Suppose that \eqref{Eq:ROMSystem} possesses a unique solution $y^\ell=y^\ell(u;y_\circ)\in\mathbb Y^\ell=H^1(0,\infty;\mathbb R^\ell)$ for any admissible control $u\in\UAD$. Let us introduce the linear, orthogonal projection $\mathcal P^\ell:\mathbb R^n\to V^\ell$ as
\[
\mathcal P^\ell y=\Psi\Psi^\top\mathrm My=\sum_{i=1}^\ell {\langle y,\psi_i\rangle}_\mathrm M\,\psi_i\quad\text{for }y\in \mathbb R^n.
\]
We note that the error of a solution to \eqref{Eq:ContrDynamic} and \eqref{Eq:ROMSystem} on a finite time horizon can be estimated; see \cite[Theorem~2.1 and Remark~2.2]{Vol11}.

\begin{proposition}
For given $y_\circ\in\Omega$, $u\in\UAD$ and $\te>0$ let $y=y(u;y_\circ)\in H^1(0,\te;\mathbb R^n)$ be the unique solution to \eqref{Eq:ContrDynamic} on the finite time interval $[0,\te]$. In \eqref{Pell} we choose $\wnp=2$, $y^1=y$ and $y^2=\dot y$. Suppose that $y^\ell\in H^1(0,\te;\mathbb R^n)$ is the unique solution to \eqref{Eq:ROMSystem} for $\ell\in\{1,\ldots,d\}$. Then, it follows that
\[
\int_0^\te{\|y(t)-y^\ell(t)\|}_\mathrm M^2\,\mathrm dt\le\widehat C\sum_{i=\ell+1}^d\lambda_i
\]
for a constant $\widehat C>0$.
\end{proposition}

%%%%%%%%%%%%%%%%%%%%%%%%%%%%%%%%%%%%%%%%%%%%%%%%%%%%%
\subsection{Reduced-order modelling for the optimal control problem}
\label{Section3.3}
%%%%%%%%%%%%%%%%%%%%%%%%%%%%%%%%%%%%%%%%%%%%%%%%%%%%%

Next we introduce the POD reduced-order model for \eqref{Phat}. For given $(u,y_\circ)\in\UAD\times\Omega$ let $y^\ell=y^\ell(u;y_\circ)\in\mathbb Y^\ell$ denote the unique solution to \eqref{Eq:ROMSystem}. Then, the reduced POD cost is given by
\begin{align*}
\widehat J^\ell(u;y_\circ)&=J(y^\ell(u;y_\circ),u)=J(y^\ell,u)\\
&=\int_0^\infty g\big(\Psi y^\ell(t),u(t)\big)e^{-\lambda t}\,\mathrm dt=\int_0^\infty g^\ell\big(y^\ell(t),u(t)\big)e^{-\lambda t}\,\mathrm dt,
\end{align*}
where we have set $g^\ell(y^\ell,u)=g(\Psi y^\ell,u)$ for $(y^\ell,u)\in\mathbb R^\ell\times\Uad$.
Then, the POD approximation for \eqref{Phat} reads as follows: for given $y_\circ\in\Omega$ we consider
\begin{equation}
\label{Phatell}
\tag{$\mathbf{\widehat P}^{\boldsymbol \ell}$} 
\min \widehat J^\ell(u;y_\circ)\quad\text{such that }\quad u\in\UAD.
\end{equation}

%%%%%%%%%%%%%%%%%%%%%%%%%%%%%%%%%%%%%%%%%%%%%%%%%%%%%
\section{A-priori error for the HJB-POD  approximation}
\label{Section4}
\setcounter{section}{4}
\setcounter{equation}{0}
\setcounter{theorem}{0}
\setcounter{algorithm}{0}
\renewcommand{\theequation}{\arabic{section}.\arabic{equation}}
%%%%%%%%%%%%%%%%%%%%%%%%%%%%%%%%%%%%%%%%%%%%%%%%%%%%%

In this section we present the a-priori error analysis for the coupling between the HJB equation and the POD method. Our first a-priori error estimate is better from a theoretical point of view, whereas for the numerical realization the second a-priori error estimate is much more appropriate. In the first estimate we assume to work in $\mathbb R^\ell$ on a number of vertices which have been obtained mapping the $y_i$ nodes of $\mathbb R^n$ into $\mathbb R^\ell$. Even if the maximum distance between the $y_i$ neighbouring nodes is bounded by $k$, this clearly produces a non uniform grid where the distance between the neighbouring nodes can not be predicted a-priori since it depends on $\Psi$. The second error estimate takes into account a 
uniform grid of size $K$ in $\mathbb R^\ell$. 

%%%%%%%%%%%%%%%%%%%%%%%%%%%%%%%%%%%%%%%%%%%%%%%%%%%%%
\subsection{First a-priori error estimate}
\label{Section4.1}
%%%%%%%%%%%%%%%%%%%%%%%%%%%%%%%%%%%%%%%%%%%%%%%%%%%%%

We introduce two different POD approximations for the HJB equation. The first one is based on \eqref{hjb_hk}, where we project all vertices $\{y_i\}_{i=1}^{n_\mathscr S}$ into $\mathbb R^\ell$ by setting
\[
y_i^\ell=\Psi^\top\mathrm My_i\quad\text{for }i=1,\ldots,n_\mathscr S.
\]
Here we assume that $y_i^\ell\neq y_j^\ell$ holds for $i,j\in\{1,\ldots,n_\mathscr S\}$ with $i\neq j$. Then, a POD discretization of \eqref{hjb_hk} is given by
\begin{equation}
\label{hjb_hkPOD}
v_{hk}^\ell(y_i^\ell)=\min_{u\in\Uad}\big\{(1-\lambda h)v_{hk}^\ell\big(y_i^\ell+h f^\ell(y_i^\ell,u)\big)+hg^\ell(y_i^\ell,u)\big\}
\end{equation}
for $1\le i\le n_\mathscr S$. We define the mapping $\widetilde v_{hk}^\ell:\overline\Omega\to\mathbb R$ by
\[
\widetilde v_{hk}^\ell(y)=v_{hk}^\ell(\Psi^\top\mathrm My)\quad\text{for all }y\in\overline\Omega\text{ with }\Psi^\top\mathrm My\in\overline\Omega.
\]
Using $\mathcal P^\ell=\Psi\Psi^\top\mathrm M\in\mathbb R^{n\times n}$ we have
\begin{align*}
&\hspace{-2.5mm}\widetilde v_{hk}^\ell(y_i)=v_{hk}^\ell(\Psi^\top\mathrm My_i)=v_{hk}^\ell(y_i^\ell)\\
&\hspace{-2.5mm}=\min_{u\in\Uad}\big\{(1-\lambda h)v_{hk}^\ell\big(y_i^\ell+h f^\ell(y_i^\ell,u)\big)+hg^\ell(y_i^\ell,u)\big\}\\
&\hspace{-2.5mm}=\min_{u\in\Uad}\big\{(1-\lambda h)v_{hk}^\ell\big(\Psi^\top\mathrm M (y_i+h f(\mathcal P^\ell y_i,u)\big)+hg(\mathcal P^\ell y_i,u)\big\}\\
&\hspace{-2.5mm}=\min_{u\in\Uad}\big\{(1-\lambda h)\widetilde v_{hk}^\ell\big(y_i+h f(\mathcal P^\ell y_i,u)\big)+hg(\mathcal P^\ell y_i,u)\big\}
\end{align*}
for $1\le i\le n_\mathscr S$. Thus, \eqref{hjb_hkPOD} can be equivalently expressed as
\begin{equation}
\label{HJB-Est-3}
\widetilde v_{hk}^\ell(y_i)=\min_{u\in\Uad}\big\{(1-\lambda h)\widetilde v_{hk}^\ell\big(y_i+h f(\mathcal P^\ell y_i,u)\big)+hg(\mathcal P^\ell y_i,u)\big\}
\end{equation}
for $1\le i\le n_\mathscr S$. The following result measures the error between a solution to \eqref{hjb_h} and a solution to \eqref{HJB-Est-3}. The proof is similar to the proof of Theorem~1.3 in \cite{Fal97} and requires an invariance condition which will be discussed later in Remark \ref{invarrem}.

\begin{proposition}
\label{ROM-Theo}
Assume that Assumption~{\rm\ref{Assumption1}}, \eqref{AssumpOmega} and $\lambda>L_f$ hold. Let $v_h$ and $\widetilde v^\ell_{hk}$ be the solutions of \eqref{hjb_h} and \eqref{HJB-Est-3}, respectively and let the invariance condition
\begin{equation}
\label{AssumptionPol}
y_i+hf(\mathcal P^\ell y_i,u)\in\overline\Omega\quad\text{for }i=1,\ldots,n_\mathscr S\text{ and for all }u\in\Uad
\end{equation}
be satisfied. Then, there exist two constants $\widehat C_0,\widehat C_1$ such that
\[
\sup_{y\in\overline\Omega}\big|v_h(y)-\widetilde v_{hk}^\ell(y)\big|\le \widehat C_0\,\frac{k}{h}+ \widehat C_1\bigg(\sum_{i=1}^{n_\mathscr S}{\|y_i-\mathcal P^\ell y_i\|}_2^2\bigg)^{1/2}\quad\text{for any }h\in[0,1/\lambda).
\]
\end{proposition}

\begin{proof}
For any $y\in\Omega$ there are real coefficients $\mu_i=\mu_i(y)$, $1\le i\le n_\mathscr S$, of the convex combination representation of $y$ satisfying
\[
y=\sum_{i=1}^{n_\mathscr S}\mu_iy_i,\quad 0\le\mu_i\le 1\quad\text{and}\quad\sum_{i=1}^{n_\mathscr S}\mu_i=1.
\]
Since $\widetilde v_{hk}^\ell$ is piecewise affine, we obtain $\widetilde v_{hk}^\ell(y)=\sum_{i=1}^{n_\mathscr S}\mu_i\widetilde v_{hk}^\ell(y_i)$. Thus, we have
\begin{equation}
\label{Est-00}
\big|v_h(y)-\widetilde v_{hk}^\ell(y)\big|\le\bigg|\sum_{i=1}^{n_\mathscr S}\mu_i\big(v_h(y)-v_h(y_i)\big)\bigg|+\bigg|\sum_{i=1}^{n_\mathscr S}\mu_i\big(v_h(y_i)-\widetilde v_{hk}^\ell(y_i)\big)\bigg|.
\end{equation}
From $y\in\overline\Omega$ we infer that there exists an index $j$ with $y\in\overline{\mathscr S}_j\subset\overline\Omega$. Let $\mathscr I_j=\{i_1,\ldots,i_\mathsf k\}\subset\{1,\ldots,n_\mathscr S\}$ denote the index subset such that $y_i\in\overline{\mathscr S}_j$ holds for $i\in\mathscr I_j$. Then, $\mu_i=0$ holds for all $i\not\in\mathscr I_j$. Moreover, $\sum_{i=1}^{n_\mathscr S}\mu_i=\sum_{i\in\mathscr I_j}\mu_i=1$ and $\|y-y_i\|_2\le k$ for any $i\in\mathscr I_j$. From \eqref{vh-Lipschitz} we have
\begin{equation}
\label{Est-0}
\sum_{i=1}^{n_\mathscr S}\mu_i\big|v_h(y)-v_h(y_i)\big|=\sum_{i\in\mathscr I_j}\mu_i\big|v_h(y)-v_h(y_i)\big|\le \frac{L_g}{\lambda-L_f}\,k
\end{equation}
for $h\in[0,1/\lambda)$. Using \eqref{HJB-Est-3} and \eqref{hjb_h} we have
\begin{equation}
\label{Est-1}
\begin{aligned}
&v_h(y_i)-\widetilde v_{hk}^\ell(y_i)\\
&\le v_h(y_i)-(1-\lambda h)\widetilde v_{hk}^\ell\big(y_i+h f(\mathcal P^\ell y_i,\bar u_{hk}^{\ell,i})\big)+hg(\mathcal P^\ell y_i,\bar u_{hk}^{\ell,i})\\
&\le (1-\lambda h)\Big(v_h\big((y_i+h f(y_i,\bar u_{hk}^{\ell,i})\big)-\widetilde v_{hk}^\ell\big(y_i+h f(\mathcal P^\ell y_i,\bar u_{hk}^{\ell,i})\big)\Big)\\
&\quad+h\big(g(y_i,\bar u_{hk}^{\ell,i})-g(\mathcal P^\ell y_i,\bar u_{hk}^{\ell,i})\big),
\end{aligned}
\end{equation}
where $\bar u_{hk}^{\ell,i}\in\Uad$ is defined as
\begin{equation}
\label{uMin}
\bar u_{hk}^{\ell,i}=\mathrm{arg}\hspace{-0.6mm}\min_{\hspace{-3mm}u\in\Uad}\big\{(1-\lambda h)\widetilde v_{hk}^\ell\big(y_i+h f(\mathcal P^\ell y_i,u)\big)+hg(\mathcal P^\ell y_i,u)\big\}.
\end{equation}
Applying \eqref{vh-Lipschitz} again we deduce that
\[
\big|v_h\big((y_i+h f(y_i,\bar u_{hk}^{\ell,i})\big)-v_h\big(y_i+h f(\mathcal P^\ell y_i,\bar u_{hk}^{\ell,i})\big)\big|\le \frac{hL_gL_f}{\lambda-L_f}\,{\|y_i-\mathcal P^\ell y_i\|}_2
\]
for $1\le i\le n_\mathscr S$ and $h\in[0,1/\lambda)$. Hence, from \eqref{AssumptionPol} it follows
\begin{align*}
&v_h\big((y_i+h f(y_i,\bar u_{hk}^{\ell,i})\big)-\widetilde v_{hk}^\ell\big(y_i+h f(\mathcal P^\ell y_i,\bar u_{hk}^{\ell,i})\big)\\
&\quad\le \Big(v_h\big((y_i+h f(y_i,\bar u_{hk}^{\ell,i})\big)-v_h\big(y_i+h f(\mathcal P^\ell y_i,\bar u_{hk}^{\ell,i})\big)\Big)\\
&\qquad+ \Big(v_h\big(y_i+h f(\mathcal P^\ell y_i,\bar u_{hk}^{\ell,i})\big)-\widetilde v_{hk}^\ell\big(y_i+h f(\mathcal P^\ell y_i,\bar u_{hk}^{\ell,i})\big)\Big)\\
&\quad\le \frac{hL_gL_f}{\lambda-L_f}\,{\|y_i-\mathcal P^\ell y_i\|}_2+\sup_{y\in\overline\Omega}\big|v_h(y)-\widetilde v_{hk}^\ell(y)\big|
\end{align*}
for $1\le i\le n_\mathscr S$ and $h\in[0,1/\lambda)$. Using the inequality
\[
h\big(g(y_i,\bar u_{hk}^{\ell,i})-g(\mathcal P^\ell y_i,\bar u_{hk}^{\ell,i})\big)\le hL_g\,{\|y_i-\mathcal P^\ell y_i\|}_2
\]
we derive from \eqref{Est-1}
\[
v_h(y_i)-\widetilde v_{hk}^\ell(y_i)\le \widetilde C_1h\,{\|y_i-\mathcal P^\ell y_i\|}_2+(1-\lambda h)\sup_{y\in\overline\Omega}\big|v_h(y)-\widetilde v_{hk}^\ell(y)\big|
\]
for $1\le i\le n_\mathscr S$ and $h\in[0,1/\lambda)$ with $\widetilde C_1=L_g(L_f/(\lambda-L_f)+1)$. By interchanging the role of $v_h$ and $\widetilde v_{hk}^\ell$ in \eqref{Est-1} we derive
\begin{equation}
\label{Est-2}
\big|v_h(y_i)-\widetilde v_{hk}^\ell(y_i)\big|\le \widetilde C_1h\,{\|y_i-\mathcal P^\ell y_i\|}_2+(1-\lambda h)\sup_{y\in\overline\Omega}\big|v_h(y)-\widetilde v_{hk}^\ell(y)\big|
\end{equation}
for $1\le i\le n_\mathscr S$ and $h\in[0,1/\lambda)$. Note that $0\le\sum_{i=1}^{n_\mathscr S}\mu_i^2\le\sum_{i=1}^{n_\mathscr S}\mu_i=1$ holds for the coefficients in the convex combination representation. Inserting \eqref{Est-0} and \eqref{Est-2} into \eqref{Est-00} we find
\begin{align*}
&\big|v_h(y)-\widetilde v_{hk}^\ell(y)\big|\\
&\le(1-\lambda h)\sup_{y\in\overline\Omega}\big|v_h(y)-\widetilde v_{hk}^\ell(y)\big|+\widetilde C_0\,k+\widetilde C_1h\sum_{i=1}^{n_\mathscr S}\mu_i\,{\|y_i-\mathcal P^\ell y_i\|}_2\\
&\le(1-\lambda h)\sup_{y\in\overline\Omega}\big|v_h(y)-\widetilde v_{hk}^\ell(y)\big|+\widetilde C_0\,k+\widetilde C_1h\bigg(\sum_{i=1}^{n_\mathscr S}{\|y_i-\mathcal P^\ell y_i\|}_2^2\bigg)^{1/2}
\end{align*}
for $h\in[0,1/\lambda)$ with $\widetilde C_0=L_g/(\lambda-L_f)$, which implies
\[
\sup_{y\in\overline\Omega}\big|v_h(y)-\widetilde v_{hk}^\ell(y)\big|\le \widehat C_0\,\frac{k}{h}+\widehat C_1\bigg(\sum_{i=1}^{n_\mathscr S}{\|y_i-\mathcal P^\ell y_i\|}_2^2\bigg)^{1/2}
\]
for $h\in[0,1/\lambda)$ with $\widehat C_i=\widetilde C_i/\lambda$, $i=0,1$. 
\end{proof}

\begin{remark}\label{invarrem}
\rm
Let us give sufficient conditions for \eqref{AssumptionPol}. First, we observe that for any $i\in\{1,\ldots,n_\mathscr S\}$ and $u\in\Uad$ we have
\[
y_i+hf(\mathcal P^\ell y_i,u)=y_i+hf(y_i,u)+h\big(f(\mathcal P^\ell y_i,u)-f(y_i,u)\big).
\]
To ensure \eqref{AssumptionPol} we replace \eqref{AssumpOmega} by the stronger assumption
\[
y+hf(y,u)\in\mathrm{int}\,\Omega\quad\text{for all }y\in\overline\Omega\text{ and for all }u\in\Uad,
\]
where $\mathrm{int}\,\Omega$ stands for the (open) interior of the set $\Omega$, we have $y_i+hf(y_i,u)\in\mathrm{int}\,\Omega$ for any $i\in\{1,\ldots,n_\mathscr S\}$. Moreover, Assumption~\ref{Assumption1}-1) implies that
\[
{\|f(\mathcal P^\ell y_i,u)-f(y_i,u)\|}_2\le L_f\,{\|\mathcal P^\ell y_i-y_i\|}_2
\]
holds. Consequently, if the mesh size $h$ or if $\|\mathcal P^\ell y_i-y_i\|_2$ are sufficiently small, the norm of  the vector $h(f(\mathcal P^\ell y_i,u)-f(y_i,u))$ can be made sufficiently small so that $y_i+hf(\mathcal P^\ell y_i,u)\in\overline\Omega$.\hfill$\Diamond$
\end{remark}

From Theorem~\ref{Th:ConvRate-vh} and Proposition~\ref{ROM-Theo} we derive the following a-priori error estimate.

\begin{theorem}
\label{APr-Theo2}
Assume that Assumption~{\rm\ref{Assumption1}}, \eqref{AssumpOmega} and \eqref{AssumptionPol} hold. Suppose that $f,g$ satisfy the semiconcavity conditions \eqref{SemiConc}. Let $v$ and $\widetilde v^\ell_{hk}$ be the solutions of \eqref{hjb} and \eqref{HJB-Est-3}, respectively. If $\lambda>\max\{L_f,L_g\}$, then there exists constants $c_0,c_1,c_2\ge0$ such that
\begin{equation}
\label{AEst1}
\sup_{y\in\overline\Omega}\big|v(y)-\widetilde v_{hk}^\ell(y)\big|\le c_0h+c_1\,\frac{k}{h}+c_2\bigg(\sum_{i=1}^{n_\mathscr S}{\|y_i-\mathcal P^\ell y_i\|}_2^2\bigg)^{1/2}
\end{equation}
for any $h\in[0,1/\lambda)$.
\end{theorem}

\begin{remark}
\label{Rem-APr-Theo-2}
\rm
The a-priori error estimate presented in Theorem \ref{APr-Theo2} is natural, because it combines the discretization error between $v$ and $v_{hk}$ (compare \eqref{hjb:est2}) with the POD approximation quality for the (finite many) vertices $\{y_i\}_{i=1}^{n_\mathscr S}$. In particular, if we determine the POD basis by solving
\[
\left\{
\begin{aligned}
&\min \sum_{i=1}^{n_\mathscr S}\Big\|y_i-\sum_{j=1}^\ell {\langle y_i,\psi_j\rangle}_2\,\psi_j\Big\|_2^2\\
&\hspace{1mm}\text{such that }\{\psi_i\}_{i=1}^\ell\subset\mathbb R^n\text{ and }{\langle\psi_i,\psi_j\rangle}_2=\delta_{ij},~1\le i,j\le \ell,
\end{aligned}
\right.
\]
we get the a-priori error estimate
\[
\sup_{y\in\overline\Omega}\big|v(y)-\widetilde v_{hk}^\ell(y)\big|\le c_0h+c_1\,\frac{k}{h}+c_2\bigg(\sum_{i=\ell+1}^{n_\mathscr S}\lambda_i\bigg)^{1/2}.
\]
However, the POD grid points $\{y_i^\ell\}_{i=}^{n_\mathscr S}$ are not well distributed in general, which is disadvantageous for the numerical realization.\hfill$\Diamond$
\end{remark}

%%%%%%%%%%%%%%%%%%%%%%%%%%%%%%%%%%%%%%%%%%%%%%%%%%%%%
\subsection{Second a-priori error estimate}
\label{Section4.2}
%%%%%%%%%%%%%%%%%%%%%%%%%%%%%%%%%%%%%%%%%%%%%%%%%%%%%

From a numerical point of view \eqref{hjb_hkPOD} is not appropriate, because in general the grid points $\{y_i^\ell\}_{i=1}^{n_\mathscr S}$ are not uniformly distributed in $\mathbb R^\ell$ and their distribution will strongly depend on $\Psi$.  Therefore, we define a second POD discretization of the HJB equations where we have an explicit  bound on the distance between the neighbouring nodes. Clearly in this case we will need an interpolation operator defined on the grid (tipically, this will be a piecewise linear interpolation operator).
 With \eqref{AssumpOmega} holding we assume that there exists a bounded polyhedron $\Omega^\ell\subset\mathbb R^\ell$ satisfying
\begin{equation}
\label{AssumpOmegaell}
\Psi^\top \mathrm M y\in\mathrm{int}\,\Omega^\ell\quad\text{for all }y\in\overline\Omega.
\end{equation}

\begin{remark}
\label{Condi1}
\rm
Let $y\in\overline\Omega$, $u\in\Uad$ be arbitrarily chosen and set $y^\ell=\Psi^\top\mathrm M y\in\mathbb R^\ell$. Then,
\begin{align*}
y^\ell+hf^\ell(y^\ell,u)&=\Psi^\top \mathrm My+h\Psi^\top\mathrm Mf(\Psi\Psi^\top\mathrm My,u)\\
&=\Psi^\top \mathrm M\big(y+hf(y,u)\big)+h\Psi^\top\mathrm M\big(f(\Psi\Psi^\top\mathrm M y,u)-f(y,u)\big).
\end{align*}
We infer from \eqref{AssumpOmega} that $z=y+hf(y,u)\in\overline\Omega$ holds. Hence, by \eqref{AssumpOmegaell}, we have $\Psi^\top \mathrm Mz\in\mathrm{int}\,\Omega^\ell$. Furthermore, we derive from
\[
\big\|h\Psi^\top\mathrm M\big(f(\Psi\Psi^\top\mathrm M y,u)-f(y,u)\big)\big\|_2\le h L_f{\|\Psi^\top\mathrm M\|}_2\,{\|\Psi\Psi^\top\mathrm M y-y\|}_2
\]
and $\mathrm{span}\,\{\psi_1,\ldots,\psi_n\}=\mathbb R^n$ that $y^\ell+hf^\ell(y^\ell,u)\in\overline\Omega^\ell$ holds for step size $h$ or $\|\Psi\Psi^\top\mathrm M y-y\|_2$ sufficiently small. If $\mathrm M=\mathrm{Id}\in\mathbb R^{n\times n}$ holds, we have $\|\Psi^\top\mathrm M\|_2=1$.\hfill$\Diamond$
\end{remark}

Let $\{\mathscr S_j^\ell\}_{j=1}^{\mathsf m_\mathscr S}$ be a family of simplices which defines a regular triangulation of the polyhedron $\Omega^\ell$ such that
\[
\overline\Omega^\ell=\bigcup_{j=1}^{\mathsf m_\mathscr S}\mathscr S_j^\ell\quad\text{and}\quad K=\max_{1\le j\le \mathsf m_\mathscr S}\big(\mathrm{diam}\,\mathscr S_j^\ell\big).
\]
Let $V^K$ be the space of piecewise affine functions from $\overline\Omega^\ell$ to $\mathbb R$ which are continuous in $\overline\Omega^\ell$ having constant gradients in the interior of any simplex $\mathscr S_j^\ell$ of the triangulation. Then, we introduce the following POD scheme for the HJB equations
\begin{equation}
\label{hjb_hkPOD2}
v_{hK}^\ell(\mathsf y_i^\ell)=\min_{u\in\Uad}\big\{(1-\lambda h)v_{hK}^\ell\big(\mathsf y_i^\ell+h f^\ell(\mathsf y^\ell_i,u)\big)+hg^\ell(\mathsf y^\ell_i,u)\big\}
\end{equation}
for any vertex $\mathsf y_i^\ell\in\overline\Omega^\ell$. Throughout this paper we assume that we have $\mathsf n_\mathscr S$ vertices $\mathsf y^\ell_1,\ldots,\mathsf y^\ell_{\mathsf n_\mathscr S}$. We set $\mathsf y_i=\Psi\mathsf y_i^\ell$ for $1\le i\le\mathsf n_\mathscr S$ and define
\[
\widetilde v_{hK}^\ell(y)=v_{hK}^\ell(\Psi^\top\mathrm My)\quad\text{for all }y\in\overline\Omega
\]
Recall that \eqref{AssumpOmegaell} ensures $\Psi^\top\mathrm My\in\mathrm{int}\,\Omega^\ell$ for any $y\in\overline\Omega$. Moreover, $\mathsf y_i=\Psi\mathsf y_i^\ell$ and $\Psi^\top\mathrm M\Psi=\mathrm{Id}\in\mathbb R^{\ell\times\ell}$ implies that
\[
\widetilde v_{hK}^\ell(\mathsf y_i)=v_{hK}^\ell(\Psi^\top\mathrm M\mathsf y_i)=v_{hK}^\ell(\mathsf y_i^\ell)\quad\text{for }1\le i\le \mathsf n_\mathscr S.
\]
Using \eqref{hjb_hkPOD2} we obtain
\begin{align*}
\widetilde v_{hK}^\ell(\mathsf y_i)&=v_{hK}^\ell(\mathsf y_i^\ell)=\min_{u\in\Uad}\big\{(1-\lambda h)v_{hK}^\ell\big(\mathsf y_i^\ell+h f^\ell(\mathsf y^\ell_i,u)\big)+hg^\ell(\mathsf y^\ell_i,u)\big\}\\
&=\min_{u\in\Uad}\big\{(1-\lambda h)v_{hK}^\ell\big(\Psi^\top\mathrm M(\Psi \mathsf y_i^\ell+h f(\Psi\mathsf y^\ell_i,u))\big)+hg(\Psi\mathsf y^\ell_i,u)\big\}\\
&=\min_{u\in\Uad}\big\{(1-\lambda h)v_{hK}^\ell\big(\Psi^\top\mathrm M(\mathsf y_i+h f(\mathsf y_i,u))\big)+hg(\mathsf y_i,u)\big\}\\
&=\min_{u\in\Uad}\big\{(1-\lambda h)\widetilde v_{hK}^\ell\big(\mathsf y_i+h f(\mathsf y_i,u)\big)+hg(\mathsf y_i,u)\big\}\quad\text{for }1\le i\le \mathsf n_\mathscr S.
\end{align*}
Thus, \eqref{hjb_hkPOD2} can be written as
\begin{equation}
\label{hjb_hkPOD2a}
\widetilde v_{hK}^\ell(\mathsf y_i)=\min_{u\in\Uad}\big\{(1-\lambda h)\widetilde v_{hK}^\ell\big(\mathsf y_i+h f(\mathsf y_i,u)\big)+hg(\mathsf y_i,u)\big\}
\end{equation}
for $1\le i\le\mathsf n_\mathscr S$. 

\begin{proposition}
\label{ROM-Theo2}
Assume that Assumption~{\rm\ref{Assumption1}}, \eqref{AssumpOmega} and \eqref{AssumpOmegaell} hold. Let $v_h$ and $\widetilde v^\ell_{hK}$ be the solutions of \eqref{hjb_h} and \eqref{hjb_hkPOD2a}. Let
\begin{equation}
\label{AssumptionPol2}
\Psi^\top\mathrm M\big(y+hf(\mathcal P^\ell y,u)\big)\in\overline\Omega^\ell\quad\text{for all }y\in\overline\Omega\text{ and for all }u\in\Uad.
\end{equation}
be satisfied. For $\lambda>L_f$ there exists constants $c_0,c_1$ such that
\[
\sup_{y\in\overline\Omega}\big|v_h(y)-\widetilde v_{hk}^\ell(y)\big|\le c_0\,{\|\Psi\|}_2\,\frac{K}{h}+c_1\sup_{y\in\overline\Omega}{\|y-\mathcal P^\ell y\|}_2\,\frac{1}{h}
\]
for any $h\in[0,1/\lambda)$.
\end{proposition}

\begin{proof}
Let $y\in\overline\Omega$ be chosen arbitrarily. We set $y^\ell=\Psi^\top\mathrm My$. By \eqref{AssumpOmegaell} we have $y^\ell\in\mathrm{in}\,\Omega^\ell$. Then, there are real coefficients $\mu_i^\ell=\mu_i^\ell(y^\ell)$, $1\le i\le\mathsf n_\mathscr S$, of the convex combination representation of $y^\ell$ satisfying
\[
y^\ell=\sum_{i=1}^{\mathsf n_\mathscr S}\mu^\ell_i\mathsf y_i^\ell,\quad 0\le\mu_i^\ell\le 1\quad\text{and}\quad\sum_{i=1}^{\mathsf n_\mathscr S}\mu_i^\ell=1.
\]
Since $v_{hK}^\ell$ is piecewise affine we have $v_{hK}(y^\ell)=\sum_{i=1}^{\mathsf n_\mathscr S}\mu_i^\ell v_{hK}^\ell(\mathsf y_i^\ell)$. Using $\mathsf y_i=\Psi\mathsf y_i^\ell$, we have
\begin{equation}
\label{est-00}
\begin{aligned}
\big|v_h(y)-\widetilde v_{hk}^\ell(y)\big|&\le\big|v_h(y)-v_h(\mathcal P^\ell y)\big|+\bigg|\sum_{i=1}^{\mathsf n_\mathscr S}\mu_i^\ell\big(v_h(\mathcal P^\ell y)-v_h(\mathsf y_i)\big)\bigg|\\
&\quad+\bigg|\sum_{i=1}^{\mathsf n_\mathscr S}\mu_i^\ell\big(v_h(\mathsf y_i)-v_{hk}^\ell(\mathsf y_i^\ell)\big)\bigg|.
\end{aligned}
\end{equation}
By \eqref{vh-Lipschitz} the first term on the right-hand side of \eqref{est-00} can be bounded as follows:
\begin{equation}
\label{est-A}
\big|v_h(y)-v_h(\mathcal P^\ell y)\big|\le \frac{L_g}{\lambda-L_f}\,{\|y-\mathcal P^\ell y\|}_2\text{ for all }h\in[0,1/\lambda).
\end{equation}
Furthermore, there exists an index $j$ with $y\in\overline{\mathscr S}_j\subset\overline\Omega^\ell$. Let $\mathscr I_j=\{i_1,\ldots,i_\mathsf k\}\subset\{1,\ldots,\mathsf n_\mathscr S\}$ denote the index subset such that $y^\ell\in\overline{\mathscr S}_j$ holds for $i\in\mathscr I_j$. Then, $\mu_i^\ell=0$ holds for all $i\not\in\mathscr I_j$. Moreover, $\sum_{i=1}^{\mathsf n_\mathscr S}\mu_i^\ell=\sum_{i\in\mathscr I_j}\mu_i^\ell=1$ and $\|y^\ell-\mathsf y_i^\ell\|_2\le K$ for any $i\in\mathscr I_j$. Recall that $\mathcal P^\ell y=\Psi\Psi^\top\mathrm M y=\Psi y^\ell$ holds. Again using \eqref{vh-Lipschitz} we find
\begin{equation}
\label{est-B}
\sum_{i=1}^{n_\mathscr S}\mu_i^\ell\big|v_h(\mathcal P^\ell y)-v_h(\mathsf y_i)\big|=\sum_{i\in\mathscr I_j}\mu_i^\ell\big|v_h(\Psi y^\ell)-v_h(\Psi\mathsf y_i^\ell)\big|\le \frac{L_g\,{\|\Psi\|}_2}{\lambda-L_f}\,K
\end{equation}
for $h\in[0,1/\lambda)$. Using \eqref{hjb_hkPOD2a} and \eqref{hjb_h} we have
\begin{align*}
&v_h(\mathsf y_i)-v_{hk}^\ell(\mathsf y_i^\ell)=v_h(\mathsf y_i)-\widetilde v_{hk}^\ell(\mathsf y_i)\\
&\le v_h(\mathsf y_i)-(1-\lambda h)\widetilde v_{hK}^\ell\big(\mathsf y_i+h f(\mathsf y_i,\bar u_{hK}^{\ell,i})\big)+hg(\mathsf y_i,\bar u_{hK}^{\ell,i})\big\}\\
&\le (1-\lambda h)\Big(v_h\big((\mathsf y_i+h f(\mathsf y_i,\bar u_{hK}^{\ell,i})\big)-\widetilde v_{hK}^\ell\big(\mathsf y_i+h f(\mathsf y_i,\bar u_{hK}^{\ell,i})\big)\Big)\\
&\le(1-\lambda h)\sup_{\widetilde y\in\overline\Omega} \big|v_h(\widetilde y)-\widetilde v_{hK}^\ell(\widetilde y)\big|,
\end{align*}
where $\bar u_{hK}^{\ell,i}\in\Uad$ is defined as
\[
\bar u_{hK}^{\ell,i}=\mathrm{arg}\hspace{-0.6mm}\min_{\hspace{-3mm}u\in\Uad}\big\{(1-\lambda h)\widetilde v_{hK}^\ell\big(\mathsf y_i+h f(\mathsf y_i,u)\big)+hg(\mathsf y_i,u)\big\}.
\]
By interchanging the role of $v_h$ and $v_{hK}^\ell$ we find
\begin{equation}
\label{est-C}
\big|v_h(\mathsf y_i)-v_{hk}^\ell(\mathsf y_i^\ell)\big|\le(1-\lambda h)\sup_{\widetilde y\in\overline\Omega} \big|v_h(\widetilde y)-\widetilde v_{hK}^\ell(\widetilde y)\big|.
\end{equation}
Inserting \eqref{est-A}, \eqref{est-B} and \eqref{est-C} into \eqref{est-00} we find
\begin{equation}
\label{est-D}
\sup_{y\in\overline\Omega}\big|v_h(y)-\widetilde v_{hk}^\ell(y)\big|\le\widetilde c_1\bigg(\frac{1}{h}\,\sup_{y\in\overline\Omega}{\|y-\mathcal P^\ell y\|}_2+{\|\Psi\|}_2\,\frac{K}{h}\bigg)
\end{equation}
with $\widetilde c_1=L_g/(\lambda(\lambda-L_f))$. 
\end{proof}

\begin{remark}
\label{remark4.6}
\rm
\begin{enumerate}
\item [1)] In Remark~\ref{Condi1} we have showed that \eqref{AssumptionPol2} can be ensured provided \eqref{AssumpOmegaell} holds and the step size $h$ or $\|y-\mathcal P^\ell y\|_2$ are sufficiently small.
\item [2)] We should balance
\[
K \sim \sup_{y\in\overline\Omega}{\|y-\mathcal P^\ell y\|}_2
\]
in order to get a rate $K/h$ for the convergence.\hfill$\Diamond$
\end{enumerate}
\end{remark}

Combining Theorem~\ref{Th:ConvRate-vh} and Proposition~\ref{ROM-Theo2} we obtain the following result.

\begin{theorem}
\label{APr-Theo}
Assume that Assumption~{\rm\ref{Assumption1}}, \eqref{AssumpOmega}, \eqref{AssumpOmegaell} and \eqref{AssumptionPol2} hold. Suppose that $f,g$ satisfy the semiconcavity conditions \eqref{SemiConc}. Let $v$ and $\widetilde v^\ell_{hK}$ be the solutions of \eqref{hjb} and \eqref{hjb_hkPOD2a}, respectively. If $\lambda>\max\{L_f,L_g\}$, then there exists constants $c_0,c_1,c_2\ge0$ such that
\begin{equation}
\label{AEst2}
\sup_{y\in\overline\Omega}\big|v(y)-\widetilde v_{hK}^\ell(y)\big|\le c_0h+c_1\,{\|\Psi\|}_2\,\frac{K}{h}+c_2\sup_{y\in\overline\Omega}{\|y-\mathcal P^\ell y\|}_2\,\frac{1}{h}
\end{equation}
for any $h\in[0,1/\lambda)$.
\end{theorem}

\begin{remark}
\rm
Let us comment on the differences between the a-priori error estimates \eqref{AEst1} and \eqref{AEst2}. First of all, both estimates involve the terms depending on $h$ and on $k/h$ or $K/h$. Note that $\|\Psi\|_2=1$ if we choose $\mathrm M=\mathrm{Id}$ in the computation of the POD basis. However, the POD approximation errors have different impacts. In \eqref{AEst1} there is no factor $1/h$. Moreover, in \eqref{AEst2} the term $\|y-\mathcal P^\ell y\|_2$ has to be small for all $y\in\overline\Omega$, whereas in \eqref{AEst1} this is need only for the vertices $y_1,\ldots,y_{n_\mathscr S}\in\overline\Omega$. In order to get convergence from \eqref{AEst2} one has to guarantee that $K=o(h)$ and $\sup_{y\in\overline\Omega}{\|y-\mathcal P^\ell y\|}_2=o(h)$ but, as we will see in our numerical examples, the method seems to be rather efficient also for larger $K$. \hfill$\Diamond$
\end{remark}

%%%%%%%%%%%%%%%%%%%%%%%%%%%%%%%%%%%%%%%%%%%%%%%%%%%%%
\section{Practical implementation of the algorithm}
\label{Section5}
\setcounter{section}{5}
\setcounter{equation}{0}
\setcounter{theorem}{0}
\setcounter{algorithm}{0}
\renewcommand{\theequation}{\arabic{section}.\arabic{equation}}
%%%%%%%%%%%%%%%%%%%%%%%%%%%%%%%%%%%%%%%%%%%%%%%%%%%%%

In this section we present an algorithm for the HJB equation based on the POD a-priori analysis presented in Section\ref{Section4.2}. Estimate in Theorem~\ref{APr-Theo} suggests the following steps.

\subsubsection*{(1)~Time discretization}

First the infinite time horizon $[0,\infty)$ has to be replaced by a finite one. Thus, we choose $\te\gg 0$ sufficiently large and define a (possibly non equidistant) grid in $[0,\te]$ by $0=t_1<t_2<\ldots<t_{n_T}=\te$.

\subsubsection*{(2)~Snapshots computation}

Let us suppose that the set of admissible controls $\mathbb U_{ad}$ is given by choosing a discrete set $\Uad=\{u_1,\ldots,u_p\}\subset U=\mathbb R$. To solve \eqref{Eq:ContrDynamic} we apply the implicit Euler method on the time grid $\{t_j\}_{j=1}^{n_T}$. By $y_j^\nu\approx y^\nu(t_j)\in\mathbb R^n$, $1\le j\le n_T$ and $1\le\nu\le p$, we denote the computed implicit Euler approximation of the solution to \eqref{Eq:ContrDynamic} at time instance $t_j$ for the controls $u^\nu(t)=u_\nu$ for all $t\in[0,\te]$ and $1\le\nu\le p$.

\subsubsection*{(3)~Rank $\boldsymbol\ell$ of the POD basis}

In \eqref{Pelldisc} we choose $\mathrm M$ as the identity matrix and $\alpha_j^{n_T}=1$ for $j=1,\ldots,n_T$. The rank $\ell$ of the POD basis $\{\psi_i^{n_T}\}_{i=1}^\ell$ is chosen such that we can show the decay of the error when we increase $\ell$. In this paper we choose $\ell\in\{2,3,4\}$.

\subsubsection*{(4)~Reduction of the polyhedron $\boldsymbol\Omega$}

We first define all the POD grid points $y_i^\ell=\Psi^\top  y_i\in\mathbb R^\ell$ and choose the hypercube $\Omega^\ell=[a_1,b_1]\times\dots\times[a_\ell,b_\ell]\subset\mathbb R^\ell$, where we set
\[
a_j=\min\big\{(y_1^\ell)_j,\ldots,(y_{\mathsf m_y}^\ell)_j\big\},\quad b_j=\max\big\{(y_1^\ell)_j,\ldots,(y_{\mathsf m_y}^\ell)_j\big\}
\]
for $1\le j\le\ell$ and $(y_i^\ell)_j$ stands for the $j$-th component of the vector $y_i^\ell\in\mathbb R^\ell$. It follows that $y_i^\ell\in\Omega^\ell$ for $1\le i\le\mathsf m_y$. Then, when the domain $\Omega^\ell$ is obtained, we build an equidistant grid with step size computed as explained in Remark \ref{remark4.6}. 
We note that $\Omega^\ell$ should be large enough in order to contain all possible trajectories, this is also the reason we compute several snapshots in order to have a sufficiently accurate overview of the problem.

\subsubsection*{(5)~Computation of the value function $\boldsymbol{v_{hK}^\ell}$}

The piecewise linear value function $v_{hK}^\ell$ is determined on the vertices $\mathsf y_i$ for $1\le i\le \mathsf n_\mathscr S$ of the domain $\Omega^\ell$. Since the reduced-order approach yields a small $\ell<10$, we are able to perform a standard fixed point iteration method, e.g. the value iteration method. We refer the reader also to the faster algorithm introduced in \cite{AFK15} and the references therein.

\subsubsection*{(6)~Feedback law and closed-loop control}

We compute the value function $\widetilde v_{hK}^\ell(y)=v_{hK}(\Psi^\top y)$ satisfying \eqref{hjb_hkPOD} at each grid point $y=y_i$ for $1\le i\le \mathsf m_y$. At any grid point $y_i$ we store the associated optimal control $\bar u_{hK}^{\ell,i}\in\Uad$ solving
\begin{equation*}
u_{hK}^{\ell,i}:=\mathrm{arg}\hspace{-0.6mm}\min_{\hspace{-3mm}u\in\Uad}\big\{(1-\lambda h)\widetilde v_{hK}^\ell(y_i+h f(\mathcal P^\ell y_i,u))+hg(\mathcal P^\ell y_i,u)\big\}.
\end{equation*}
Then, the (suboptimal) feedback operator $\Phi^\ell:\overline\Omega\to U_{ad}$ is defined as
\[
\Phi^\ell(y)=\sum_{i=1}^{\mathsf m_y}\mu_i \bar u_{hK}^{\ell,i}\quad\text{for }y\in\overline\Omega,
\]
where the coefficients $\{\mu_i\}_{i=1}^{\mathsf m_y}$ are given by the convex combination
\begin{equation*}
%\label{Eq:yconv}
y=\sum_{i=1}^{\mathsf m_y}\mu_iy_i,\quad0\le\mu_i\le1,\quad\sum_{i=1}^{\mathsf m_y}\mu_i=1.
\end{equation*}
Now the closed-loop system for \eqref{Eq:ContrDynamic} is
\begin{equation}
\label{Eq:ClosedLoop}
\dot y(t)=f\big(y(t),\Phi^\ell(y(t))\big)\in\mathbb R^n \text{ for } t>0\quad y(0)=y_\circ\in\mathbb R^n.
\end{equation}
Equation \eqref{Eq:ClosedLoop} is solved by a semi implicit Euler scheme, where the second argument $\Phi^\ell(y(t))$ is evaluated at the previous time step. We note that every time step $t_i$ we plug the suboptimal control $u_{h,K}^{\ell,i}$ into \eqref{Eq:ClosedLoop} and then project into the POD space in order to have the next initial condition. The algorithm is summarized below.

\begin{algorithm}[H]
\caption{HJB-POD feedback control}
\label{Alg:hjbpod}
\begin{algorithmic}[1]
\REQUIRE distance $K$, step size $h$, final time $\te$, time grid $\{t_j\}_{j=1}^{n_T}$, discrete control set $\Uad=\{u_1,\ldots,u_p\}\subset\mathbb R$;
\STATE Set $Y=[]$ and $\Delta=K$;
\FOR{$\nu=1,\ldots, p$}
\STATE Compute approximation $\{y_j^\nu\}_{j=1}^{n_T}$ for the solution to \eqref{Eq:ContrDynamic} with $u\equiv u_\nu$;
\STATE Set $Y=[Y\,|\,y_1^\nu,\ldots,y_{n_T}^\nu]\in\mathbb R^{n\times(\nu n_T)}$;
\ENDFOR
\STATE Set $\mathsf m_y=p n_T$ and $Y=[y_1,\ldots,y_{\mathsf m_y}]\in\mathbb R^{n\times\mathsf m_y}$;
\STATE Determine a POD basis of rank $\ell$ by solving \eqref{Pelldisc};
\STATE Compute $\Omega^\ell$ and the reduced value function $\widetilde v_{hK}^\ell$;
\STATE Compute the suboptimal control and the optimal trajectory;
\end{algorithmic}
\end{algorithm}

%%%%%%%%%%%%%%%%%%%%%%%%%%%%%%%%%%%%%%%%%%%%%%%%%%%%%
\section{Numerical Tests}
\label{Section6}
\setcounter{section}{6}
\setcounter{equation}{0}
\setcounter{theorem}{0}
\setcounter{algorithm}{0}
\renewcommand{\theequation}{\arabic{section}.\arabic{equation}}
%%%%%%%%%%%%%%%%%%%%%%%%%%%%%%%%%%%%%%%%%%%%%%%%%%%%%

In this section we present our numerical tests. First let us describe the optimal control problem in detail. The governing equation is given by 
\begin{equation}
\label{state:eq}
\begin{aligned}
w_t-\varepsilon w_{xx}+\gamma w_x +\mu (w-w^3)&=b u&&\text{in } \omega\times(0,\te),\\
w(\cdot,0)&=w_\circ&&\text{in } \omega,\\
w(\cdot,t)&=0&&\text{in } \partial\omega\times(0,\te),
\end{aligned}
\end{equation}
where $\omega=(a,b)\subset\mathbb R$ is an open interval, $w:\omega\times[0, \te]\rightarrow \R$ denotes the state, and the parameters $\varepsilon$, $\gamma$ and $\mu$ are real positive constants. The controls are elements of the closed, convex, bounded set $\UAD=\{u\in L^2(0,\te;\mathbb R)\,|\,u(t)\in\Uad\text{ for }t\ge0\}$ with $\Uad=\{u\in\R\,|\,u_a\leq u\leq u_b\}$ with given $u_a, u_b\in\R$. Later, we will consider $\Uad$ as a discrete set in the approximation of the HJB equation. The initial value and the shape function are denoted, respectively, by $w_\circ$ and $b$. Note that we deal with zero Dirichlet boundary conditions. Equation \eqref{state:eq} includes, e.g., the linear heat equation ($\mu=0$, $\gamma=0$), linear advection diffusion equation ($\mu=0$) and a semi-linear parabolic problem with a reaction term ($\mu\neq0$). As explained in the previous section we need to choose $\te$ big enough to have an accurate approximation of the infinite horizon problem.

The cost functional we want to minimize is given by
\begin{equation}
\label{CostEx6.1}
\widehat J(u;w_\circ)=\int_0^\te \left(\|w(\cdot,t;u)-\bar w\|^2_{L^2(\omega)}+\alpha\,|u(t)|^2\right)e^{-\lambda t}\,\mathrm dt,
\end{equation}
where $w(\cdot,t;u)$ is the solution to \eqref{state:eq} at time $t$, $\bar w$ is the desired state, $\alpha\in\R^+$ holds and $\lambda>0$ is the discount factor. The optimal control problem can be formulated as
\begin{equation}\label{ocp}
\min\widehat J(u;w_\circ)\mbox{ such that } u\in\UAD.
\end{equation}
Existence and uniqueness results for \eqref{ocp} can be found in \cite{Lio71} for finite time horizons. We spatially discretize the state equation \eqref{state:eq} by the standard finite difference method. This approximation leads to the following semi-discrete system of ordinary differential equations:
\begin{equation}
\label{state:eqdis}
\begin{aligned}
\mathrm My _t-\varepsilon \mathrm Ay+\gamma \mathrm Hy+\mu\mathrm F(y)=\mathrm Bu&\quad&\mbox{ in } (0,\te],\\
y(0)=y_\circ,
\end{aligned}
\end{equation}
where $y:[0,\te]\to\mathbb R^n$ is an approximation for the solution $w(\cdot\,,t)$ to \eqref{state:eq} at $n$ spatial grid points, $\mathrm A$, $\mathrm H \in\R^{n\times n}$, $\mathrm B$, $y_\circ \in\R^n$ and $\mathrm F:\mathbb R^n\to\mathbb R^n$ is given by
\[
\mathrm F(y)=\left(
\begin{array}{c}
y_1-y_1^3\\
\vdots\\
y_n-y_n^3
\end{array}
\right)\quad\text{for }y=(y_1,\ldots,y_n)\in\mathbb R^n.
\]
In general,  the dimension $n$ of the dynamical system \eqref{state:eqdis} is rather large (i.e., $n>10$) so that we can not solve the HJB equations numerically. Therefore, we apply the POD method in order to reduce the dimension of optimal control problem and solve it by HJB equations. This problem fits into our problem \eqref{Phat} and therefore we can apply Algorithm~\ref{Alg:hjbpod} to solve our optimal control problem.

Next subsections will present our numerical tests, in particular we draw our attention on the estimate presented in Theorem \ref{APr-Theo}. In order to check the quality of the computed suboptimal control $u^\ell$ we will plug it into the full model $y(u^\ell)$ and into the surrogate model $y^\ell(u^\ell).$ Moreover we evaluate the cost functional and compute the error with respect the true solution where is known.

\subsection{Test 1: Advection-diffusion equation}

Our first test concerns the linear advection-diffusion equation, we set in \eqref{state:eq}: $\te=3$, $\varepsilon=10^{-1}$, $\gamma=1$, $\mu=0$, $\omega=(0,2)$ and $w_\circ(x)=0.5\sin(\pi x)$. The shape function $b$ is the characteristic function over the subset $(0.5,1)\subset\omega$. In \eqref{CostEx6.1} we choose $\lambda=1$, and $\bar w=0$. To compute the POD basis we determine solutions to the state equation for controls in the set $U_{\mbox{snap}}=\{-2.2,1.1,0\}$. In \eqref{hjb_hkPOD2} we consider $K\in\{0.1,0.05\}$, $h=0.1K$ and the optimal trajectory is obtained with a time stepsize of $0.05$ for the implicit Euler method. The set $\Uad$ of admissible controls is, then, given by 23 controls equally distributed from -2.2 to 0. 

In the left plot of Figure~\ref{fig1:hjb} we show the solution of the uncontrolled equation \eqref{state:eq}, i.e. for $u\equiv0$.
\begin{figure}[htbp]
\centering
\includegraphics[scale=0.21]{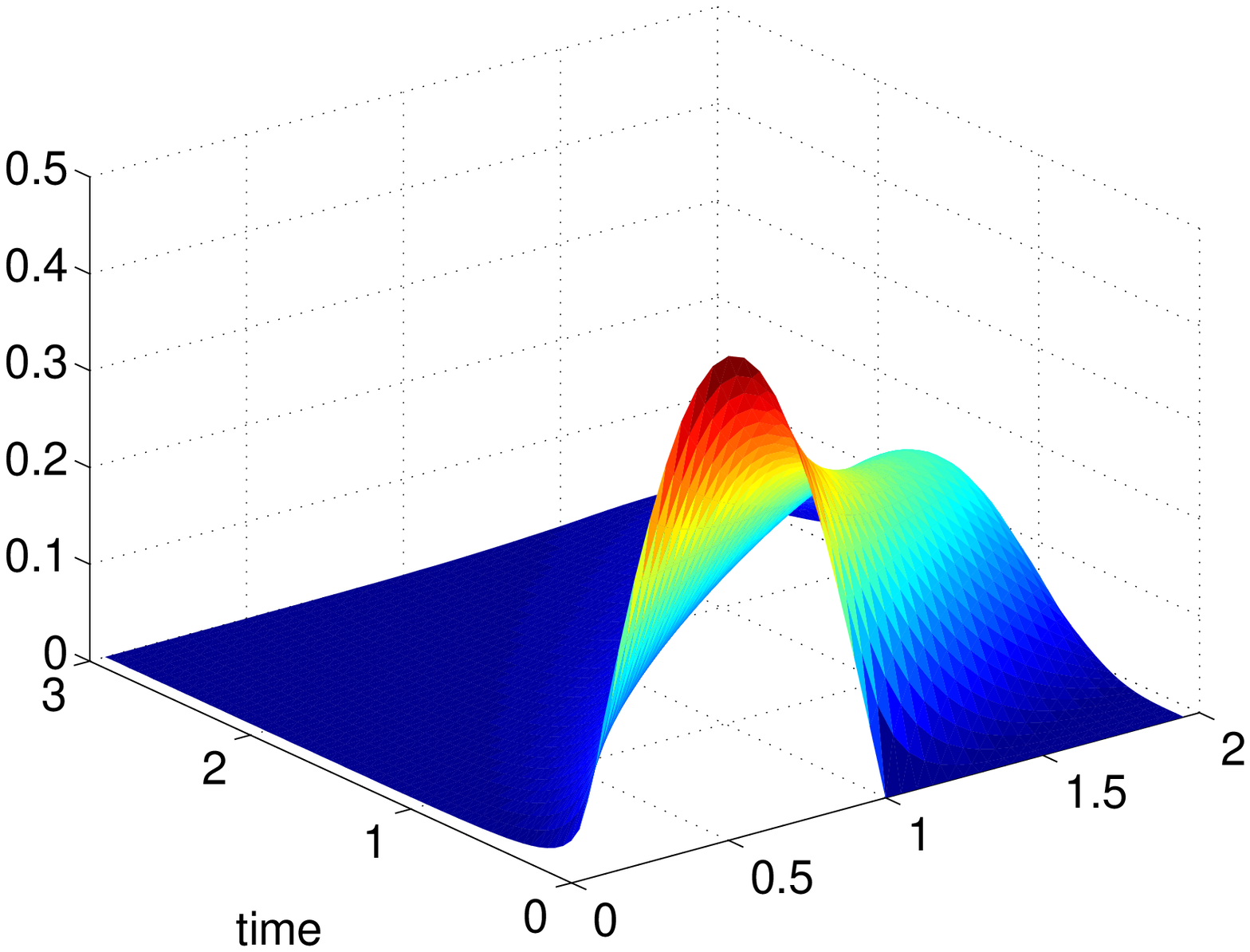}
\includegraphics[scale=0.21]{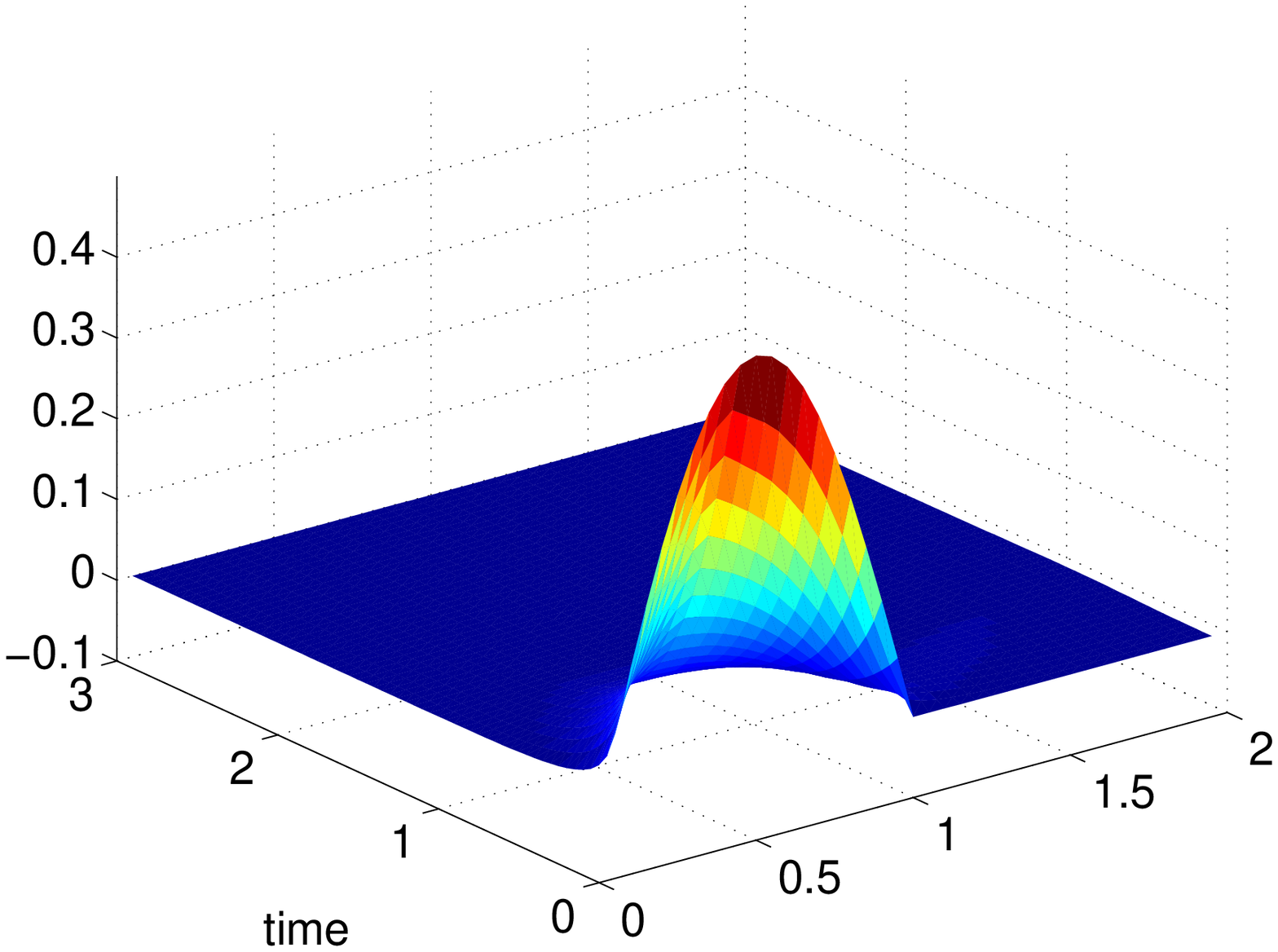}
\vspace{-0.2cm}
\includegraphics[scale=0.21]{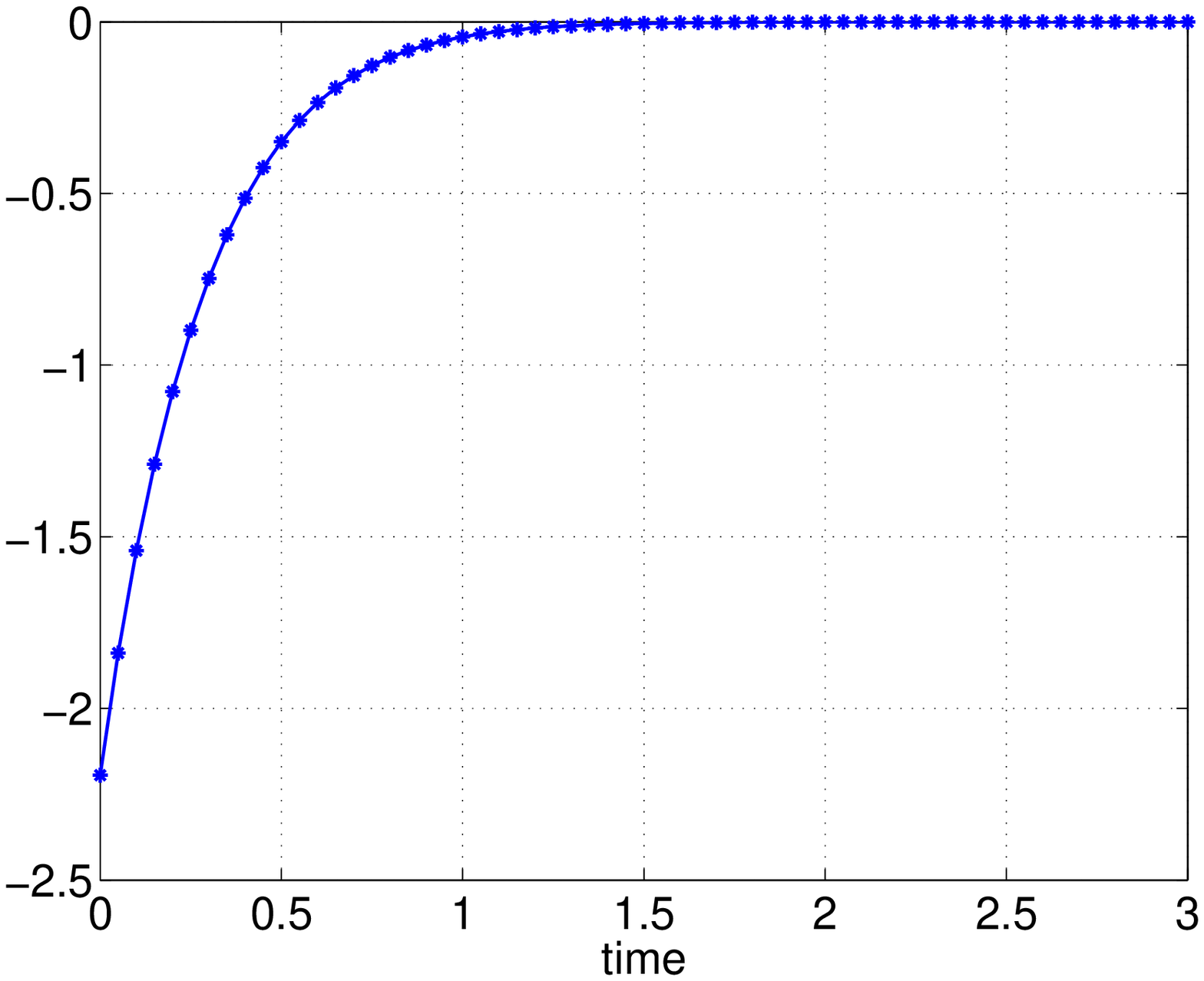}
\caption{Test 1: Uncontrolled solution (left), LQR optimal solution (middle), LQR optimal control (right). }
\label{fig1:hjb}
\end{figure}
Since our problem is linear-quadratic, the solution of the HJB equation can be computed by solving the well-known Riccati's equation (for LQR approach see \cite{CZ95}). Then, the optimal LQR state is presented in the middle of Figure~\ref{fig1:hjb}, whereas the optimal LQR control is plotted on the right of Figure~\ref{fig1:hjb}. Then, we show the controlled solution computed by means of Algorithm \ref{Alg:hjbpod} on the left of Figure~\ref{fig1:hjbsol}.
\begin{figure}[htbp]
\centering
\includegraphics[scale=0.215]{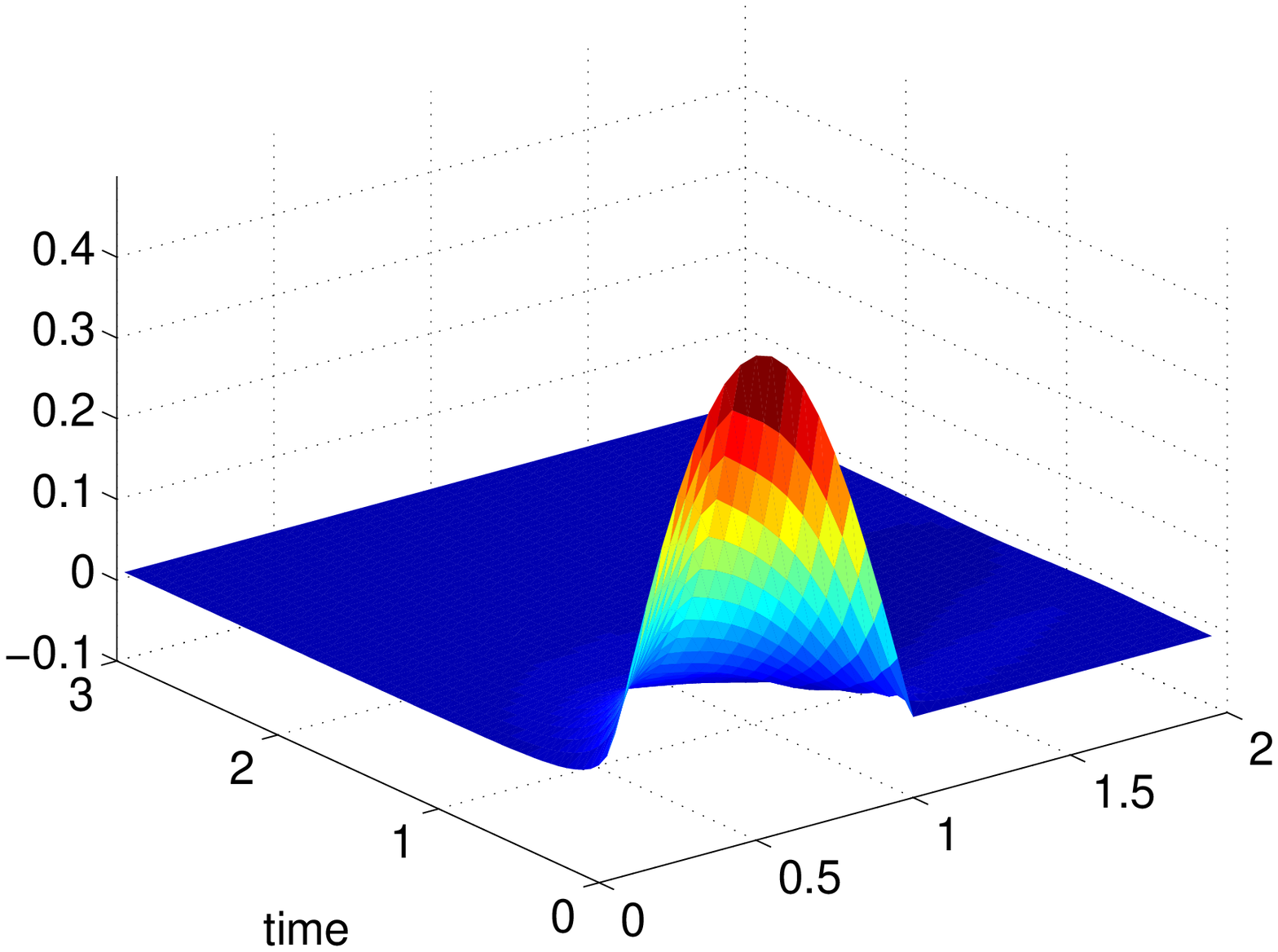}
\includegraphics[scale=0.215]{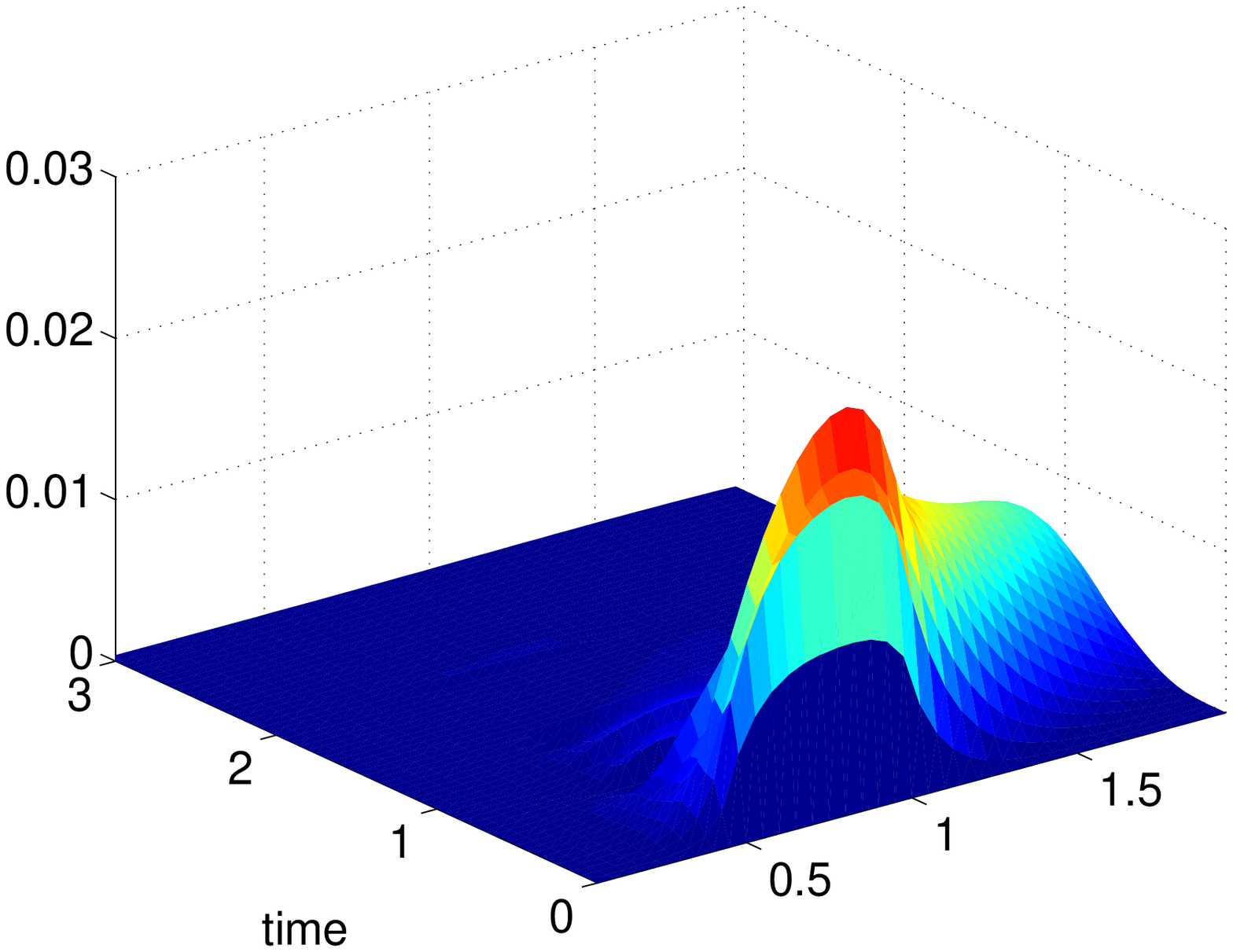}
\includegraphics[scale=0.215]{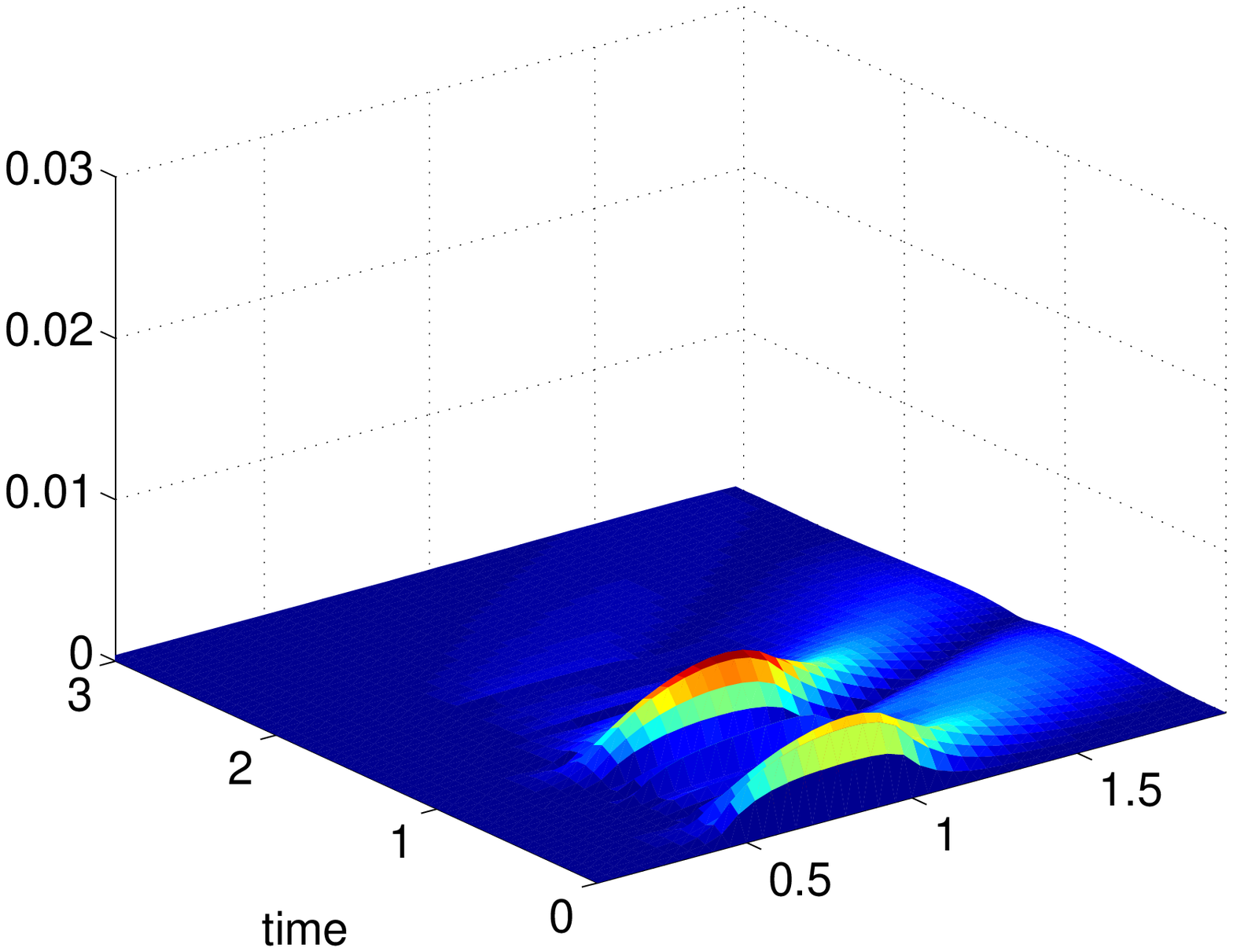}\\
\includegraphics[scale=0.215]{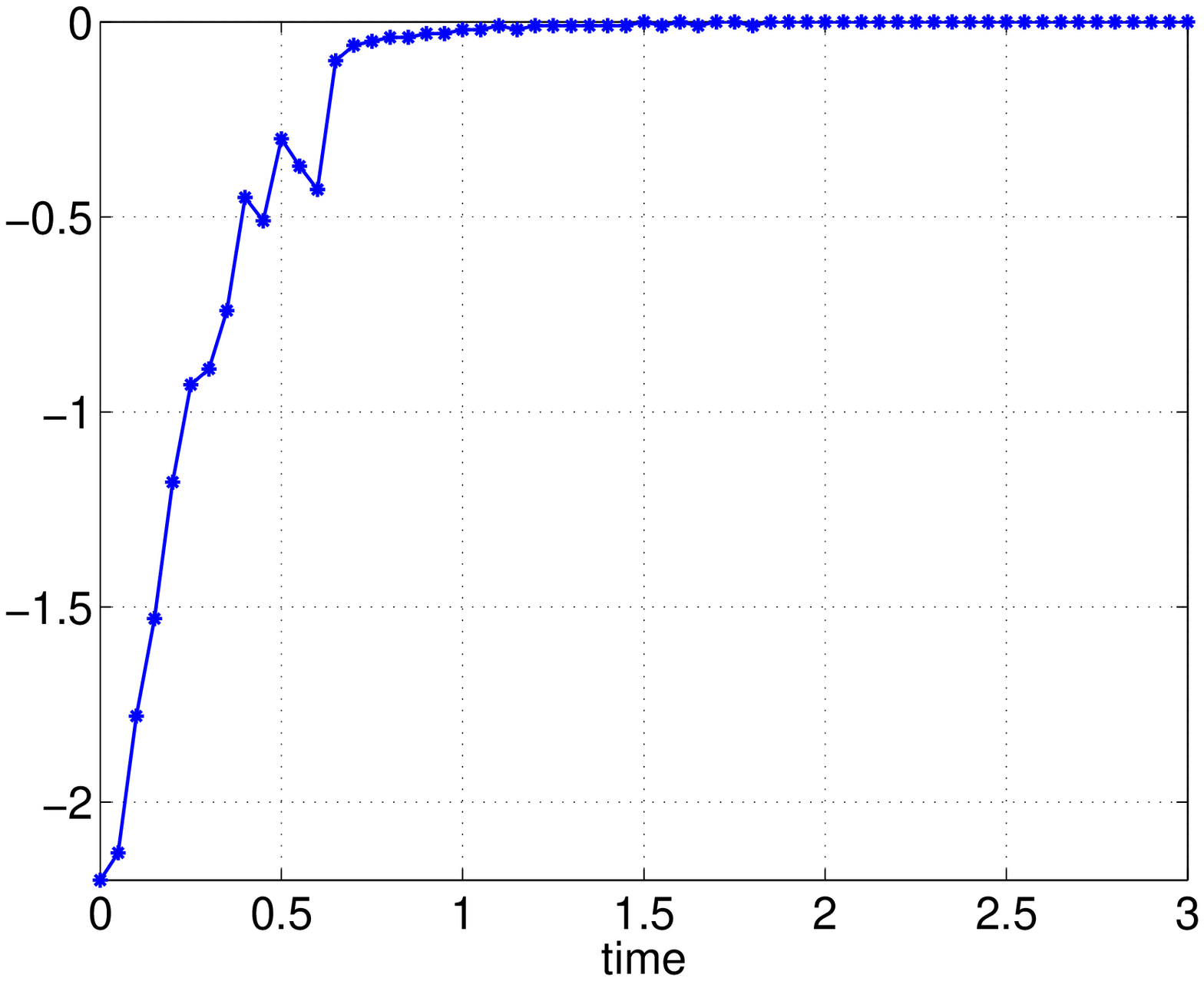}
\includegraphics[scale=0.215]{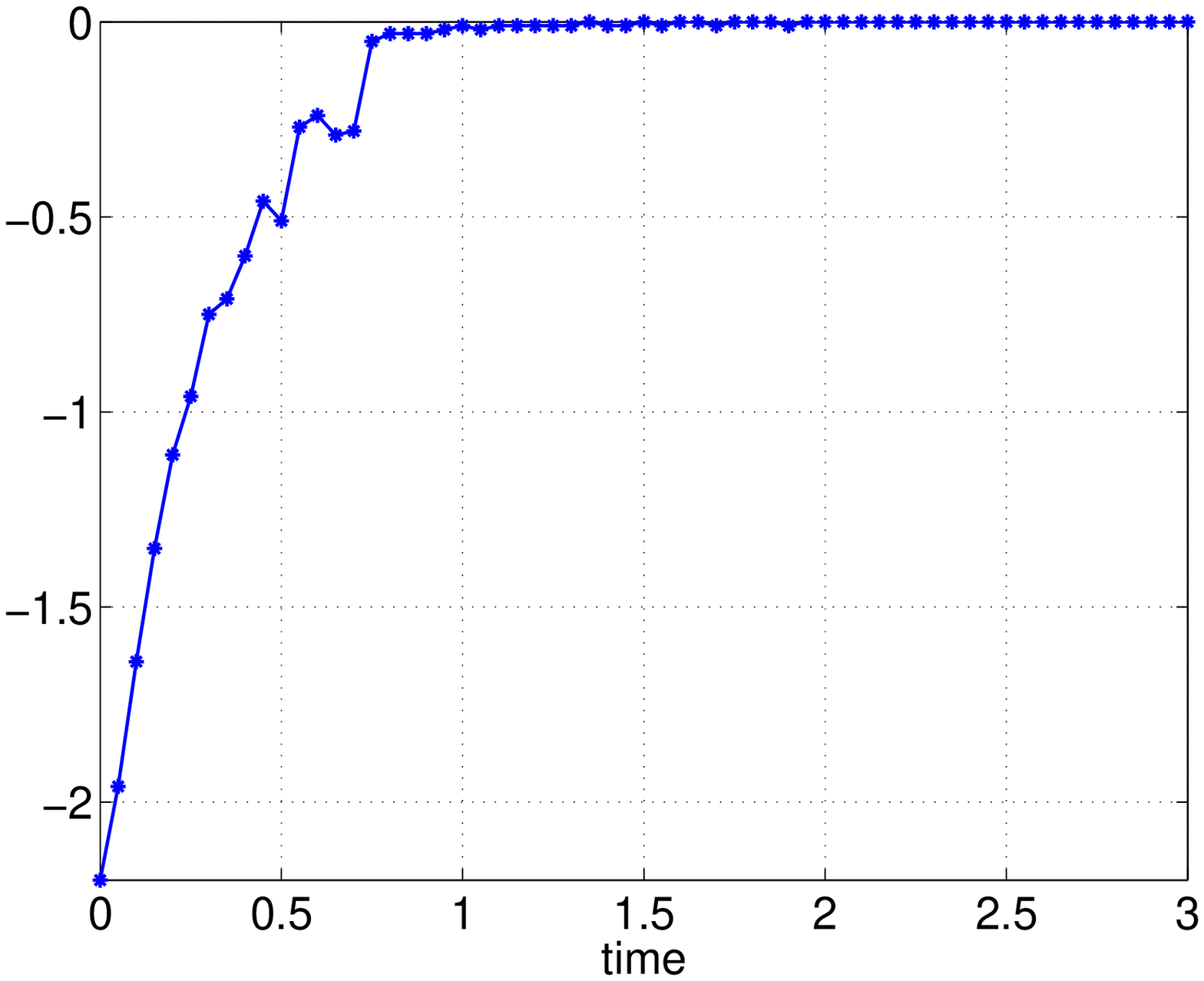}
\includegraphics[scale=0.215]{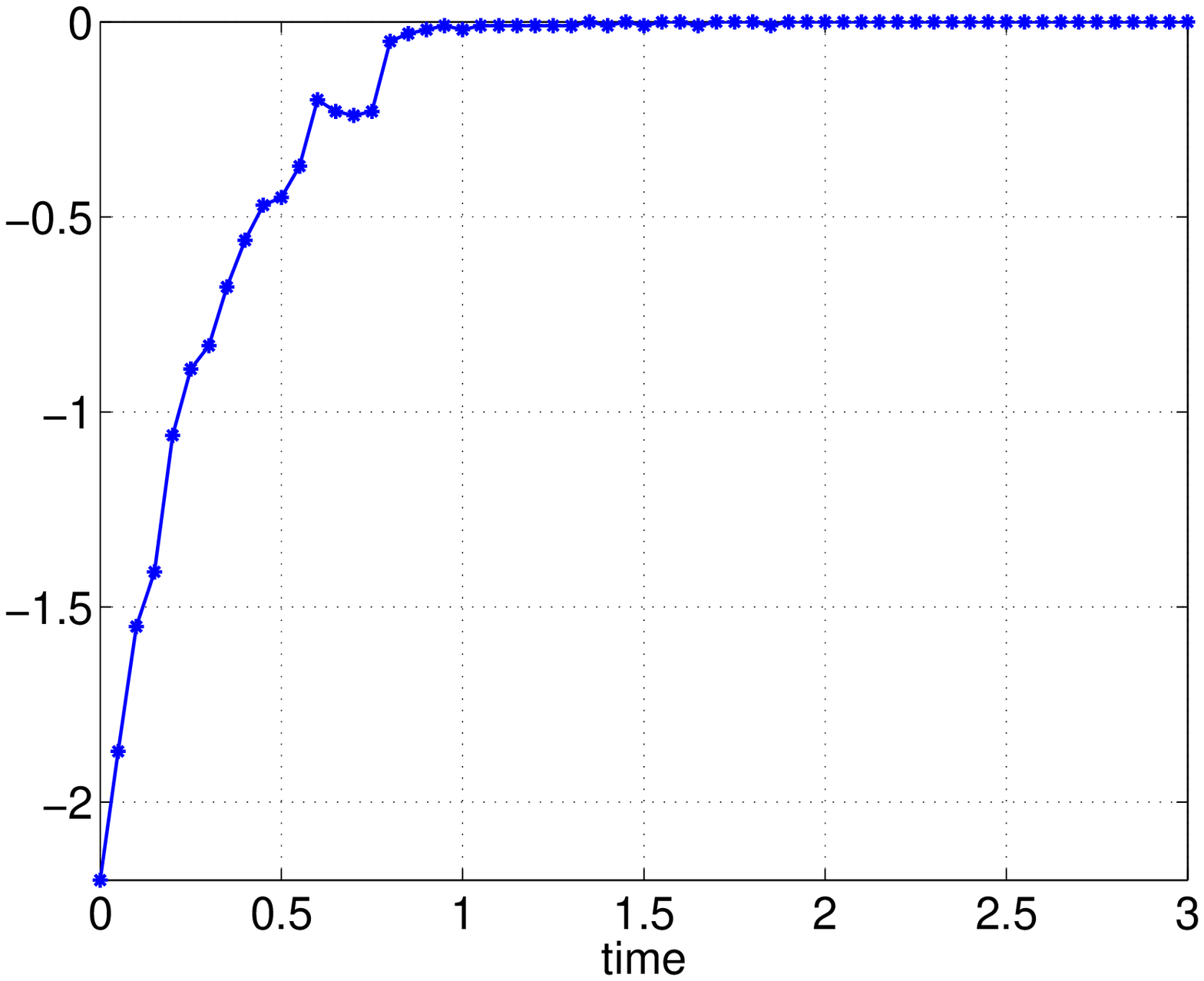}
\caption{Test 1: Optimal HJB states computed with Algorithm \ref{Alg:hjbpod} with $\ell=4$ POD basis functions (top-left), difference between optimal solution with 4 POD basis and 2 POD basis (top-middle), difference between optimal solution with 4 POD basis and 3 POD basis (top-right),  optimal HJB controls with $\ell=2,3,4$ (bottom). }
\label{fig1:hjbsol}
\end{figure}
Since it is hard to visualize differences from the optimal solutions we plot the difference between the optimal solution obtained with 4 POD basis functions and 2 POD basis in the middle of Figure \ref{fig1:hjbsol} and with 3 POD basis on the right side. Nevertheless, one can even have a look at the error analysis in Figure~\ref{fig1:hjberr}. 
\begin{figure}[htbp]
\centering
\includegraphics[scale=0.215]{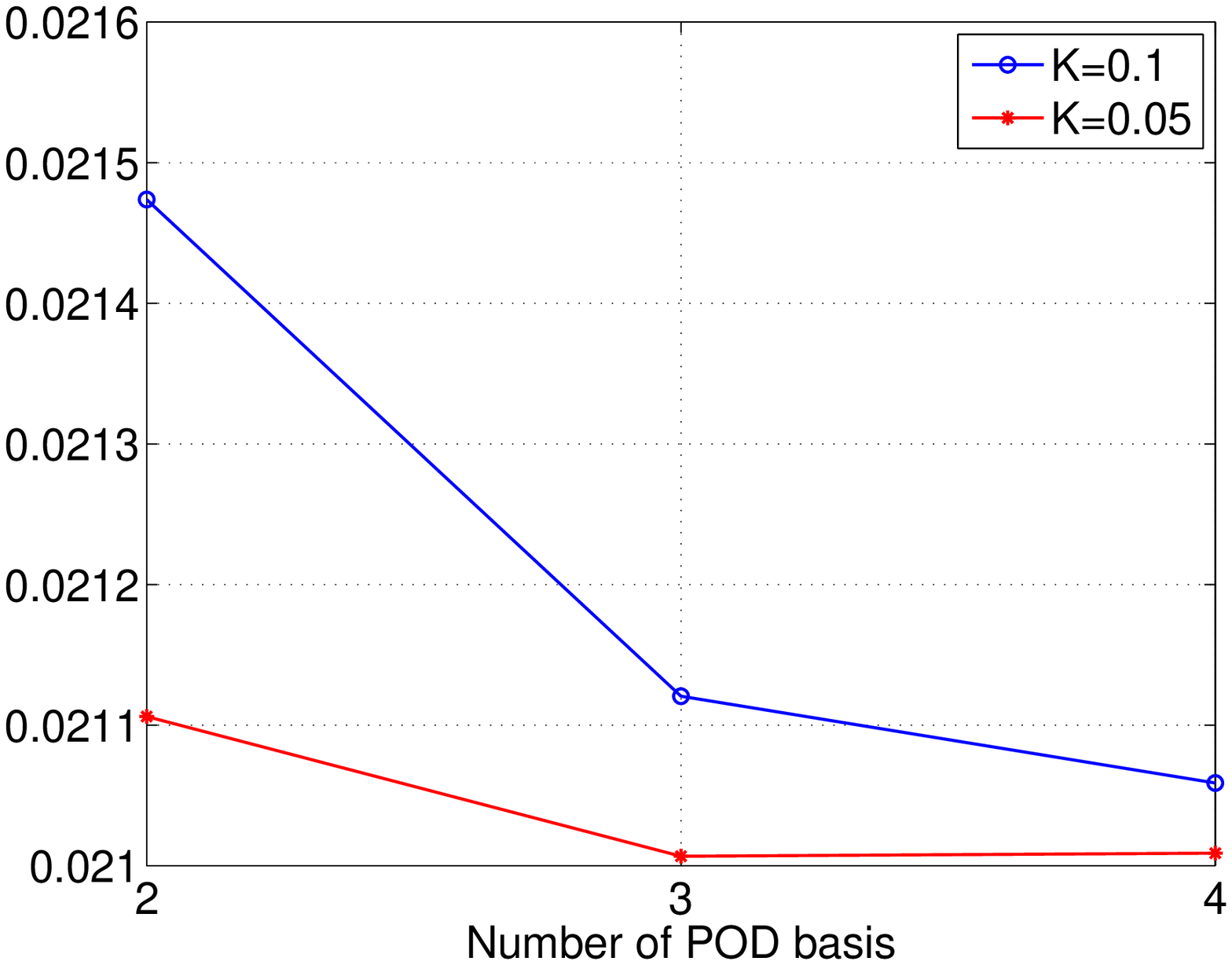}\hfill
\includegraphics[scale=0.215]{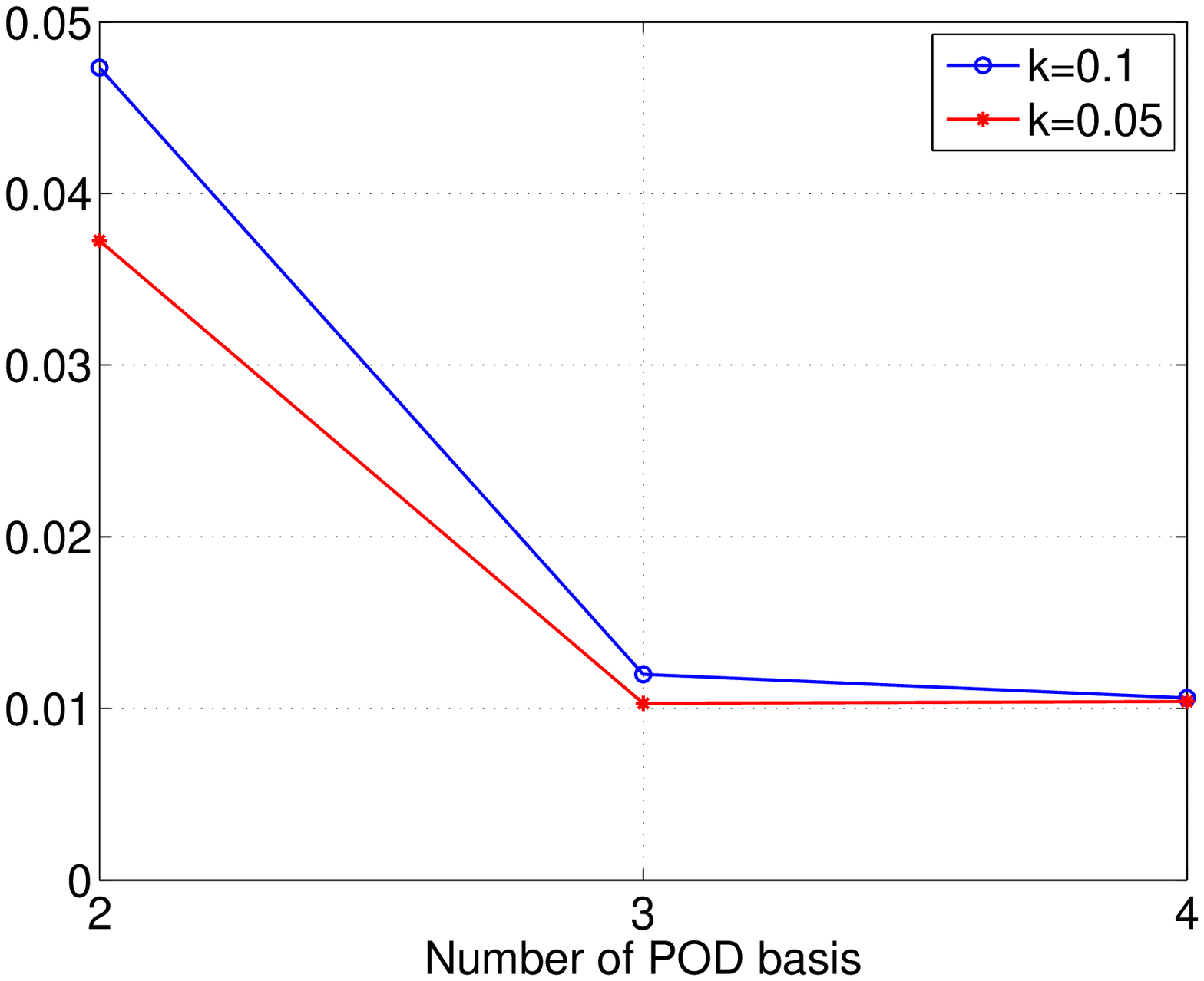}\hfill
\includegraphics[scale=0.215]{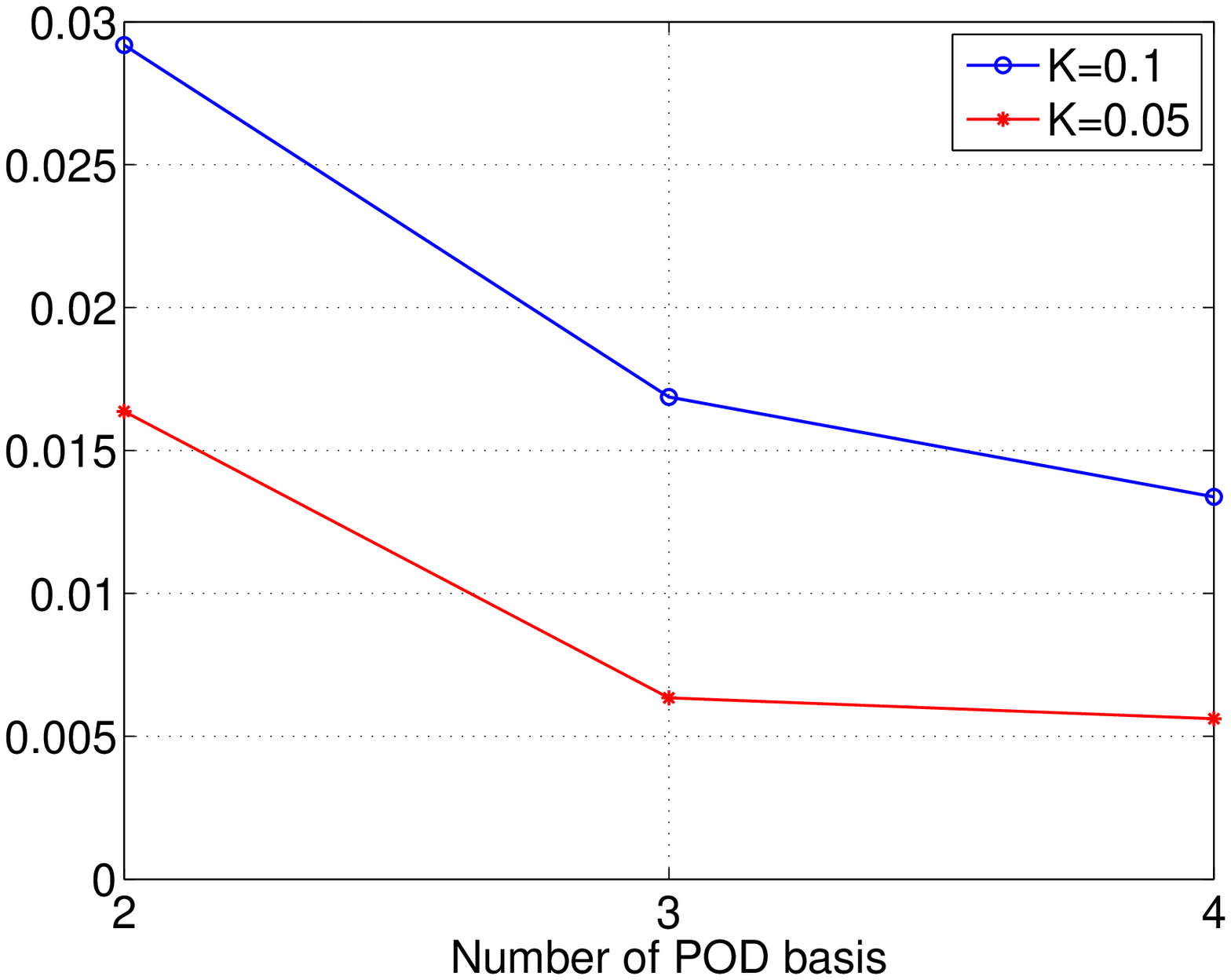}\hfill
\caption{Test 1: Evaluation of the cost functional (left), $L^2-$error for $y(u^\ell)$ and $y^\ell(u^\ell)$ (middle) and $L^2$-error between LQR solution and $y(u^\ell)$. The blue line refers to the approximation with $K=0.1$  in HJB, whereas the red one to $K=0.05$. }
\label{fig1:hjberr}
\end{figure}
In order to analyze our numerical approximation we consider the evaluation of the cost functional $\widehat J(u^\ell;y_\circ)$, the distance between $y(u^\ell)$ and $y^\ell(u^\ell)$ and the error between the truth solution and the suboptimal $y(u^\ell)$.  The error analysis is shown in Figure~\ref{fig1:hjberr}. On the left we show the decay of the cost functional when we increase the number of POD basis functions.  In the middle we compute the $L^2-$error between the optimal reduced solution $y^\ell(u^\ell)$ and the suboptimal solution $y(u^\ell)$. Even in this case the error decays when $\ell$ increases and $K$ decreases. This error measures the quality of the surrogate model, since we want to check whether the suboptimal control fits into the non-reduced problem. Finally, on the right, we compute the error between the optimal solution and the suboptimal $y(u^\ell)$. As expected, increasing the number of basis function and decreasing the step size $K$ (remember $h$ and $K$ are linked) for the approximation of the value function the optimal solution is improved.

The decay of the singular values is presented in Figure~\ref{fig1:svd}. As we can see they do not decay really fast with respect to the right plot which refers to the next example where the convection term is not dominated.\\
Finally, we want to give an idea of the term $\sup_{y\in\overline\Omega}{\|y-\mathcal P^\ell y\|}_2/h$ in the error estimate \eqref{AEst2}. It is clear we do not know $\Omega$, but we chose several randomly control sequences in the set of admissible controls. in order to have an approximation of the set. Now, we can compute the aforementioned error term. The decay is shown on the right of Figure \ref{fig1:svd}.
\begin{figure}[htbp]
\centering
\includegraphics[scale=0.215]{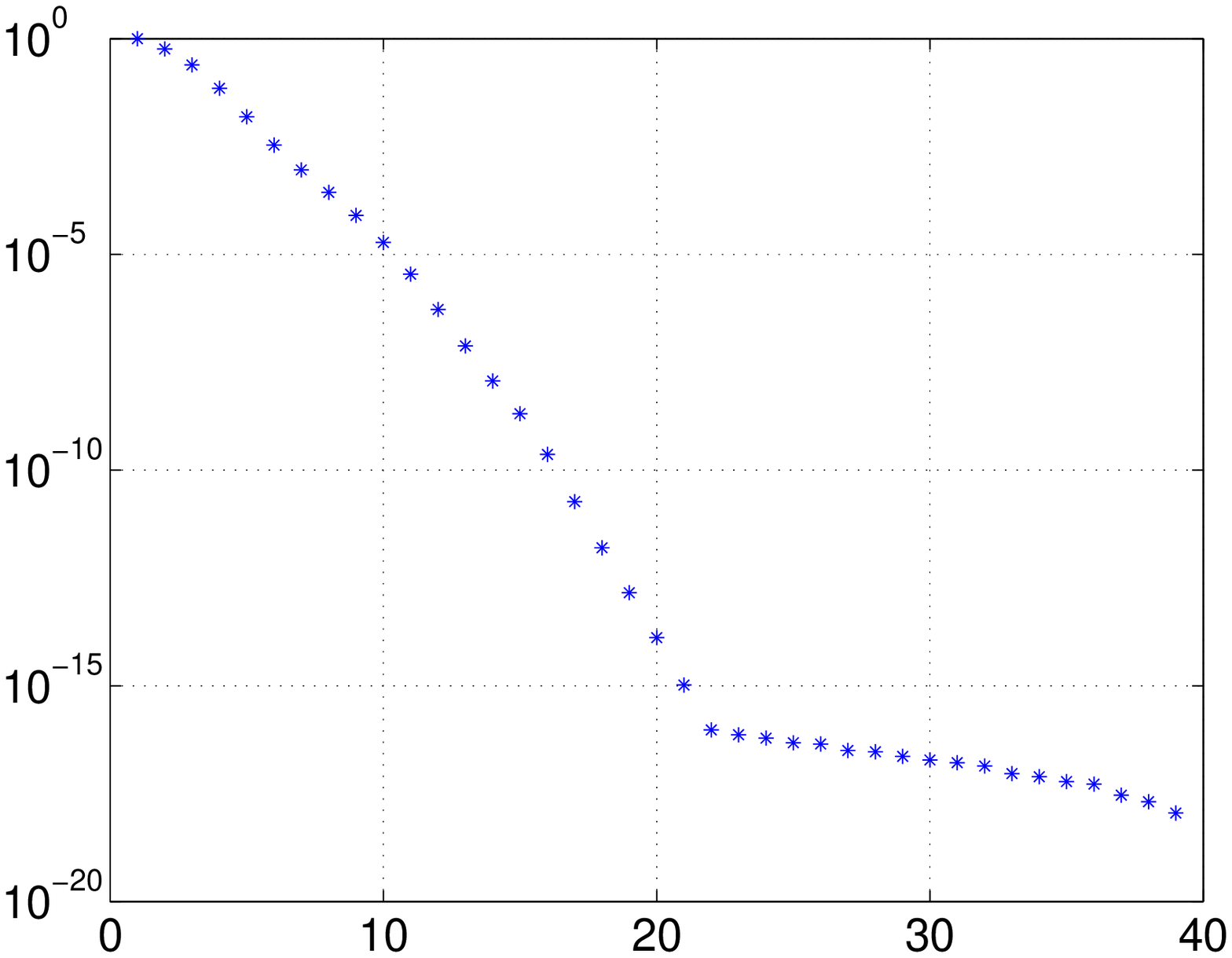}\hfill
\includegraphics[scale=0.215]{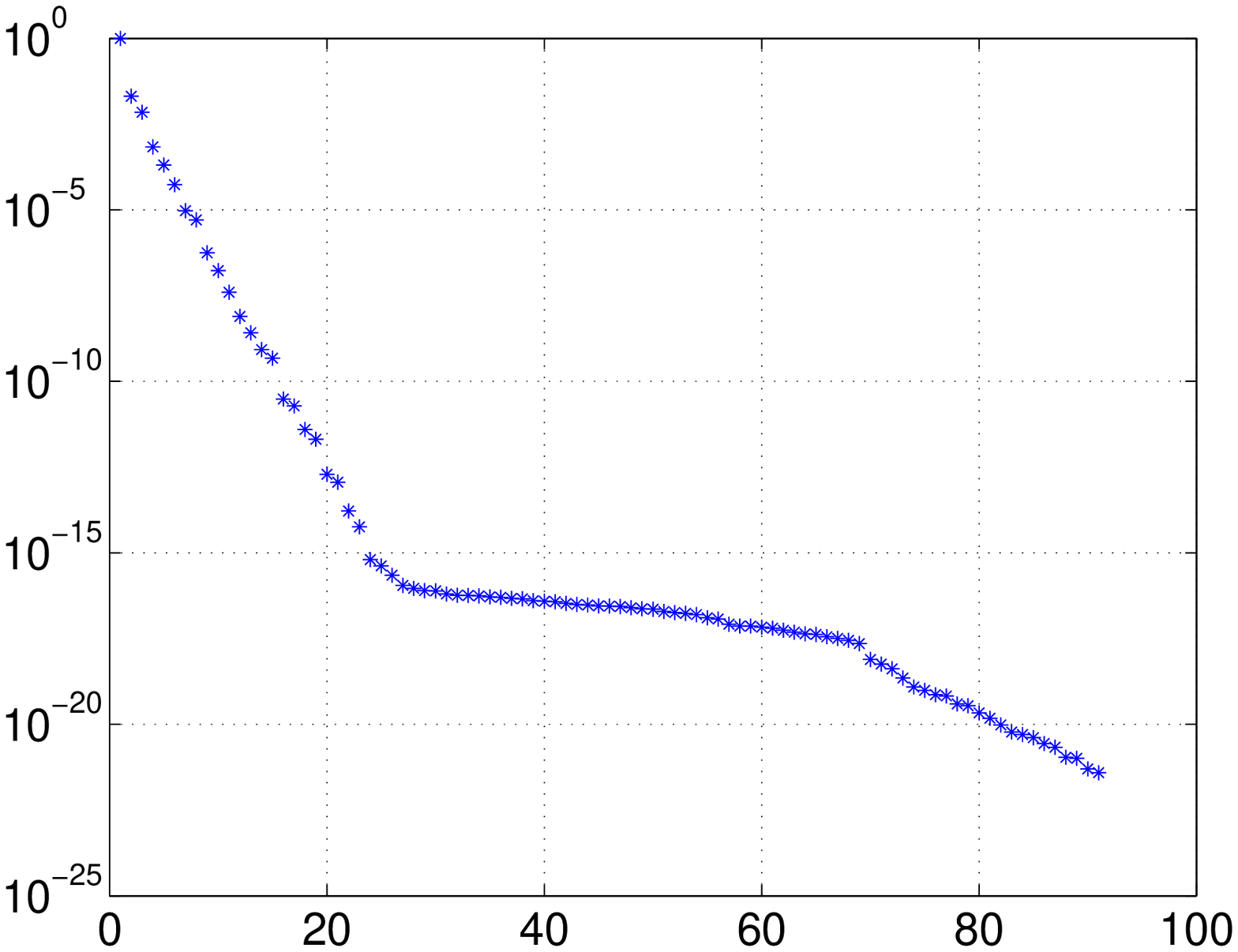}\hfill
\includegraphics[scale=0.215]{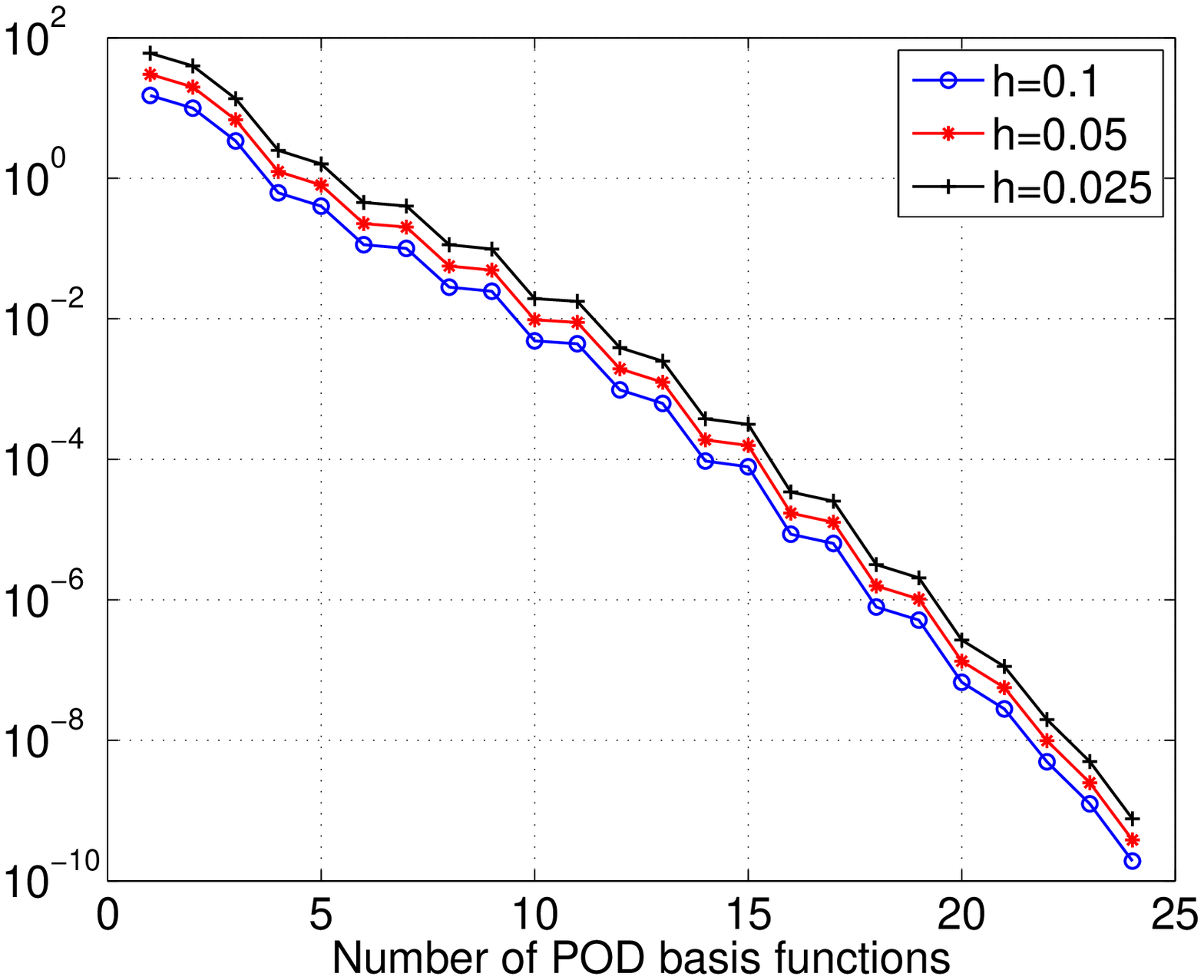}
\caption{Decay of the singular values for the snapshots set associated with Test 1 (left) and Test 2 (right).}
\label{fig1:svd}
\end{figure}

\subsection{Test 2: Semi-linear equation}

The second test concerns the semi-linear equation. In \eqref{state:eq} we set $\te=3$, $\varepsilon=10^{-1}$, $\gamma=0.1$, $\mu=1$, $\omega=(0,1)$ and $w_\circ(x)=2(x-x^2)$. The shape function $b$ is equal to the initial condition $w_\circ$. In \eqref{CostEx6.1} we choose $\lambda=1$, and $\bar w=0$. To compute the POD basis we determine solutions to the state equation for controls in the set $U_{\mbox{snap}}=\{-1,0,1\}$ with a semi-implicit finite difference scheme with time step of $0.05$ and space step of $0.01$. In \eqref{hjb_hkPOD2} we consider $K\in\{0.1,0.05\}, h=0.1K$. The optimal trajectory is obtained with a time step size of $0.05$. The control set is given by 21 controls equally distributed from -1 to 1. The shape function $b$ is equal to the initial condition $w_\circ$.

The uncontrolled solution is shown on the left of Figure~\ref{fig2:hjbsol}. 
\begin{figure}[htbp]
\centering
\includegraphics[scale=0.215]{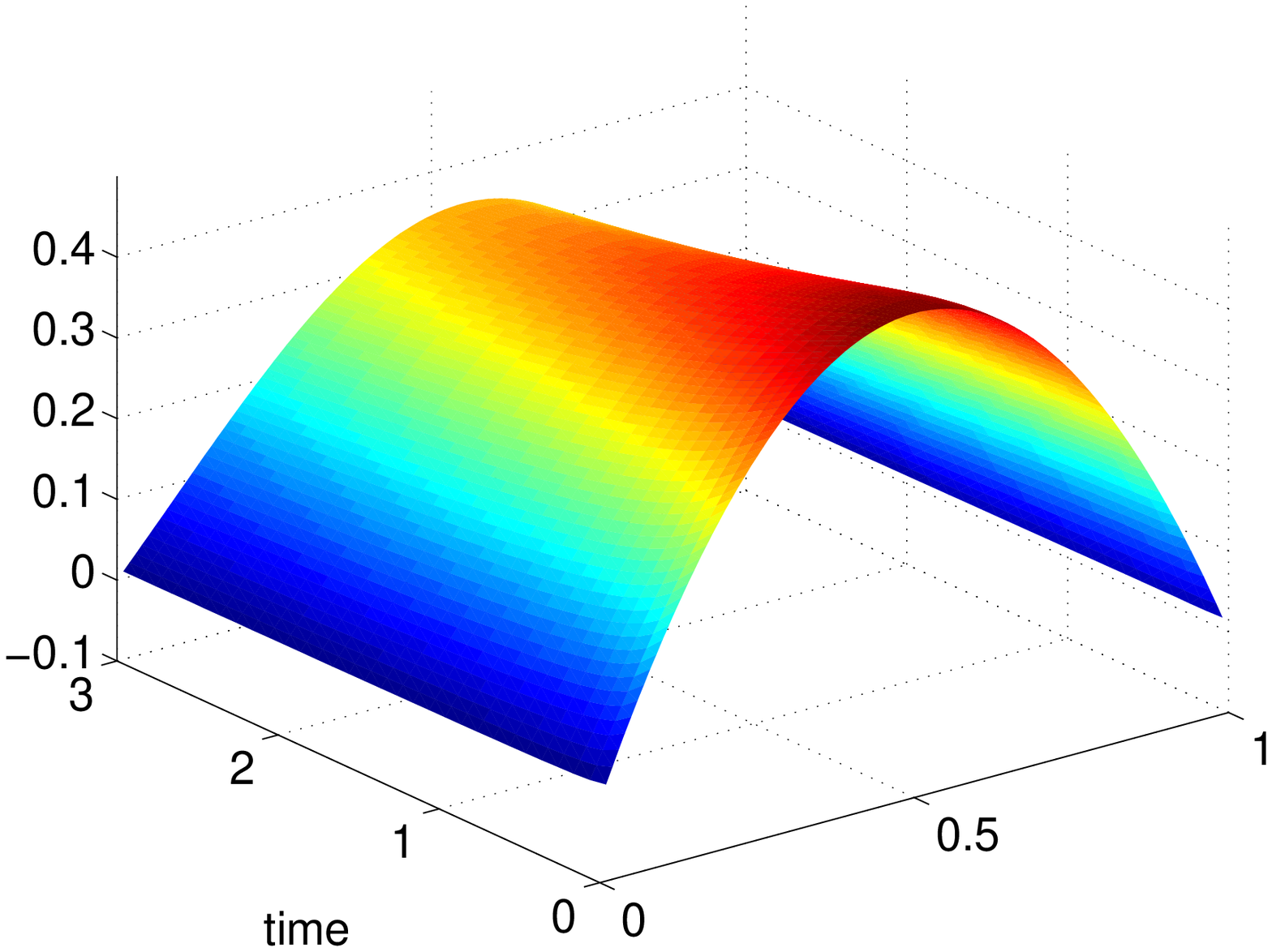}\hfill
\includegraphics[scale=0.215]{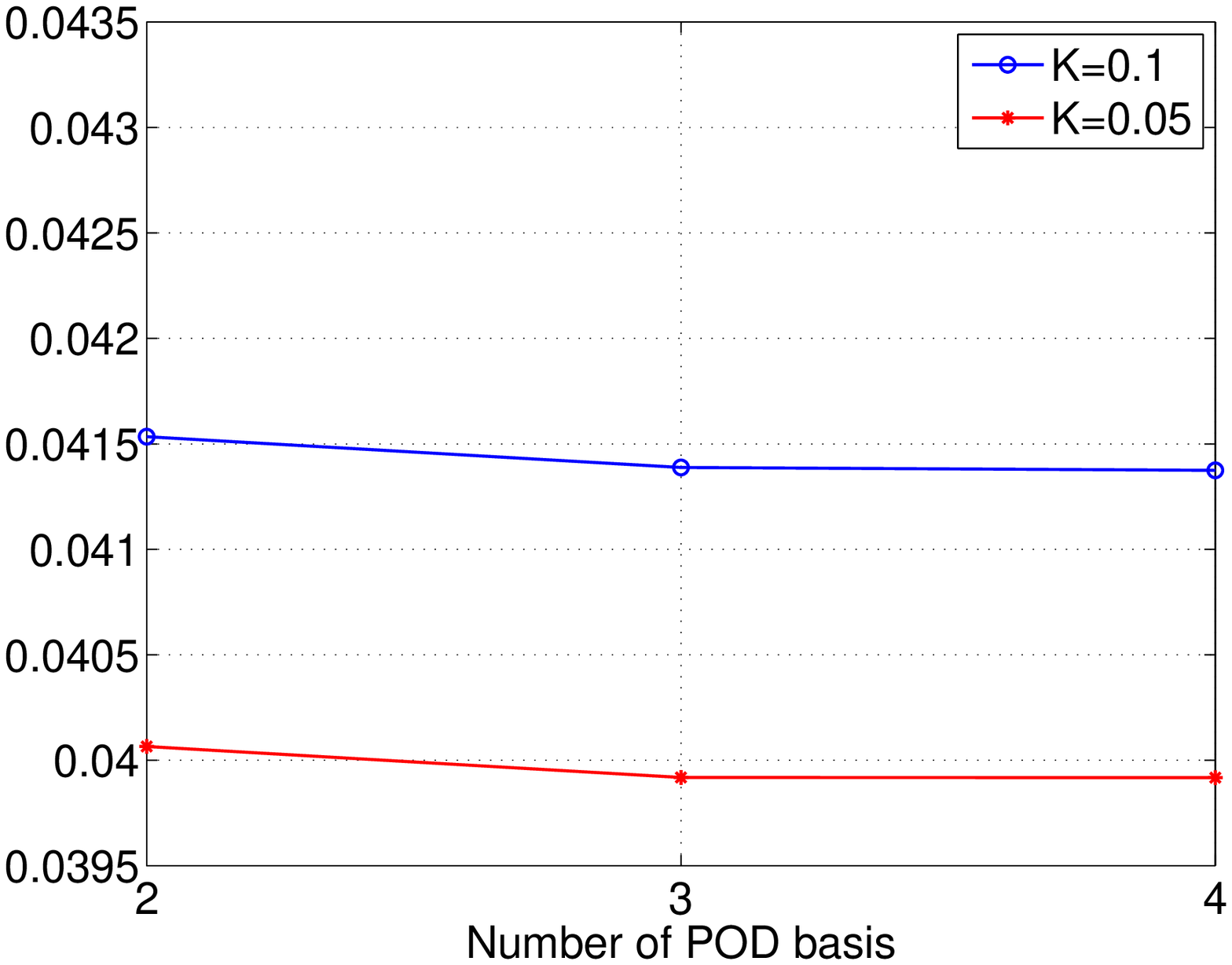}\hfill
\includegraphics[scale=0.215]{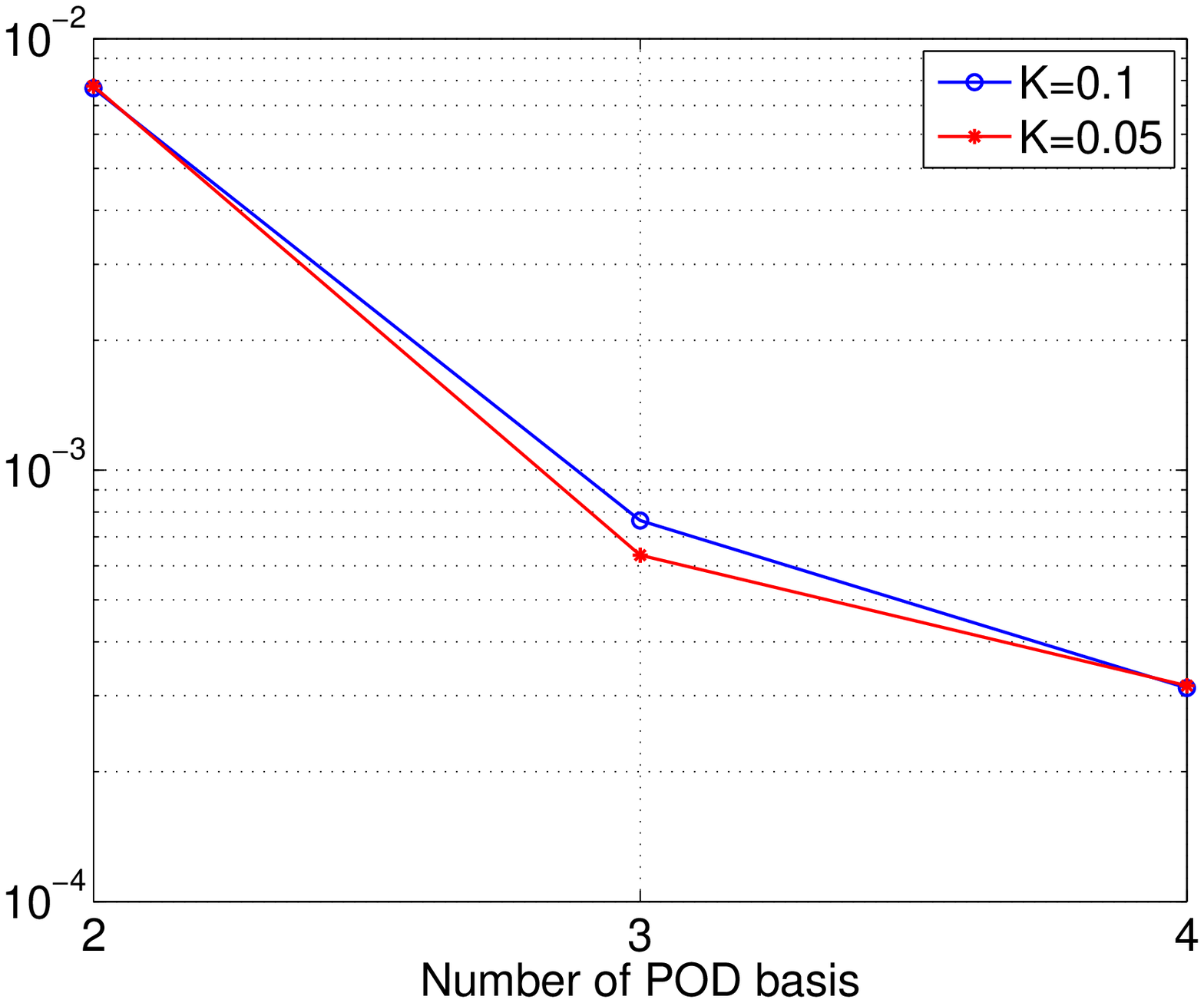}\hfill
\caption{Test 2: Uncontrolled solution (left), evaluation of the cost functional (middle) and distance between $y(u^\ell)$, $y^\ell(u^\ell)$ (right).}
\label{fig2:hjbsol}
\end{figure}
As we can see the semi-linear part does not allow to stabilize to zero the solution. Our goal is to steer the solution to the origin. The optimal solutions and its correspondent optimal controls are shown in Figure~\ref{fig2:hjbcon}. Moreover we plot the differences between the computed solutions (please note the different scaling of the pictures). As we can see the difference decreases when the number of POD basis functions increase.
\begin{figure}[htbp]
\centering
\includegraphics[scale=0.215]{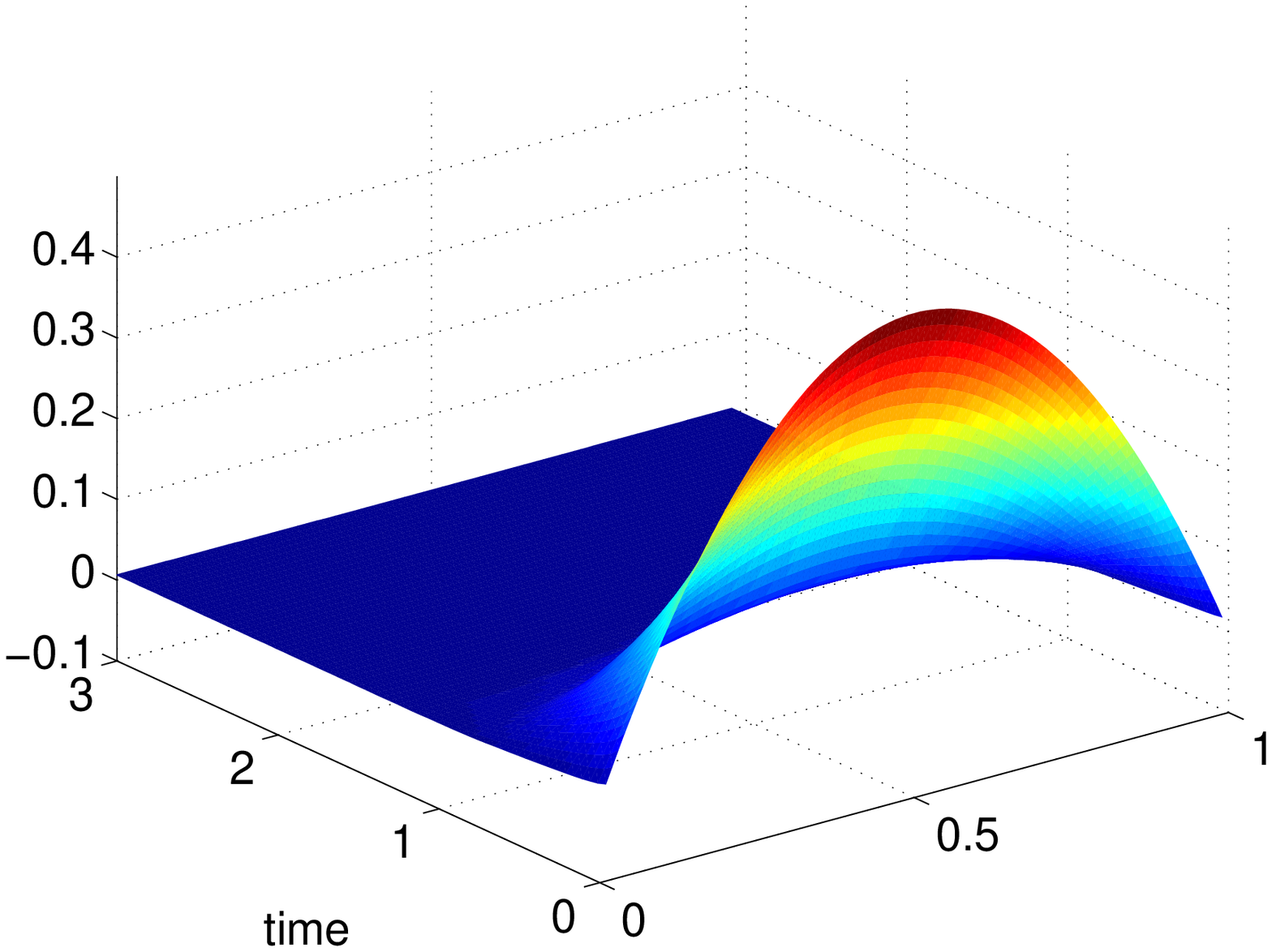}\hfill
\includegraphics[scale=0.215]{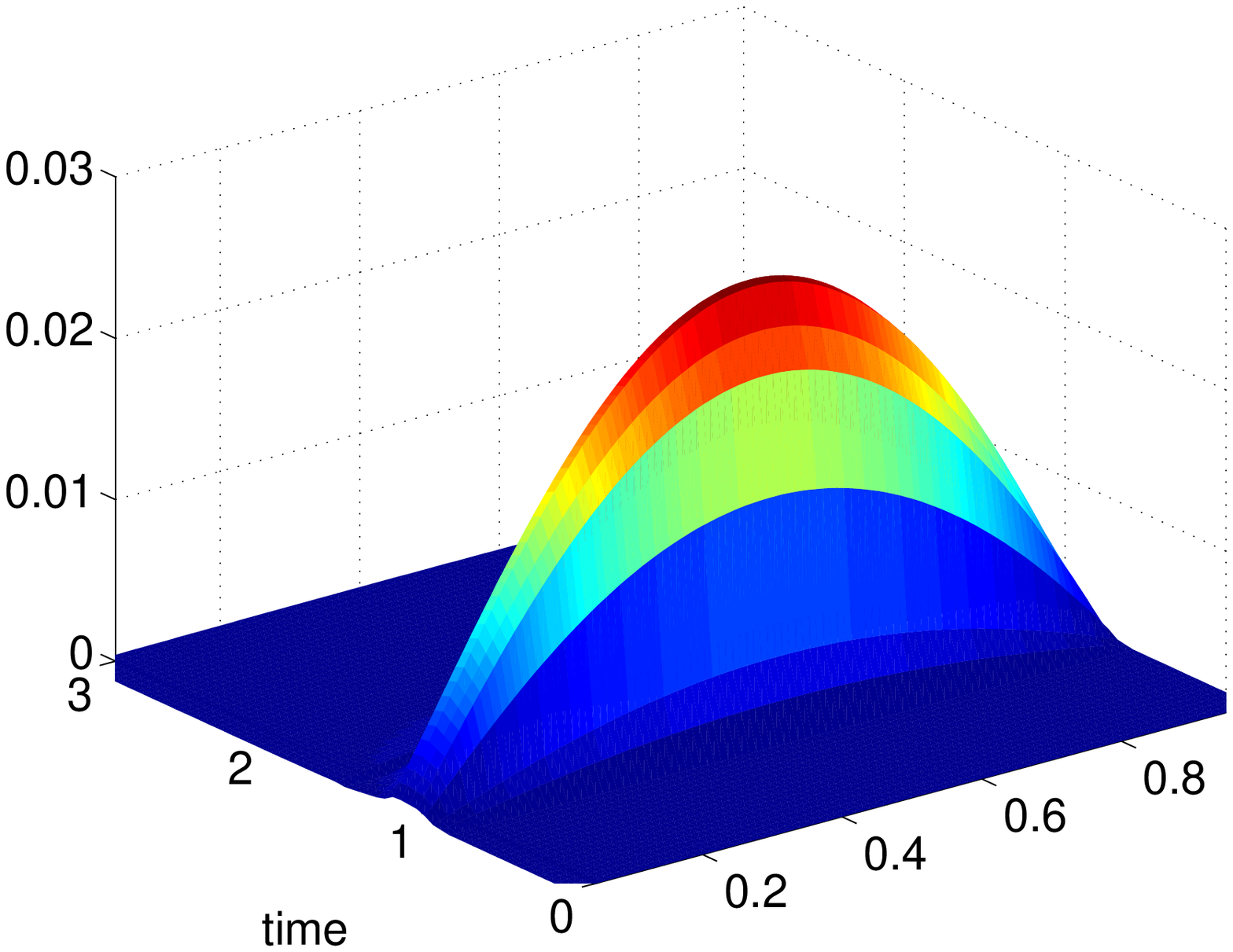}\hfill
\includegraphics[scale=0.215]{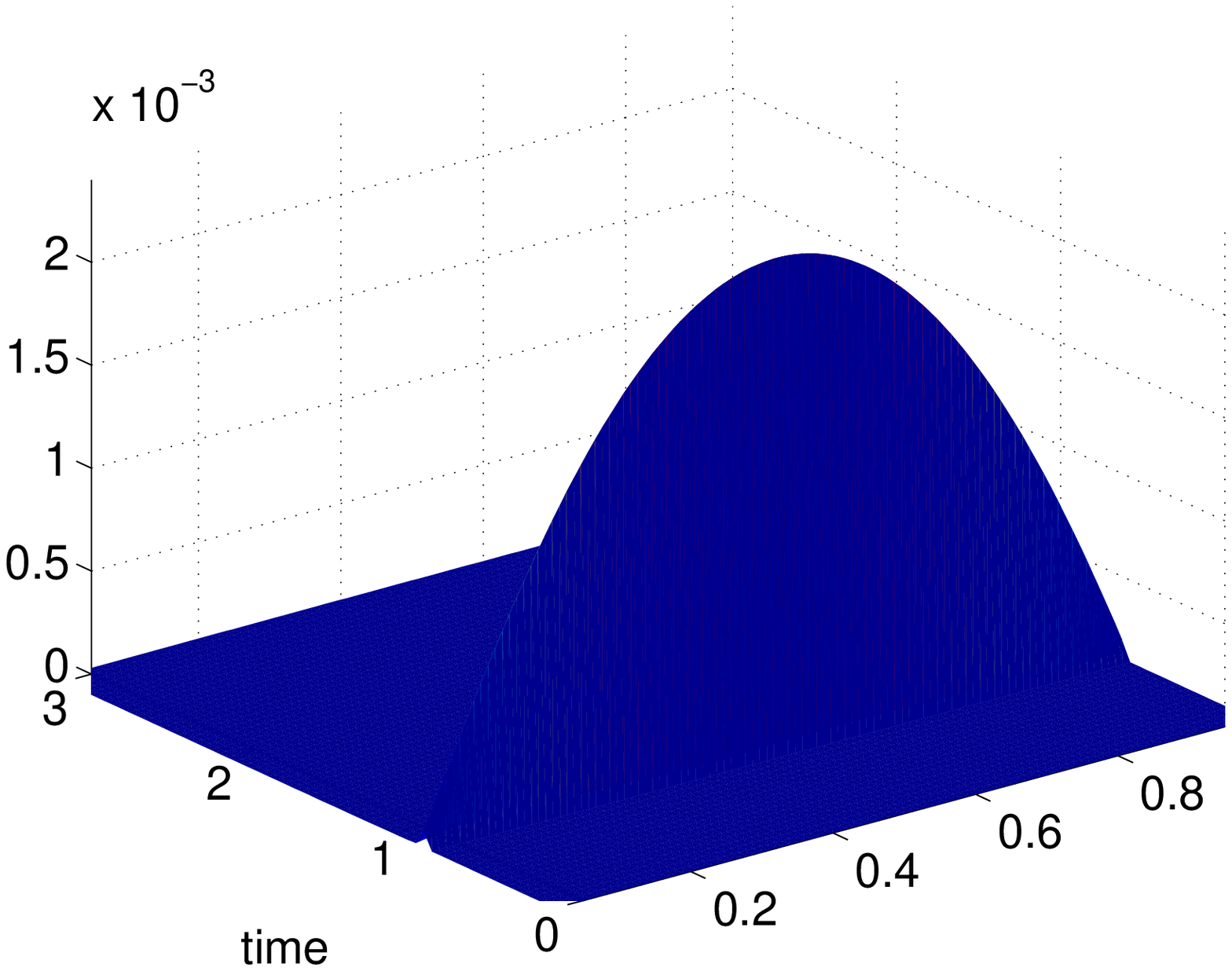}\hfill
\includegraphics[scale=0.215]{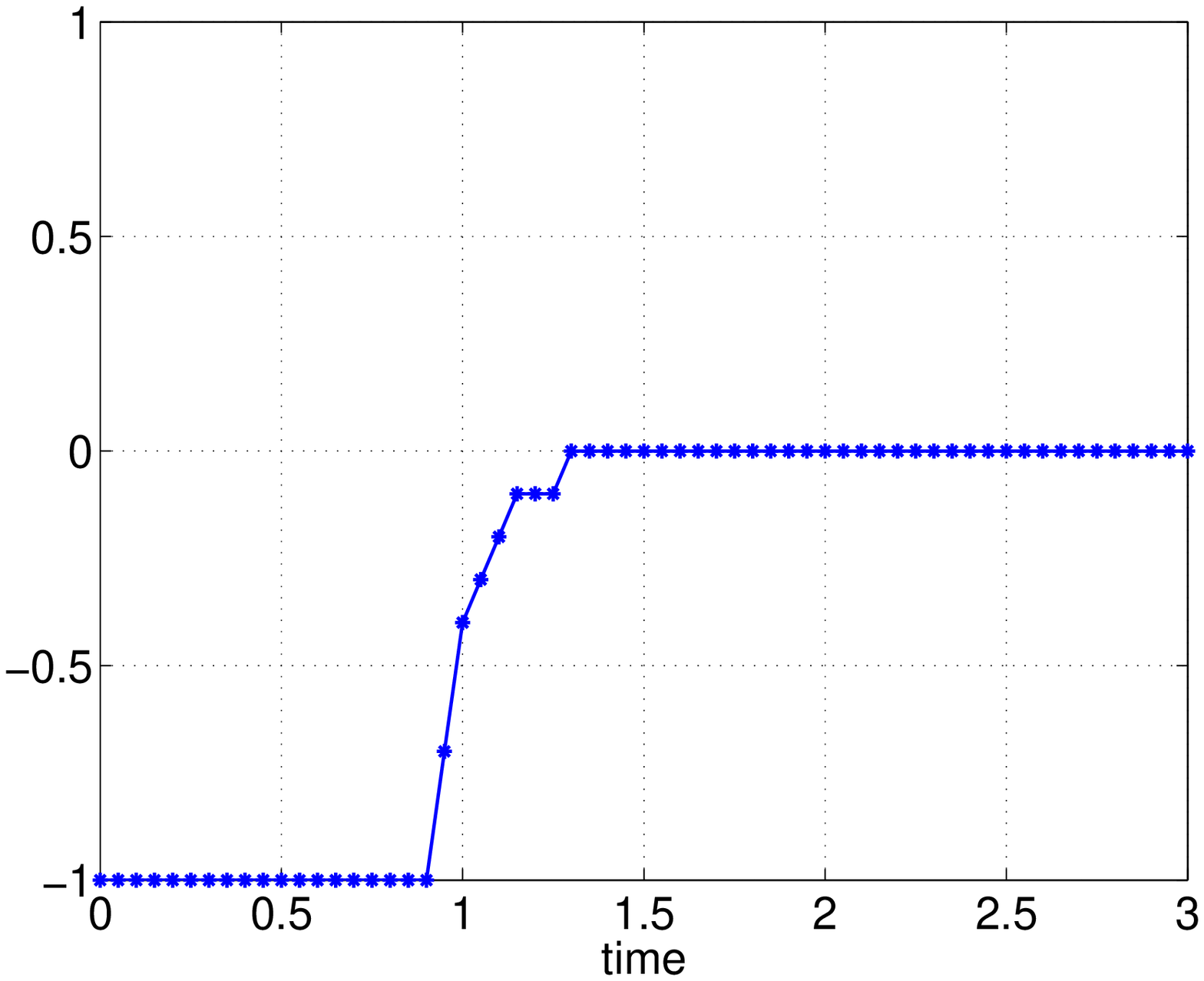}\hfill
\includegraphics[scale=0.215]{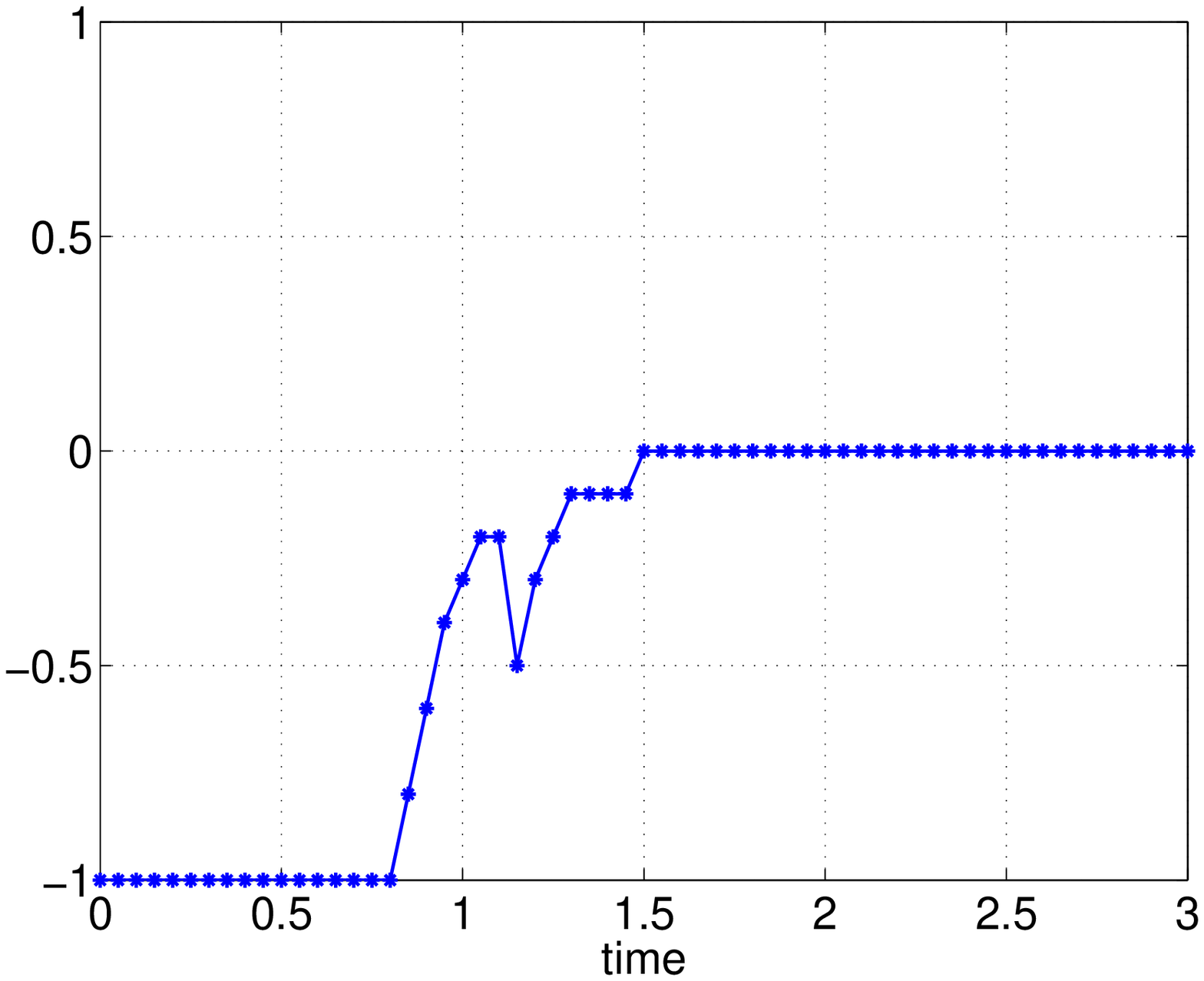}\hfill
\includegraphics[scale=0.215]{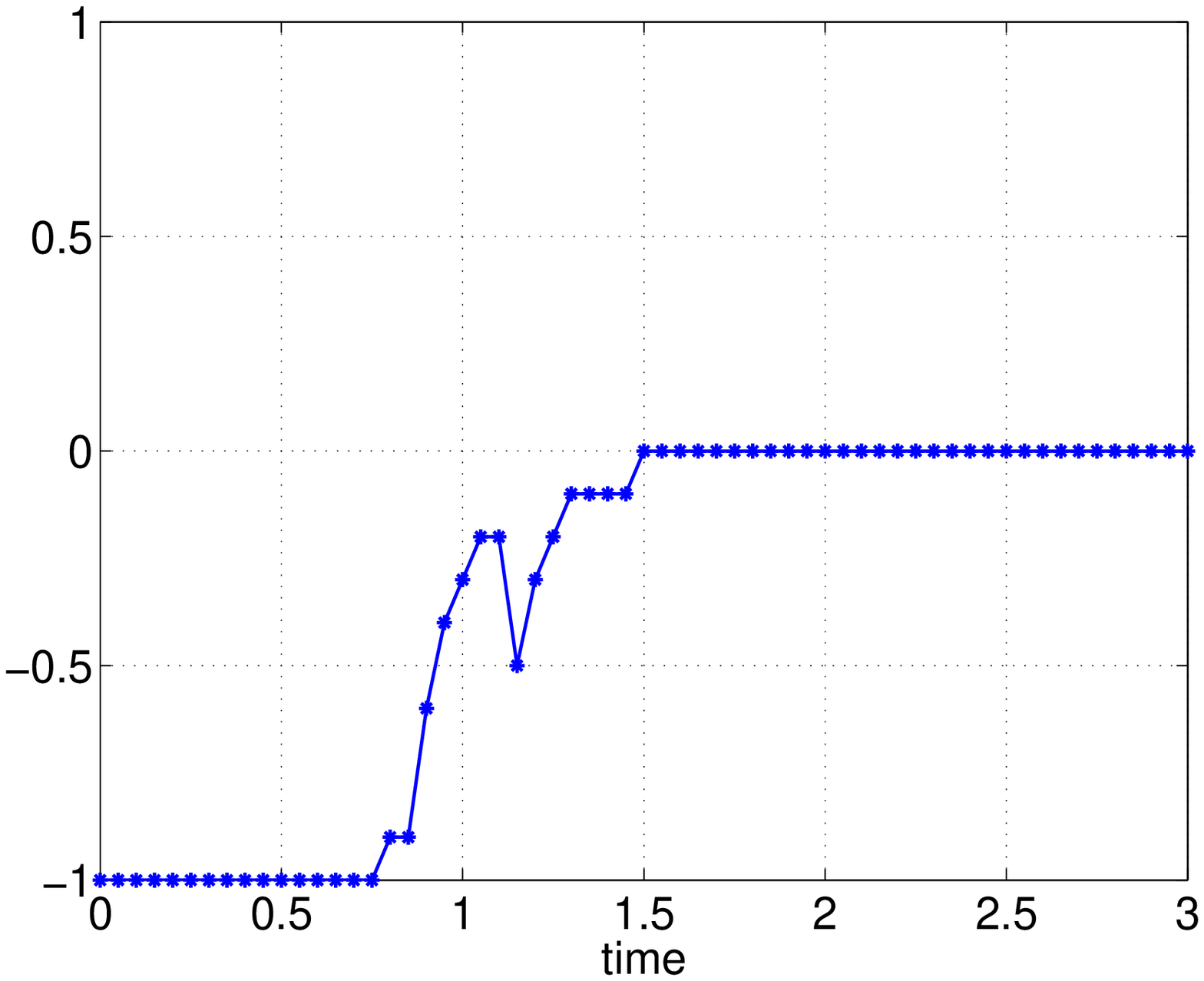}\hfill
\caption{Test 2: Optimal HJB states computed with Algorithm \ref{Alg:hjbpod} with $\ell=4$ POD basis functions (top-left), difference between optimal solution with 4 POD basis and 2 POD basis (top-middle), difference between optimal solution with 4 POD basis and 3 POD basis (top-right),  optimal HJB controls with $\ell=2,3,4$ (bottom). }
%\caption{Test 2: Optimal HJB state computed with Algorithm~\ref{Alg:hjbpod} with $\ell=2,3,4$ POD basis functions (top) and corresponding optimal HJB controls (bottom). {\bf ?? cambia caption} }
\label{fig2:hjbcon}
\end{figure}
The quality of our approximation is confirmed by Figure \ref{fig2:hjbsol} where we can see from the evaluation of the cost functional and consistency of the suboptimal control. In this case the error decays much faster than in the previous example. This depends on the decay of the singular value of the snapshots set as shown in Figure~\ref{fig1:svd}. %Nevertheless, when we deal with only 2 POD basis functions the control is not able to steer the solution to the origin and provide optimal controls who jumps several time from -1 to 1.

\subsection{Test 3: Semi-linear equation with uniform noise}

In this test we deal with the semi-linear equation discussed in the previous example but  we neglect the convection term ($\gamma=0$) and we add noise to the optimal trajectory. The uncontrolled solution is shown on the left of Figure \ref{fig3:nopert}, the optimal trajectory and control computed by means of Algorithm \ref{Alg:hjbpod} are in the middle and the right side of Figure \ref{fig3:nopert}.

\begin{figure}[htbp]
\includegraphics[scale=0.215]{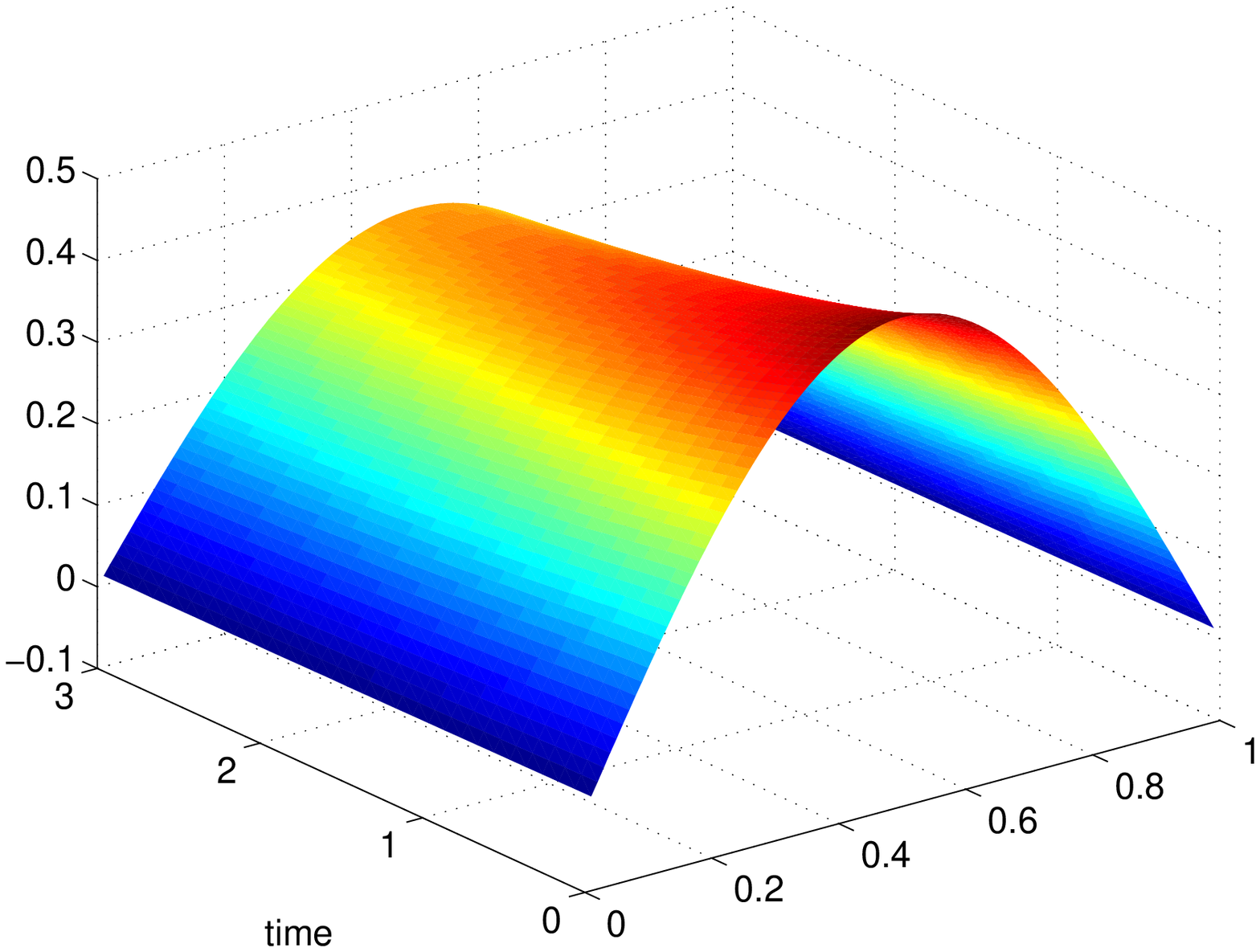}\hfill
\includegraphics[scale=0.215]{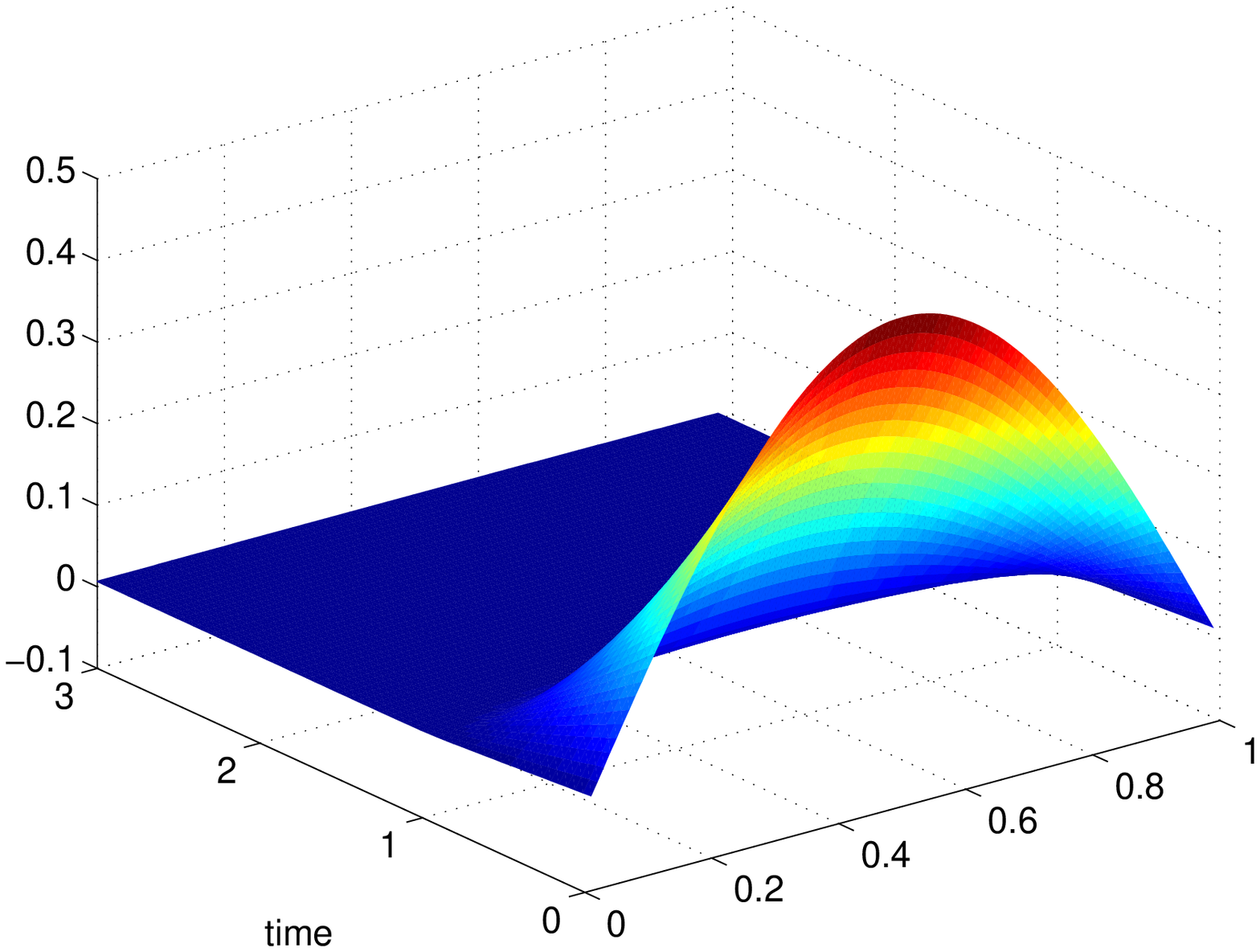}\hfill
\includegraphics[scale=0.215]{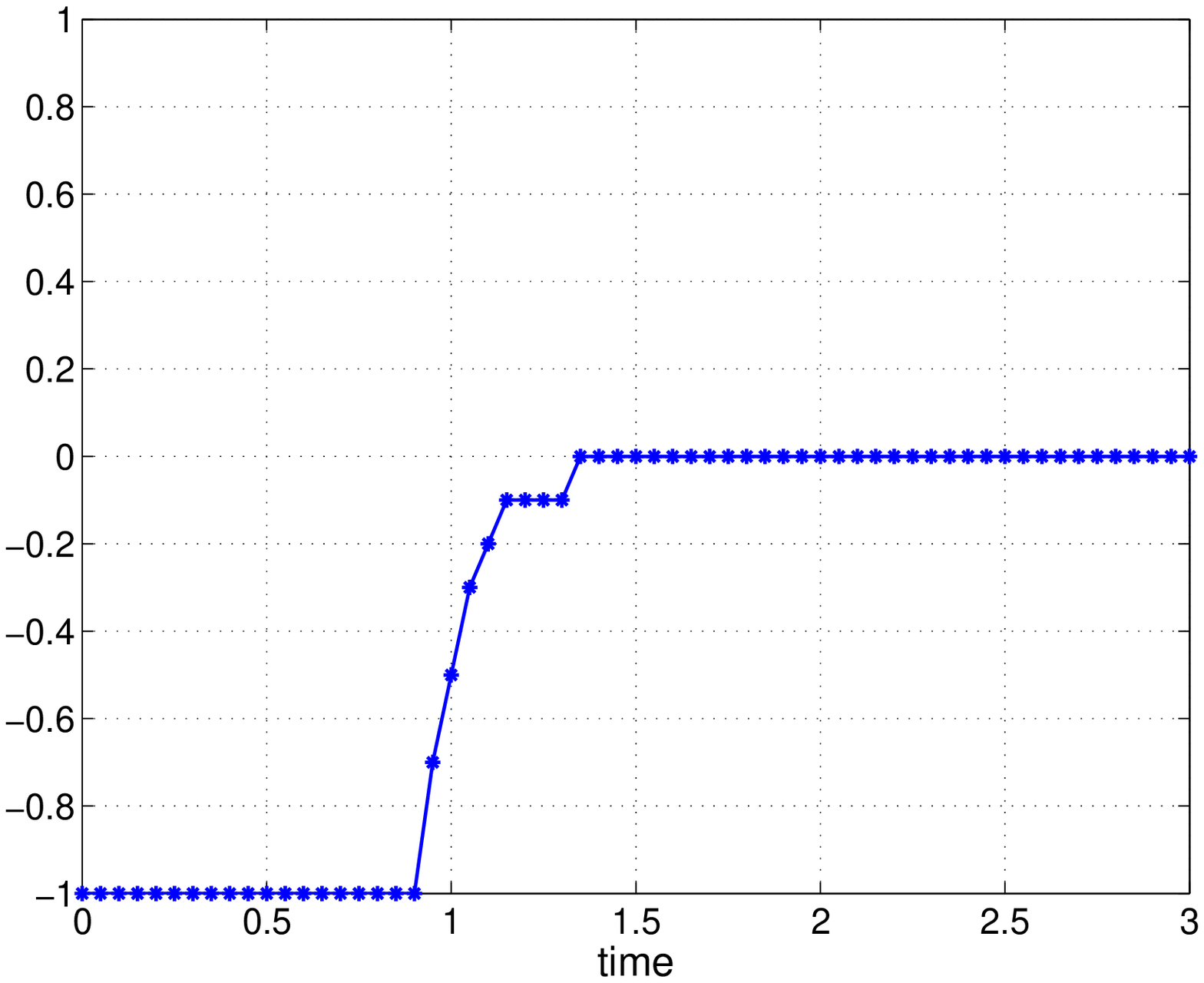}\\
\caption{Test 3: Uncontrolled solution (left),  Optimal state (middle) and corresponding optimal HJB control (right).}
\label{fig3:nopert}
\end{figure}

\noindent
The goal is to show the stabilization of the feedback control under strong perturbations of the system. We note that in this case the value function is stored from the system without perturbation, but the reconstruction of the feedback control is affected by uniform noise $\eta(x)$ between $[-1,1]$ in every time step: $y(x,\cdot)=(1+\eta(x))y(x,\cdot)$. Figure~\ref{fig3:pert} shows optimal solution and control corresponding to different noise levels ($|\eta(x)|\leq50\%$ (top) and $|\eta(x)\leq90|\%$ (bottom)).
\begin{figure}[htbp]
\includegraphics[scale=0.33]{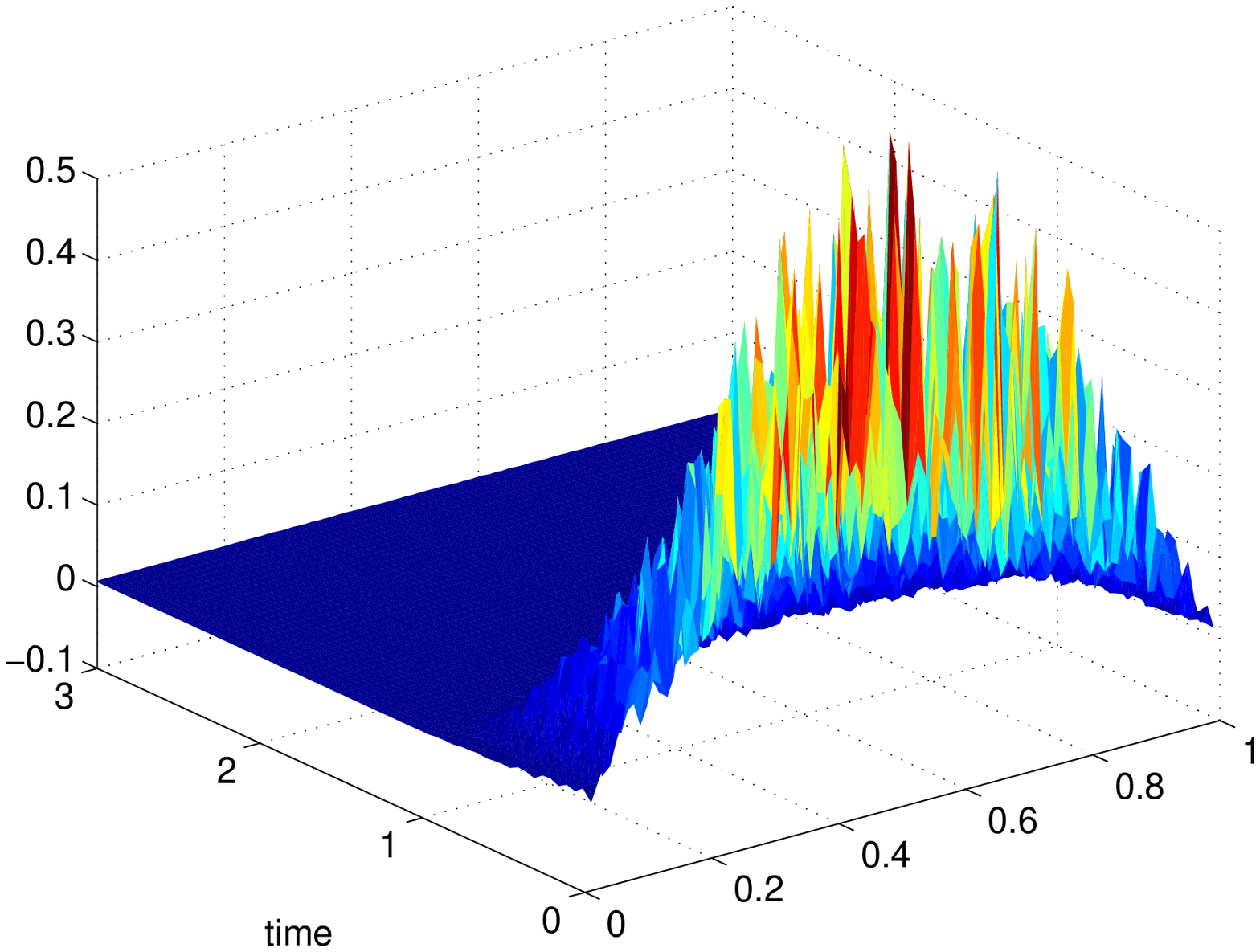}
\includegraphics[scale=0.33]{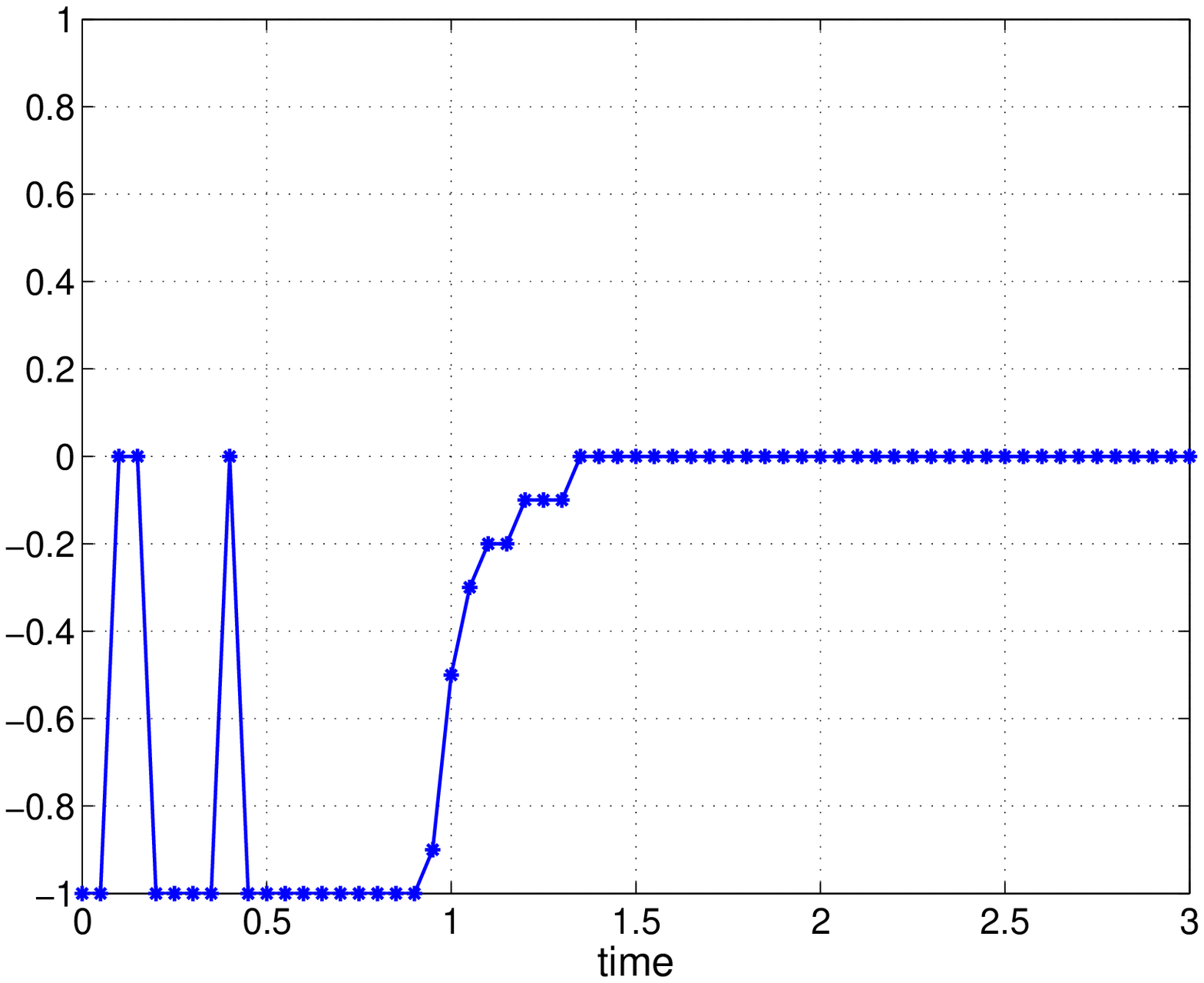}\\
\includegraphics[scale=0.33]{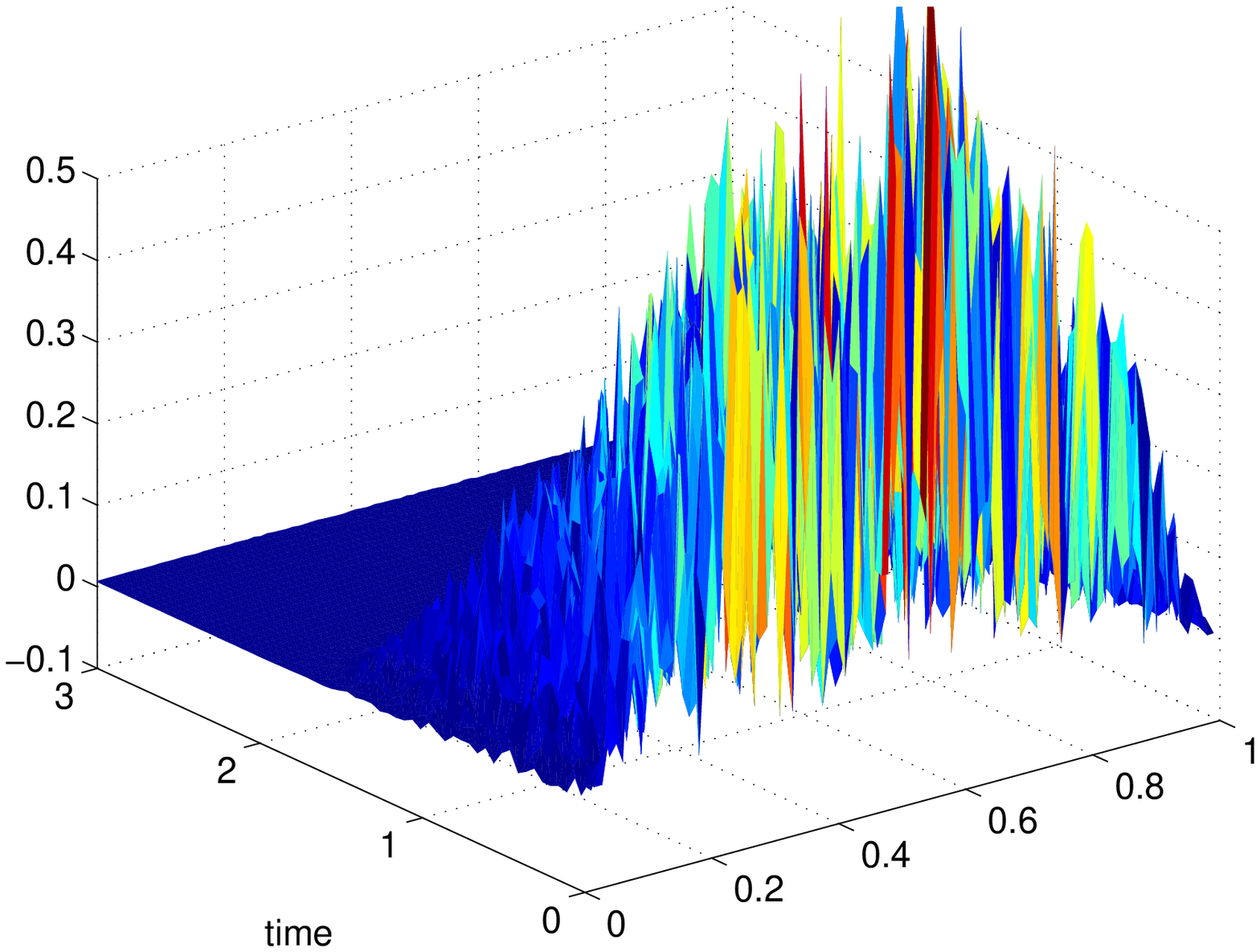}
\includegraphics[scale=0.33]{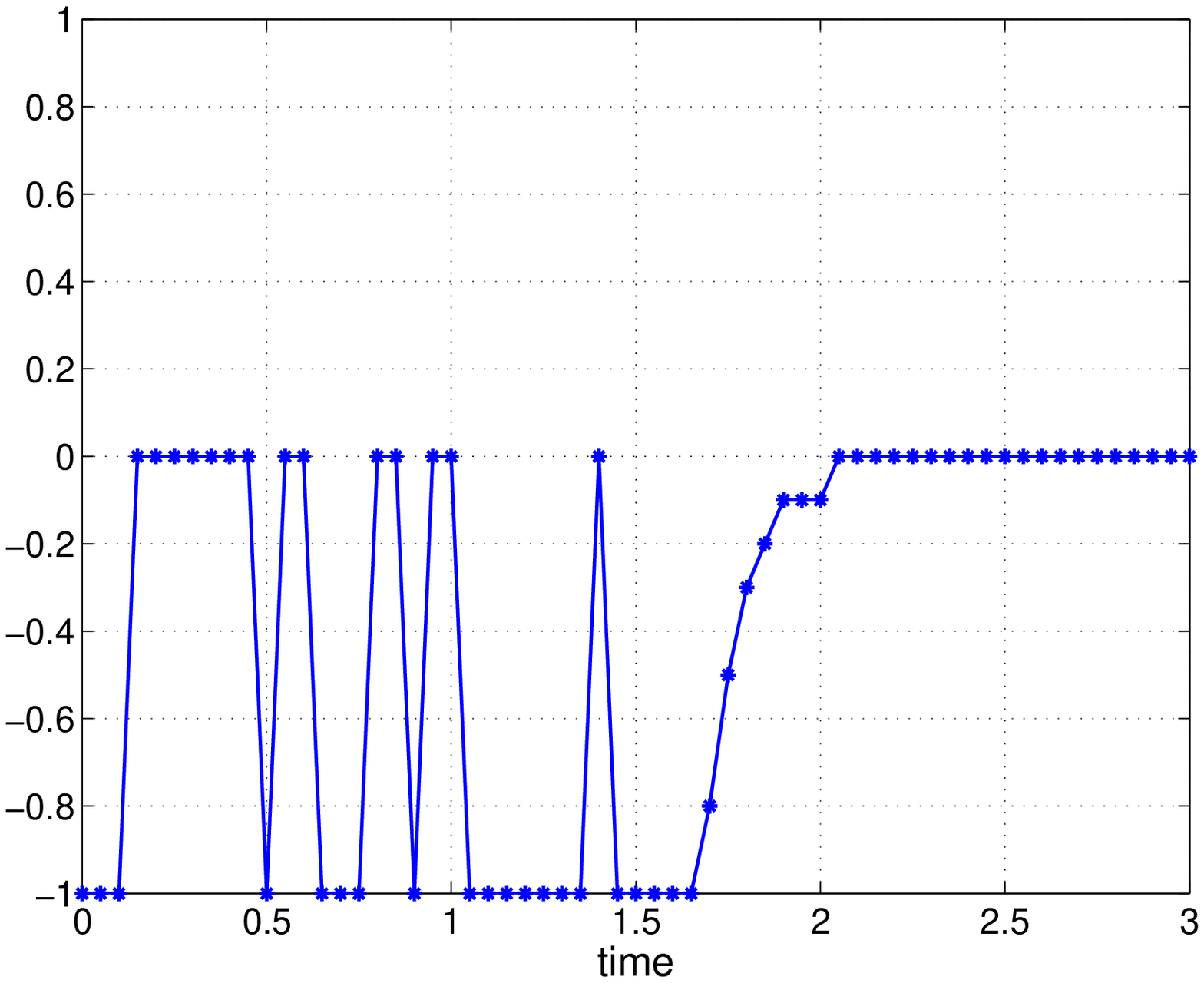}
\caption{Test 3: Optimal HJB-POD state (left) and corresponding optimal control (right) with $|\eta(x)|\leq50\%$ (top) and $|\eta(x)|\leq90\%$ noise (bottom).}
\label{fig3:pert}
\end{figure}
In this example we can see the power of the feedback control, and in particular, the importance of the knowledge of the value function. In both examples, the trajectory is stabilized close to the origin. If we have a look at the optimal control input we can observe a strong chattering. In both cases the optimal control jumps often from -1 to 0. In particular, in the second case, it is possible to observe a stronger chattering due to the high disturbances. Nevertheless, the feedback control is able to stabilize the perturbed system. 

%%%%%%%%%%%%%%%%%%%%%%%%%%%%%%%%%%%%%%%%%%%%%%%%%%%%%
%\newpage
\section{Conclusion and Remarks}
\label{Section7}
%%%%%%%%%%%%%%%%%%%%%%%%%%%%%%%%%%%%%%%%%%%%%%%%%%%%%

In this paper we present a new a-priori error analysis for the coupling between the HJB equation and the POD method. The proposed estimate is presented for the infinite horizon control problems with linear and nonlinear dynamical systems but this approach could be also applied to other optimal control problems provided one has a priori estimates on the approximation based on the HJB equation. The convergence of the method is guaranteed under rather general assumptions on the optimal control problem and some technical assumptions on the dynamics and on the POD approximation. For the latter, it is clear that a clever choice of the snapshots set can play a  crucial role in the estimate in order to reduce the contribution of the POD approximation in the a-priori estimate. Several choices are possible based on greedy techniques or on a previous open-loop approximation, these choices will be investigated in a future paper.  At present, the numerical tests illustrated in the last section confirms our theoretical findings and show the robustness of the Bellman's approach also under strong disturbances of the dynamical system.

%%%%%%%%%%%%%%%%%%%%%%%%%%%%%%%%%%%%%%%%%%%%%%%%%%%%%
% References
%%%%%%%%%%%%%%%%%%%%%%%%%%%%%%%%%%%%%%%%%%%%%%%%%%%%%


\begin{thebibliography}{99}

\bibitem{AF12}
A. Alla and M. Falcone. {\em An adaptive POD approximation method for the control of advection-diffusion equations}, in K. Kunisch, K. Bredies, C. Clason, G. von Winckel, (eds) {\em Control and Optimization with PDE Constraints,} International Series of Numerical Mathematics, \textbf{164}, Birkh\"auser, Basel, 2013, 1-17.

\bibitem{AF13b}
A. Alla and M. Falcone. {\em A time-adaptive POD method for optimal control problems}, in the Proceedings of the 1st IFAC Workshop on Control of Systems Modeled by Partial Differential Equations, \textbf{1}, 2013, 245-250.

\bibitem{AFK13}
A. Alla, M. Falcone, and D. Kalise. {\em An accelerated value/policy iteration scheme for the solution of DP equations}, in A. Abdulle, S. Deparis, D. Kressner, F. Nobile, M. Picasso (eds.), Numerical Mathematics and Advanced Applications, Proceedings of ENUMATH 2013, Lecture Notes in Computational Science and Engineering, \textbf{103}, 2015, 489-497

\bibitem{AFK15}
A. Alla, M. Falcone, and D. Kalise. {\em An efficient policy iteration algorithm for dynamic programming equations},  SIAM J. Sci. Comput., \textbf{37}, 2015, 181-200. 

\bibitem{AH14}
A. Alla and M. Hinze. {\em HJB-POD feedback control for Navier-Stokes equations}, Conference Proceeding ECMI, 2014.

\bibitem{AH15}
A. Alla and M. Hinze. {\em HJB-POD feeback control of advection-diffusion equation with a model predictive control snapshot sampling. }, submitted, 2015.\hfill\\
\texttt{http://preprint.math.uni-hamburg.de/public/papers/hbam/hbam2015-15.pdf}




\bibitem{AK01}
J.A. Atwell and B.B. King. {\em Proper orthogonal decomposition for reduced basis feedback controllers for parabolic equations}, Matematical and Computer Modelling, \textbf{33}, 2001, 1-19. 
	
\bibitem{BC97}
M. Bardi and I. Capuzzo-Dolcetta. {\em Optimal Control and Viscosity Solutions of Hamilton-Jacobi-Bellman Equations}. Birkh\"auser, Basel, 1997.

\bibitem{BMNP04}
M. Barrault, Y. Maday, N.C. Nguyen, and A.T. Patera. {\em An 'empirical interpolation' method: application to efficient reduced-basis discretization of partial differential equations}, Comptes Rendus Mathematique, \textbf{339}, 2004, 667-672.

\bibitem{CCFP12}
S. Cacace, E. Cristiani. M. Falcone, and A. Picarelli. {\em A patchy dynamic programming scheme for a class of Hamilton-Jacobi-Bellman equations}, SIAM Journal on Scientific Computing, \textbf{34}, 2012, A2625-A2649.

\bibitem{CFLS94} 
F. Camilli, M. Falcone, P. Lanucara, and A. Seghini. {\em A domain decomposition method for Bellman equations}, 
in D. E. Keyes and J. Xu (eds.), Domain Decomposition methods in Scientific and Engineering Computing, 
Contemporary Mathematics n.180, AMS, 1994,  477-483.

\bibitem{CGS12}
D. Chapelle, A. Gariah, and J. Saint-Marie. {\em Galerkin approximation with proper orthogonal decomposition: new error estimates and illustrative examples}, ESAIM: Math. Model. Numer. Anal., \textbf{46}, 2012, 731-757.

\bibitem{CS10}
S. Chaturantabut and D.C. Sorensen. {\em Nonlinear model reduction via discrete empirical interpolation}, SIAM Journal on Scientific Computing, \textbf{32}, 2010, 2737-2764.

\bibitem{CZ95}
R.F. Curtain and H.J. Zwart. {\em An Introduction to Infinite-Dimensional Linear Systems Theory}, Springer, 1995.

\bibitem{Fal87}
M. Falcone. {\em A numerical approach to the infinite horizon problem of deterministic control theory}, Applied Mathematics and Optimization, \textbf{15}, 1987, 1-13.

\bibitem{Fal97}
M. Falcone. {\em Numerical solution of dynamic programming equations}, Appendix A of \cite{BC97}, Birkh\"auser, Basel, 1997.

\bibitem{FFbook}
M. Falcone and R. Ferretti. {\em Semi-Lagrangian Approximation Schemes for Linear and Hamilton-Jacobi equations}, 
SIAM, 2014.  

\bibitem{FLS94}
M. Falcone, P. Lanucara, and A. Seghini. {\em A splitting algorithm for Hamilton-Jacobi-Bellman equations}
Applied Numerical Mathematics, \textbf{15}, 1994, 207-218.

\bibitem{FS93}
W.H. Fleming and H.M. Soner. {\em Controlled Markov processes and viscosity solutions}, Springer-Verlag, New York, 1993.

\bibitem{GLT76}
R. Glowinski, J.L. Lions, and R. Tr\'emoli\`eres. {\em Analyse Numerique des In\'equations Variationelles}, Dunod-Bordas, Paris, 1976.

\bibitem{GV13}
M. Gubisch and S. Volkwein. {\em Proper orthogonal decomposition for linear-quadratic optimal control}, submitted, 2013.\hfill\\
\texttt{http://kops.ub.uni-konstanz.de/handle/urn:nbn:de:bsz:352-250378}

\bibitem{HLBR12}
P. Holmes, J.L. Lumley, G. Berkooz, and C.W. Rowley. {\em Turbulence, Coherent Structures, Dynamical Systems and Symmetry}, Cambridge Monographs on Mechanics, Cambridge University Press, second edition, 2012.

\bibitem{HPUU09} 
M. Hinze, R. Pinnau, M. Ulbrich, and S. Ulbrich. {\em Optimization with PDE Constraints. Mathematical Modelling: Theory and Applications}, \textbf{23}, Springer Verlag, 2009.

\bibitem{HV05}
M. Hinze and S. Volkwein. {\em Proper orthogonal decomposition surrogate models for nonlinear dynamical systems: error estimates and suboptimal control}, in Reduction of Large-Scale Systems, P. Benner, V. Mehrmann, D. C. Sorensen (eds.), Lecture Notes in Computational Science and Engineering, \textbf{45}, 2005, 261-306.

\bibitem{KK14}
D. Kalise and A. Kr\"oner. {\em Reduced-order minimum time control of advection-reaction-diffusion systems via dynamic programming}, In Proceedings of the 21st International Symposium on Mathematical Theory of Networks and Systems, 2014, 1196-1202. 

\bibitem{KV02}
K. Kunisch and S. Volkwein. {\em Galerkin proper orthogonal decomposition methods for a general equation in fluid dynamics}, SIAM Journal on Numerical Analysis, \textbf{40}, 2002, 492-515.

\bibitem{KVX04}
K. Kunisch, S. Volkwein, and L. Xie. {\em HJB-POD based feedback design for the optimal control of evolution problems}, SIAM J. on Applied Dynamical Systems, \textbf{4}, 2004, 701-722.

\bibitem{KX05}
K. Kunisch and L. Xie. {\em POD-based feedback control of Burgers equation by solving the evolutionary HJB equation,} Computers and Mathematics with Applications, \textbf{49}, 2005, 1113-1126.

\bibitem{LV06}
F. Leibfritz and S. Volkwein. {\em Reduced order output feedback control design for PDE systems using proper orthogonal decomposition and nonlinear semidefinite programming}, Linear Algebra and Its Applications, \textbf{415}, 2006, 542-575.

\bibitem{Lio71}
J.L. Lions. {\em Optimal Control of Systems Governed by Partial Differential Equations}, Springer-Verlag, New York 1971.

\bibitem{NK07}
C. Navasca and A. J. Krener. {\em Patchy solutions of Hamilton-Jacobi-Bellman partial differential equations}, in A. Chiuso et al. (eds.), Modeling, Estimation and Control, Lecture Notes in Control and Information Sciences, \textbf{364} 2007, 251-270.

\bibitem{QV99}
A. Quarteroni and A. Valli. {\em Domain Decomposition Methods for Partial Differential Equations}, Oxford University Press, 1999.

\bibitem{SV10}
E.W. Sachs and S. Volkwein. {\em POD Galerkin approximations in PDE-constrained optimization}, GAMM Mitteilungen, \textbf{33}, 2010, 194-208.

\bibitem{S99}                                                                                                                
J. A. Sethian. {\em Level set methods and fast marching methods}, Cambridge University Press, 1999.

\bibitem{SV03}
{J. A. Sethian and A. Vladimirsky}. {\em Ordered upwind methods for static Hamilton-Jacobi equations: theory and algorithms}, SIAM J. Numer. Anal., \textbf{41}, 2003, 325-363.

\bibitem{Sin13}
J.R. Singler. {\em New POD expressions, error bounds, and asymptotic results for reduced order models of parabolic PDEs}, SIAM Journal on Numerical Analysis, \textbf{52}, 2014, 852-876.

\bibitem{TCOZ04}
Y. Tsai, L. Cheng, S. Osher, and H. Zhao. {\em Fast sweeping algorithms for a class of Hamilton-Jacobi equations}, 
SIAM J. Numer. Anal., \textbf{41}, 2004, 673-694.

\bibitem{Tro10}
F.~Tr\"oltzsch. {\em Optimal Control of Partial Differential Equations: Theory, Methods and Application}, American Mathematical Society, 2010.

\bibitem{Vol11}
S. Volkwein. {\em Model Reduction using Proper Orthogonal Decomposition}, Lecture Notes, University of Konstanz, 201.\hfill\\
\texttt{http://www.math.uni-konstanz.de/numerik/personen/volkwein/teaching/scripts.php}

\end{thebibliography}
\end{document}